\documentclass[12pt]{article}

\usepackage{amsfonts,amsthm,amsmath,latexsym,amssymb}
\usepackage{bbm}
\usepackage{esvect}

\usepackage{tikz} 
\usepackage{pgfplots}
\usepackage[margin=0.75in]{geometry}
\usepackage[margin=1cm]{caption}
\usepackage{float}

\usepackage{algorithm,algorithmic,changepage}

\usepackage{color}
\usepackage{hyperref}

\usepackage{setspace}
\usepackage{authblk}

\newcommand{\be}{\begin{equation}}
\newcommand{\ee}{\end{equation}}
\newcommand{\bes}{\begin{equation*}}
\newcommand{\ees}{\end{equation*}}

\renewcommand{\l}{\lambda}

\newcommand{\G}{\Gamma}
\newcommand{\C}{\mathbb C}

\newcommand{\R}{\mathbb R}

\newcommand{\e}{\epsilon}

\newcommand{\dist}{\mathrm{dist}}
\newcommand{\diam}{\operatorname{diam}}

\renewcommand{\Im}{\operatorname{Im}}
\renewcommand{\Re}{\operatorname{Re}}

\newcommand{\cM}{\mathcal{M}}
\newcommand{\cO}{\mathcal{O}}

\newcommand{\cE}{\mathcal{E}}

\newcommand{\cN}{\mathcal{N}}

\newcommand{\cB}{\mathcal{B}}

\newcommand{\cP}{\mathcal{P}}

\newcommand{\bP}{\mathbb{P}}

\newcommand{\bE}{\mathbb{E}}

\newcommand{\ep}{\epsilon }
\newcommand{\la}{\lambda }

\newcommand{\de}{\delta }

\newcommand{\si}{\sigma }

\newcommand{\ga}{\gamma }
\newcommand{\Ga}{\Gamma }

\newcommand{\Om}{\Omega }

\newcommand{\ones}{\mathbbm{1}}
\newcommand{\one}{\mathbf{1}}

\newcommand{\conv}{\operatorname{conv}}

\newcommand{\rea}{\operatorname{Re}}
\newcommand{\Length}{\operatorname{Length}}

\newcommand{\ima}{\operatorname{Im}}

\newcommand{\area}{\operatorname{Area}}

\newcommand{\bd}[1]{\partial #1}

\newcommand{\Mod}{\operatorname{Mod}}

\newcommand{\bi}{\begin{itemize}}
\newcommand{\ei}{\end{itemize}}

\newtheorem{thm}{Theorem}
\newtheorem{prop}[thm]{Proposition}
\newtheorem{cor}[thm]{Corollary}
\newtheorem{claim}[thm]{Claim}
\newtheorem{lemma}[thm]{Lemma}

\theoremstyle{definition}
\newtheorem{remark}[thm]{Remark}
\newtheorem{defn}[thm]{Definition}
\newtheorem{example}[thm]{Example}

\makeatletter

\def\1{{\bf 1}}

\begin{document}

\title{\bf Convergence of the Probabilistic Interpretation of Modulus}

\author[1]{Nathan Albin}
\author[2]{Joan Lind}
\author[1]{Pietro Poggi-Corradini}
\affil[1]{Department of Mathematics, Kansas State University, Manhattan, KS 66506, USA.}
\affil[2]{Department of Mathematics, University of Tennessee, Knoxville, TN 37996, USA.}

\maketitle

\abstract{
Given a Jordan domain $\Om\subset\C$ and two disjoint arcs $A, B$ on $\bd\Om$, the modulus $m$ of the curve family connecting $A$ and $B$ in $\Om$ is equal to the modulus of the curve family connecting the vertical sides in the rectangle $R=[0,1]\times[0,m]$. Also, $m>0$ is the unique value such that  there is a conformal map $\psi$ mapping $\Om$ to  ${\rm int}(R)$ so that $\psi$ extends continuously to a homeomorphism of $\bd \Om$ onto $\bd R$ and the arcs $A$ and $B$ are sent to the vertical sides of $R$. Moreover, in the case of the rectangle  the family of horizontal segments connecting the two sides has the same modulus as the entire connecting family. Pulling these segments back to $\Om$ via $\psi$ yields a family of extremal curves (also known as horizontal trajectories)  connecting $A$ to $B$ in $\Om$. In this paper, we show that these extremal curves can be approximated by some discrete paths arising from an orthodiagonal approximation of $\Om$. Moreover, we show that there is a natural probability mass function (pmf) on these paths, deriving from the theory of discrete modulus, which converges to the transverse measure on the set of extremal curves. The key ingredient is an algorithm that, for an embedded planar graph, takes the current flow between two sets of nodes, and produces a unique path decomposition with non-crossing paths. Moreover, some care was taken to adapt recent results for harmonic convergence on orthodiagonal maps, due to Gurel-Gurevich, Jerison, and Nachmias, to our context. Finally, we generalize a result of N.~Alrayes from the square grid setting to the orthodiagonal setting, and prove that the discrete modulus of the approximating non-crossing paths converges to the continuous modulus.
}
\vspace{0.1in}


\vspace{0.1in}


\tableofcontents

\bigskip

\section{Introduction and results}\label{intro}

Many objects in complex analysis are approximated by discrete analogues.  
In this paper, we establish the convergence of three discrete objects, discrete modulus, discrete paths that are extremal for the modulus, and discrete harmonic functions, to their continuous counterparts, the continuous modulus, the extremal curves (or horizontal trajectories) for the modulus, and harmonic functions.

Consider a Jordan domain $\Om\subset\C$ and two disjoint arcs  on $\bd\Om$.  
The modulus of the curve family connecting these arcs in $\Omega$ is well understood.
In particular, there is a subfamily of extremal curves (or horizontal trajectories) that carry all the information for this modulus.
We wish to approximate $\Omega$ with a plane graph $G$ and approximate the boundary arcs with sets of vertices $A$ and $B$.
As previously studied in \cite{abppcw:ecgd2015, acfpc:ampa2019, apc:jan2016}, the discrete modulus of the family of paths between $A$ and $B$ in $G$ arises from putting a density on the edges of $G$.  See Section \ref{DiscreteModulus} for the definitions.

Our first task, accomplished in the result below, 
is to identify a subfamily of paths between $A$ and $B$ in $G$ that is the candidate approximating family for the extremal curves in the classical setting.  
The non-crossing paths  in this subfamily are allowed to intersect (i.e.~they may share vertices and edges in common), but they may not cross.  See Section \ref{non-crossing}~for the precise definition.

\begin{thm}\label{existence}
Let $G=(V,E, \sigma)$ be a finite plane network with the set $V$ of vertices, the set $E$ of edges, and edge-weights $\si\in\R^E_{>0}$. 
Assume that there are nonempty sets of  boundary vertices $A$ and $B$  defined as follows:
   $A = S_1 \cap V$ and $B= S_2 \cap V$, 
  where $S_1, S_2$ are two disjoint arcs from a Jordan curve $\mathcal{C}$ that encircles $G$ in the closed outer face of $G$.
Consider the family $\Ga(A,B)$ of all paths in $G$ from $A$ to $B$.

Then, there is a unique subfamily $\Ga\subset\Ga(A,B)$ such that
\bi
\item[{\rm (i)}] $\Ga$ is a family of non-crossing paths (defined in Section \ref{non-crossing}); 
\item[{\rm (ii)}] $\Ga$ supports a pmf $\mu$  that is optimal for the probabilistic interpretation of the $(2,\si)$-modulus of $\Ga(A,B)$ (defined in Section \ref{sec:prob-interp-modulus});
\item[{\rm (iii)}]  and $\Ga$ is minimal for the $(2,\si)$-modulus of $\Ga(A,B)$ (defined in Section \ref{DiscreteModulus}). 
\ei
In particular, $\mu$ is the unique pmf supported on $\Ga$ that satisfies (ii).

Moreover, both $\Ga$  and $\mu$ can be computed using the output of the Non-crossing Minimal Subfamily Algorithm \ref{alg:non-cross-min-subfam} described in Section \ref{non-crossing}.
\end{thm}

Our next goal is to show that the candidate subfamily identified in Theorem \ref{existence} does in fact approximate the family of extremal curves.  
For this, we specialize to plane graphs that arise from finite orthodiagonal maps, a class of graphs that generalize the square grid.  
Roughly, an othodiagonal map is a plane graph where every interior face is a quadrilateral with orthogonal diagonals.
Associated with the othodiagonal map is the so-called primal graph $G^\bullet$ that we utilize, and the edges of $G^\bullet$ are equipped with canonical weights defined in Equation (\ref{eq:canonical-weights}).  
See Section \ref{orthodiagIntro} for definitions and notation.

Given a Jordan domain $\Om\subset\C$ and two arcs $A, B$ on $\bd\Om$, let $\psi$ be the conformal map from $\Om$ to a rectangle $R=[0,1]\times[0,m]$ so that $A$ and $B$ are sent to the vertical sides $\{0\} \times [0,m]$ and $\{1\} \times [0,m]$. 
Then the extremal curves for modulus are the pullback under $\psi$ of the horizontal segments connecting the two vertical sides of $R$.
The transverse measure on the horizontal curves in $R$ is given by the 1-dimensional Lebesgue measure on the height ($y$-coordinate in $R$) of these curves.
We define the {\it transverse measure on the extremal curves from $A$ to $B$ in $\Omega$} to be the pullback via $\psi$ of the transverse measure on the horizontal curves in $R$.

For simplicity, we state the next two results in the case when $\Om$ is an analytic quadrilateral, see Definition \ref{def:smooth-quad}.

\begin{thm} \label{limitpmf}
Let $\Omega$ be an analytic quadrilateral with boundary arcs $A, \tau_1, B, \tau_2$.
Let $\ep_n>0$ with $\ep_n \to 0$.
Let $G_n$ be an orthodiagonal map with edge length at most $\epsilon_n$  and with boundary arcs $S^n_{1},T^n_{1}, S^n_{2},T^n_{2}$, 
so that  $(G_n;S^n_{1},T^n_{1}, S^n_{2},T^n_{2})$  approximates $(\Omega;A, \tau_1, B, \tau_2)$ within distance $\e_n$ 
(in the sense of Definition \ref{def:approx-smooth-quad}). Equip $G_n$ with canonical edge-weights $\si_n$ as in Equation (\ref{eq:canonical-weights}).
For $A_n = V_n^\bullet \cap S^n_1$ and $B_n = V_n^\bullet \cap S^n_2$, consider the family $\Ga(A_n,B_n)$ of all paths in the primal graph $G^\bullet_n$ from $A_n$ to $B_n$. Let 
$\Ga_n$ and $\mu_n$ be corresponding objects for $\Ga(A_n,B_n)$, as described in Theorem \ref{existence}; namely, 
$\Gamma_n\subset \Ga(A_n,B_n)$ is the unique minimal subfamily with non-crossing paths and  $\mu_n$  is the pmf supported on $\Gamma_n$ that is optimal for the probabilistic interpretation of  $\Mod_{2,\si_n}\Ga(A_n,B_n)$ (defined in Section \ref{sec:prob-interp-modulus}).   

Then, as $n \to \infty$, $\mu_n$ converges to the transverse measure $\mu$ on the extremal curves from $A$ to $B$ in $\Omega$
(under weak convergence using  the Hausdorff metric as the underlying metric).
\end{thm}

The measures $\mu_n, \mu$ in the above theorem are supported on different sets of curves, 
but we view them as measures on the same space as follows.  
Let $X = \left(\bigcup_n \Gamma_n \right) \cup \Gamma$ where $\Gamma$ is the set of extremal curves from $A$ to $B$ in $\Omega$.
Then we consider $\mu_n, \mu$ as elements of the set $\cP(X)$ of probability measures on $(X, \mathcal{B}(X))$, where $\cB(X)$ is the Borel $\si$-algebra induced by the Hausdorff metric.

The proof of Theorem \ref{limitpmf}, relies on the fact that discrete harmonic functions approximate  continuous harmonic functions.
The following is an adaptation of the main theorem of \cite{gurel-jerison-nachmias:advm2020}.

\begin{thm} \label{HarmConvThm}
Let $\Omega$ be an analytic quadrilateral with boundary arcs $A, \tau_1, B, \tau_2$.
Let $\psi$ be a conformal map of $\Omega$ onto the interior of the rectangle $R=[0,1] \times [0,m]$ taking the arcs $A$ and $B$ to the vertical sides of the rectangle,
and let $h_c = \rea \psi$.
Let $\epsilon, \delta \in (0,\diam(\Omega))$.
Let $G= (V^\bullet \sqcup V^\circ, E)$ 
be a finite orthodiagonal map with  boundary arcs $S_1, T_1, S_2, T_2$ 
with maximal edge length at most $\e$, so that $(G;S_1, T_1, S_2, T_2)$ 
approximates $(\Omega;A, \tau_1, B, \tau_2)$ within distance $\delta$ (in the sense of Definition \ref{def:approx-smooth-quad}). Equip $G$ with canonical edge-weights $\si$ as in Equation (\ref{eq:canonical-weights}).
Let $h_d: V^\bullet \to \mathbb{R}$ be discrete harmonic on ${\rm int}(V^\bullet)$ as in (\ref{eq:si-harmonic}), with the following boundary values:
$h_d(z) = 0$ for $z \in V^\bullet \cap S_1$ and $h_d(z) = 1$ for $z \in V^\bullet \cap S_2$.
Then, for $z \in V^\bullet \cap \overline{\Omega}$,
$$|h_d(z) - h_c(z)| \leq \frac{C}{\log^{1/2}\left(\diam(\Omega)/ (\epsilon \vee \delta) \right)},$$
where $C$ depends only on $\Omega$.
\end{thm}

Since $\Omega$ is bounded by analytic arcs, we are able to analytically extend $\psi$ by reflection (see Lemma \ref{reflection}).
The conclusion of Theorem \ref{HarmConvThm} will also hold for the extended harmonic function $h_c$ at points $z \in V^\bullet \setminus \overline{\Omega}$ with $\dist(z, A \cup B) > 2(\epsilon \vee \delta)^{1/4}$ when $\epsilon \vee \delta$ is small enough.  This is accomplished in Proposition \ref{subdiagHarmConv}, and we will use this result in proving Theorem \ref{limitpmf}. 

Theorem \ref{limitpmf} shows that the discrete optimal measures for modulus weakly converge to the continuous analog. However, we want to have a stronger notion of convergence, namely, that the expected overlaps  defined right after (\ref{eq:prob-dual}) converge to the reciprocal of the continuous modulus. In view of Theorem \ref{HarmConvThm}, it means proving that the discrete Dirichlet energy of the discrete harmonic functions $h_d$ converge to the Dirichlet energy of the continuous harmonic function $h_c$. In practice, this stronger claim confirms what we see numerically in Figure \ref{fig:rect-pack}, namely that the heights in those two pictures converge to the continuous modulus. That is the content of our next main result, where we show that the discrete modulus converges to the continuous modulus. This generalizes a result by Norah 
Alrayes in  \cite{alrayes-phdthesis} from the square grid setting to the orthodiagonal maps setting. Finally, note that the convergence in energy cannot be deduced from the results in \cite{gurel-jerison-nachmias:advm2020} and instead depends crucially on Fulkerson duality for discrete modulus, see Lemma \ref{lem:fulkerson-duality}. 
\begin{thm}\label{cor4}
In the setting of Theorem \ref{limitpmf}, as $\epsilon_n \to 0$
the discrete modulus of the path family between $A_n$ and $B_n$ in $G^\bullet_n$
converges to continuous modulus of the curve family between $A$ and $B$ in $\Omega$. 
\end{thm}

It should be noted that, in the special case when the primal graph is a Delaunay triangulation and the dual graph is the Voronoi dual, as in our Figure \ref{fig:cross_domain} below, Theorem \ref{cor4} can be deduced from the finite-element literature in numerical analysis \cite{hara-mizumoto:jmsj1990,ciarlet1978}. In particular, see  \cite[Section 5.1]{skopenkov:adv-math2013} for a nice explanation of this connection. Thus, in the language of the finite-element method, Theorem \ref{cor4} corresponds to results about convergence of the Sobolev $H^1$ norm, while Theorem \ref{HarmConvThm} is about the $L^\infty$ convergence.

We end with a note about the organization of this paper.  
In Section \ref{background}, we introduce the needed background, including discrete modulus, discrete harmonic functions, and orthodiagonal maps, and we consider two illuminating examples.
We begin Section \ref{non-crossing} by discussing the Non-crossing Minimal Subfamily Algorithm and proving Theorem \ref{existence}.
Then we prove Theorem \ref{limitpmf}, using Theorem \ref{HarmConvThm} as a tool.
We end by discussing an interpretation of Theorem \ref{existence}  in terms of current flow 
and a rectangle-packing result.
Section  \ref{harmonic} contains the proof of  Theorem \ref{HarmConvThm} and Theorem \ref{cor4}.

\section*{Acknowledgments}

Lind and Poggi-Corradini would like to thank MSRI for their support while some of this research was completed. The work of Poggi-Corradini and Albin is supported in part by NSF Grant n. 2154032.

\noindent{\bf Added in proof:} After the work on this paper was completed, A.~Nachmias sent us a personal communication noting that he had a similar unpublished result as Theorem \ref{cor4}, but with a different proof.

\section{Background and examples}\label{background}

\subsection{Discrete modulus}\label{DiscreteModulus}

Let $G=(V,E, \sigma)$ be a finite weighted graph, which is also referred to as a {\it network}, with vertex set $V$, edge set $E$, and edge weights $\sigma : E \to (0, \infty)$.  We will use the term {\it vertex} and {\it node} interchangeably.
We assume that $G$ has no self-loops, but we allow multiple edges.
Let $\Gamma$ be a collection of objects in $G$.  
In this paper, most of our attention will be on the case when $\Gamma$ is a collection of paths, but we will also need modulus to be defined for a collection of $AB$-cuts in Section \ref{non-crossing}.
A {\it path} is a sequence of edges  $(x_{k-1}x_k)_{k=1}^n$ so that the vertices $x_0,\dots,x_n$ are all distinct.
Given  $A,B \subset V$, then $S \subset V$ is an {\it $A  B  $-cut} if $A  \subset S$ and $B  \cap S=\emptyset$. 
The {\it usage matrix} $\mathcal{N}$ is a $|\Gamma | \times |E|$ matrix; 
in the case of paths  $\mathcal{N}(\gamma, e) = \mathbbm{1}_{\{ e \in \gamma\}}$, 
and in the case of 
 $A B $-cuts  $\cN(S,e)=\ones_{\{e\in\delta S\}}$, where $\delta S=\{e=xy\in E : |S\cap \{x,y\}|=1\}$. 

A function $\rho: E \to [0, \infty)$, referred to as a {\it density}, is called {\it admissible for $\Gamma$} if
$$\ell_\rho(\gamma) := \sum_{e \in E} \mathcal{N}(\gamma, e) \rho(e) \geq 1 \;\; \text{ for all } \gamma \in \Gamma.$$
When $\gamma$ is a path, we consider $\ell_\rho(\gamma)$ to be the $\rho$-length of $\gamma$.
This means that
 $\rho$ is admissible for a collection of paths if every path has $\rho$-length at least $1$.
For $ 1 \leq p < \infty$, the {\it (discrete) $(p,\si)$-modulus of $\Gamma$} is defined as
$$\Mod_{p,\si}(\Gamma) := \inf \sum_{e \in E} \sigma(e) \rho(e)^p,$$
where the infimum is taken over all admissible densities $\rho$.
See \cite{abppcw:ecgd2015, acfpc:ampa2019, apc:jan2016} 
 for further background on the discrete modulus.
In the rest of this paper, we will specialize to the case when $p=2$.

The {\it extremal density} $\rho^*$ is an admissible density that satisfies 
$$\Mod_{2,\si}(\Gamma) = \sum_{e \in E} \sigma(e) \rho^*(e)^2.$$
The extremal density is unique in this case and  is characterized in the following theorem, 
which  is the $p=2$ case of a well-known result (see for instance,  \cite[Theorem 2.1]{apc:jan2016}, \cite[Theorem 4-4]{ahlfors1973}, and also \cite{badger:2013finn}).

\begin{thm} (Beurling's Criterion) \label{BeurlingCriterion}
Let $G=(V, E, \sigma)$ be a network and $\Gamma$ a family of paths on $G$.
An admissible density $\rho$ is extremal for $\Mod_{2,\si}(\Gamma)$ if there is a subfamily $\tilde \Gamma \subset \Gamma$ that satisfies the following two conditions:
\begin{enumerate}
\item[(i)] $\ell_\rho(\gamma) = 1$ for all $\gamma \in \tilde \Gamma$. 
\item[(ii)] If $h: E \to \mathbb{R}$ satisfies $\ell_h(\gamma) = \sum_{e \in\gamma} h(e) \geq 0$ for all $\gamma \in \tilde \Gamma$, then 
$$\sum_{e \in E} h(e) \rho(e) \sigma(e) \geq 0.$$
\end{enumerate}
Furthermore, in this case
$\Mod_{2,\si}(\tilde \Gamma) = \Mod_{2,\si}(\Gamma)$.
\end{thm}

A subfamily $\tilde \Gamma$ that satisfies the two conditions of Beurling's Criterion for the extremal density $\rho^*$ is called a {\it Beurling subfamily} of $\Gamma$.

A {\it minimal subfamily} of $\Gamma$ is a subfamily $\tilde \Gamma\subset\Ga$ so that $\Mod_{2,\si}(\tilde \Gamma) = \Mod_{2,\si}(\Gamma)$
and  such that removing any path from $\tilde \G$ decreases the modulus, 
i.e. for any $\gamma \in \tilde \G$ we have $\Mod_{2,\si}(\tilde \Gamma \setminus \{\gamma \}) < \Mod_{2,\si}(\Gamma)$.
As shown in \cite[Theorem 3.5]{apc:jan2016}, $\Gamma$ will always contain a minimal subfamily, although it may not be unique, and any minimal subfamily is also a Beurling subfamily for $\Ga$.

\subsection{The probabilistic interpretation of discrete modulus}\label{sec:prob-interp-modulus}

Here we recall the fact that, following  \cite{apc:jan2016} and \cite{acfpc:ampa2019},  Lagrangian duality yields a probabilistic interpretation for discrete modulus.

Assume that $\Ga$ is a finite family of objects on $G$ with usage matrix $\cN$.
In order to formulate the probabilistic interpretation of $p$-modulus, we let $\cP(\Gamma)$
represent the set of probability mass functions (pmfs) on the set
$\Gamma$.  In other words, $\cP(\Gamma)$ contains the set of vectors
$\mu\in\mathbb{R}_{\ge 0}^\Gamma$ with the property that
$\mu^T\one= 1$ (here $\one\in\R_{\ge 0}^\Ga$ is the vector of all ones).  Given a pmf $\mu$, we can define a
$\Gamma$-valued random variable $\underline{\gamma}$ with distribution:
\[\mathbb{P}_\mu\left(\underline{\gamma}=\gamma\right) = \mu(\gamma).\]
Then, for every edge $e\in E$, the usage $\cN(\underline{\gamma},e)$ is also
a random variable, and we represent the expected edge-usage (depending on the
pmf $\mu$) as $\mathbb{E}_\mu\left[\cN(\underline{\gamma},e)\right]$.
With these notations, Lagrangian duality yields the following probabilistic interpretation
\begin{thm}[Theorem 2 of \cite{acfpc:ampa2019}]\label{thm:prob-interp}
  Let $G=(V,E,\si)$ be a finite graph with edge weights $\sigma\in\R^E_{>0}$, and let
  $\Gamma$ be a finite family of objects on $G$ with usage matrix
  $\cN$.  Then, for any $1<p<\infty$, letting $q=p/(p-1)$ be the conjugate exponent to $p$, we have
\begin{equation}
    \label{eq:prob-interp-p}
    \Mod_{p,\sigma}(\Gamma)^{-\frac{q}{p}} = 
      \min_{\mu\in\cP(\Gamma)}\sum_{e\in E}\sigma(e)^{-\frac{q}{p}}
      \bE_\mu\left[\cN(\underline{\gamma},e)\right]^q.
\end{equation}
Moreover, $\mu\in\cP(\Ga)$ attains the minimum on the right-hand side of~\eqref{eq:prob-interp-p}, if and only if
\begin{equation}\label{eq:expected-edge-usage-of-optimal-mu}
\bE_{\mu}\left[\cN(\underline{\gamma},e)\right]=\frac{\si(e)\rho^*(e)^{p-1}}{\Mod_{p,\si}(\Ga)}\qquad\forall
e\in E,
\end{equation}
where $\rho^*$ is the unique extremal density for $\Mod_{p,\si}(\Ga)$.
\end{thm}
If $\mu\in\cP(\Ga)$ attains the minimum on the right-hand side of~\eqref{eq:prob-interp-p}, or equivalently, satisfies (\ref{eq:expected-edge-usage-of-optimal-mu}), we say that $\mu$ {\it is optimal for $\Mod_{p,\si}(\Ga)$}. By a slight abuse of notation, we will use the expression ``$\mu$ is optimal for $\Mod_{p,\si}(\Ga)$'' even in cases where $\mu$ is initially defined on a subfamily $\tilde{\Ga}$, since such pmf's embed into $\cP(\Ga)$ by setting them equal to zero outside of $\tilde\Ga$. Optimal pmf's for a given modulus problem always exist, but they might not be unique in general.

In this paper, we are mostly interested in the case when $p=2$ and $\Ga$ is a family of paths. In this setting, the probabilistic interpretation says that modulus can be computed by finding a pmf on $\Ga$ that minimizes the expected (weighted) overlap of two independent paths in $\Ga$. We describe this in the following corollary.
\begin{cor}\label{cor:prob}
Consider a  graph $G=(V,E,\si)$ with edge-weights $\si\in\R^E_{>0}$. Let $\Gamma$ be a finite family of paths. Then
  \begin{equation}\label{eq:prob-dual}
    \Mod_{2,\si}(\Gamma)^{-1} =  \min_{\mu\in\cP(\Gamma)}\mathbb{E}_\mu\left|
    \underline{\gamma}\cap\underline{\gamma}'
  \right|_\si  =  \min_{\mu\in\cP(\Gamma)} \sum_{e\in E}\si(e)^{-1}\left(\bP_\mu(e\in \underline{\ga})\right)^2,
  \end{equation}
where $\underline{\gamma}$ and $\underline{\gamma}'$ are two
independent random paths chosen according to $\mu$, and the weighted overlap $\left| \gamma\cap\gamma' \right|_\si$ of two paths is $\sum_{e\in\ga\cap\ga'} \si(e)^{-1}$. 

Moreover, $\mu^*\in\cP(\Ga)$ is optimal for $\Mod_{2,\si}(\Ga)$, i.e., attains the minimum on the right-hand side of~\eqref{eq:prob-dual}, if and only if
\begin{equation}\label{eq:rhomu}
\bP_{\mu^*}\left(e\in\underline{\ga}\right) =  \frac{\si(e)\rho^*(e)}{\Mod_{2,\si}(\Gamma)}, \qquad\forall e\in E,
\end{equation}
where $\rho^*$ is the unique extremal density for $\Mod_{2,\si}(\Ga)$.
\end{cor}
\begin{proof}
Note that for paths we have $\bE_\mu\left[\cN(\underline{\gamma},e)\right]=\bP_\mu(e\in \underline{\ga})$. Hence, when $p=2$, the quantity minimized in the right hand-side of  (\ref{eq:prob-interp-p}) becomes:
\begin{align*}
\sum_{e\in E}\si(e)^{-1}\left( \bE_\mu\left[\cN(\underline{\gamma},e)\right]\right)^2 
& = \sum_{e\in E}\si(e)^{-1}\left(\bP_\mu(e\in \underline{\ga})\right)^2 \\
& = \sum_{e\in E}\si(e)^{-1}\left(\sum_{\ga\in\Ga}\mu(\ga)\cN(\ga,e)\right)^2\\
& = \sum_{\ga,\ga'\in\Ga} \mu(\ga)\mu(\ga')\left(\sum_{e\in E}\si(e)^{-1}\cN(\ga,e)\cN(\ga',e)\right)\\
& = \mathbb{E}_\mu\left|
    \underline{\gamma}\cap\underline{\gamma}'
  \right|_\si.
\end{align*}
\end{proof}
As mentioned above, the optimal pmf's $\mu^*$ for (\ref{eq:prob-dual}) are in general not unique, even though the extremal density $\rho^*$ is. When $\Ga$ is itself minimal, then the optimal $\mu^*$ does become unique. This fact can be deduced from \cite[Theorem 3.5 and Section 5]{apc:jan2016}, however, it is not explicitly formulated there, so we summarize it here, for clarity, and give a proof using the pmf's directly instead of Lagrange multipliers. We also expanded this result for our current purposes.
\begin{prop}\label{prop:pmf-minimal}
Let $G=(V,E, \sigma)$ be a finite graph with edge-weights $\si\in\R^E_{>0}$, and let $\Gamma$ be a finite (non-empty) family of paths in $G$. 
Suppose $\tilde{\Gamma}\subset\Ga$ is a minimal subfamily for $\Ga$, then there is a unique pmf $\mu^*\in\cP(\tilde{\Gamma})$ such that 
\bi
\item[(a)] $\mu^*$ is optimal for $\Mod_{2,\si}(\Ga)$, and
\item[(b)] $\mu^*$ is strictly supported on $\tilde{\Ga}$, i.e., $\mu^*(\ga)>0$ if and only if $\ga\in\tilde\Gamma$. 
\ei
\end{prop} 
\begin{proof}
Let $\cM$ be the set of pmf's in $\cP(\tilde{\Ga})$ that satisfy (a) and (b).

We want to show that, if $\tilde{\Gamma}\subset\Ga$ is a minimal subfamily for $\Ga$, then $|\cM|=1$. Since optimal pmf's always exist, there is at least one pmf $\mu^*\in\cP(\tilde{\Ga})$ that is optimal for $\Mod_{2,\si}(\tilde{\Ga})$.  Since $\Mod_{2,\si}(\tilde{\Ga})=\Mod_{2,\si}(\Ga)$, Equation (\ref{eq:prob-dual}) implies that
\[
 \min_{\tilde{\mu}\in\cP(\tilde\Gamma)}\mathbb{E}_{\tilde{\mu}}\left|
    \underline{\gamma}\cap\underline{\gamma}'
  \right|_\si=
   \min_{\mu\in\cP(\Gamma)}\mathbb{E}_\mu\left|
    \underline{\gamma}\cap\underline{\gamma}'
  \right|_\si.
\]
And, since $\tilde{\Gamma}\subset\Ga$, the pmf $\mu^*$ is also in $\cP(\Ga)$ (extending it to be zero outside of $\tilde{\Ga}$). Therefore, $\mu^*$ is also optimal for $\Mod_{2,\si}(\Ga)$. Thus, we have shown that $\mu^*$ satisfies property (a).

Next, property (b) will follow from Claim \ref{cl:min-subfam} below, and that will show that $\mu^*\in\cM$.
\begin{claim}\label{cl:min-subfam}
If $\tilde{\Ga}\subset\Ga$ is a minimal subfamily for $\Ga$ and $\mu_0$ is a pmf on $\tilde{\Ga}$ that is optimal for $\Mod_{2,\si}(\Ga)$, then $\mu_0$ is strictly supported on $\tilde{\Ga}$.
\end{claim}
\begin{proof}[Proof of Claim \ref{cl:min-subfam}]
Suppose that there is a path $\ga\in\tilde{\Ga}$ with $\mu_0(\ga)=0$. Define $\Ga':=\tilde{\Ga}\setminus\{\ga\}$. Then,  by minimality of $\tilde{\Ga}$, we must have
\begin{equation}\label{eq:minmutminmup}
 \min_{\tilde{\mu}\in\cP(\tilde\Gamma)}\mathbb{E}_{\tilde{\mu}}\left|
    \underline{\gamma}\cap\underline{\gamma}'
  \right|_\si  <
   \min_{\mu'\in\cP(\Gamma')}\mathbb{E}_{\mu'}\left|
    \underline{\gamma}\cap\underline{\gamma}'
  \right|_\si.
\end{equation}
However, the restriction of $\mu_0$ to $\Ga':=\tilde{\Ga}\setminus\{\ga\}$ is also a pmf in $\cP(\Ga')$, which we again denote by $\mu_0$. Hence,
\begin{equation}\label{eq:minmupmuo}
   \min_{\mu'\in\cP(\Gamma')}\mathbb{E}_{\mu'}\left|
    \underline{\gamma}\cap\underline{\gamma}'
  \right|_\si
  \le
  \mathbb{E}_{\mu_0}\left|
    \underline{\gamma}\cap\underline{\gamma}'
  \right|_\si
\end{equation}
Finally, we have
\begin{align*}
\mathbb{E}_{\mu_0}\left|\underline{\gamma}\cap\underline{\gamma}'\right|_\si
&  = \min_{\mu\in\cP(\Gamma)}\mathbb{E}_{\mu}\left|\underline{\gamma}\cap\underline{\gamma}'\right|_\si  
& \text{(since $\mu_0$ is optimal for $\Mod_{p,\si}(\Ga)$)}  \\
&  \le \min_{\tilde{\mu}\in\cP(\tilde\Gamma)}\mathbb{E}_{\tilde{\mu}}\left|
\underline{\gamma}\cap\underline{\gamma}'\right|_\si
& \text{(since $\tilde{\Ga}\subset\Ga$)}  \\
& < \min_{\mu'\in\cP(\Gamma')}\mathbb{E}_{\mu'}\left|\underline{\gamma}\cap\underline{\gamma}'\right|_\si.
& \text{(by (\ref{eq:minmutminmup}))} \\
& \le  \mathbb{E}_{\mu_0}\left|\underline{\gamma}\cap\underline{\gamma}'\right|_\si,
& \text{(by (\ref{eq:minmupmuo}))}
\end{align*}
which is a contradiction. Therefore, $\mu_0$ is strictly supported on $\tilde{\Ga}$.
\end{proof}
We have shown the existence of a pmf $\mu^*$ in $\cM$.
To prove uniqueness, suppose $\mu_0,\mu_1\in\cM$ are two distinct pmfs. Define the linear combination $\mu_t:=t\mu_1+(1-t)\mu_0$, for $t\in\R$.  
We claim that there are two values of $t$ so that $t_0<0<1<t_1$, and so that  $\mu_t\in\cP(\tilde{\Ga})$ for $t\in [t_0,t_1]$. 
Indeed, by linearity, $\one^T\mu_t=1$, for every $t\in\R$. Also, 
since $\mu_0$ and $\mu_1$ are strictly supported on $\tilde{\Ga}$, we have $\min_{j=0,1}\mu_j(\ga)>0$, for every $\ga\in\tilde\Ga$. In particular, if $\ga\in\tilde{\Ga}$ is such that $\mu_0(\ga)=\mu_1(\ga)$, then $\mu_t(\ga)>0$ for all $t\in\R$. Moreover,
since $\mu_0\ne \mu_1$ and $\one^T(\mu_0-\mu_1)=0$, there is at least one pair $(\ga,\ga')$ of paths in $\tilde{\Ga}$ such that $\mu_0(\ga)>\mu_1(\ga)$ and $\mu_0(\ga')<\mu_1(\ga')$. Then, $\mu_t(\ga)>0$ if and only if $t<t_1(\ga):=\frac{\mu_0(\ga)}{\mu_0(\ga)-\mu_1(\ga)}$, and
 $\mu_t(\ga')>0$ if and only if $t>t_0(\ga'):=-\frac{\mu_0(\ga')}{\mu_1(\ga')-\mu_0(\ga')}.$
 Note that  $(t_0(\ga'),t_1(\ga))\supset[0,1]$. If we take the (finite) intersection of all such intervals $(t_0(\ga'),t_1(\ga))$, over all possible pairs $(\ga,\ga')$ as above, we find an interval
$(t_0,t_1)\supset[0,1]$, such that $\mu_t(\ga)\ge 0$ for all $t\in [t_0,t_1]$ and for all $\ga\in\tilde{\Ga}$, with the further property that when $t=t_0$ or $t=t_1$, then $\mu_t$ must vanish on some path in $\tilde{\Ga}$.
In particular, we have shown that $\mu_t\in\cP(\Ga)$ for all $t\in [t_0,t_1]$. 

Finally, since $\mu_0$ and $\mu_1$ are optimal pmfs for $\Mod_{2,\si}(\Ga)$, they both satisfy  Equation (\ref{eq:rhomu}). By linearity, $\mu_t$  will also satisfy  Equation (\ref{eq:rhomu}), for $t\in [t_0,t_1]$, and therefore, by Corollary \ref{cor:prob}, the pmf $\mu_t$ will also be optimal for $\Mod_{2,\si}\Ga$, for $t\in [t_0,t_1]$. Hence, by construction, $\mu_{t_0}$ and $\mu_{t_1}$ are pmf's in $\cP(\tilde\Ga)$ that are optimal for $\Mod_{2,\si}\Ga$, and that vanish on some path $\ga\in\tilde\Ga$. This contradicts Claim \ref{cl:min-subfam}. Hence we have shown that the minimality of $\tilde\Ga$ implies $|\cM|=1$.
\end{proof}

\subsection{Discrete harmonic functions}\label{sec:discrete-harm-fnct}

For a network $G=(V,E, \sigma)$, when $v, w \in V$ are adjacent, we will use the notation $vw$ to refer to the edge incident to $v$ and $w$.  
A function $f:V \to \mathbb{R}$ is {\it discrete harmonic} at $v \in V$ if
\begin{equation}\label{eq:si-harmonic}
 f(v) \cdot \sum_{w: \, vw \in E} \sigma(vw) = \sum_{w: \, vw \in E} f(w) \sigma(vw),
 \end{equation}
or in other words, the value of $f$ at $v$ is a weighted average of the values of $f$ at the neighbors of $v$.
The  existence and uniqueness of discrete harmonic functions with specified boundary values is well known; see for instance Section 2.1 of \cite{lyons-peres:2016}.

\begin{prop} (Existence and uniqueness of discrete harmonic functions) \label{discreteDP}
Let $G= (V, E, \sigma)$ be a finite connected network. 
For $U \subsetneq V$ and $g:V \setminus U \to \mathbb{R}$, 
there is a unique function $h:V \to \mathbb{R}$ such that $h=g$ on $V \setminus U$ and $h$ is discrete harmonic on $U$.
\end{prop}

The function $h$ given by Proposition \ref{discreteDP} is called the {\it solution to the discrete Dirichlet problem} on $U$ with boundary data $g$.
Viewed from the electric network perspective, $h$ is called the {\it voltage function}.  From its discrete gradient, we obtain the {\it current flow} $f=\sigma dh$ defined by 
$$ f(vw) := \sigma(vw) \left[ h(w) - h(v) \right] \;\;\; \text{ for all } vw \in \vv{E}.$$
In the above, we use $\vv{E}$ for the set of directed edges, which contains both possible orientations for all the edges of $E$.  
In this setting, $vw$ represents the directed edge from $v$ to $w$.

More generally,  a {\it flow} on  $U \subset V$ is a function $\theta$ on $\vv{E} $ so that $\theta(vw) = -\theta(wv)$ for all $vw \in \vv{E} $ 
and each $v \in U$ satisfies {\it Kirchhoff's node law}, i.e.
$$ \sum_{w: \, vw \in \vv{E}} \theta(vw) = 0.$$
Furthermore, if $A$ and $B$ are nonempty and disjoint subsets in $V$ and $U:=V\setminus (A\cup B)$, then we say that $\theta$ is {\it a flow from $A$ to $B$}, if $\theta$ is a flow on $U$. In this case, the {\it strength} (or {\it value}) of $\theta$ is given by the formula:
\[
{\rm strength}(\theta):=\sum_{w: \, vw \in \vv{E}, v\in A} \theta(vw) =\sum_{v: \, vw \in \vv{E}, w\in B} \theta(vw).
\]
Moreover, such a flow is said to be a {\it unit flow} if it has unit strength.

The {\it energy} of a symmetric or anti-symmetric function $\theta: \vv{E} \to \mathbb{R}$ is given (with a slight abuse of notation) by 
$$\mathcal{E}(\theta) = \sum_{e \in E} \frac{1}{\sigma(e)} \theta^2(e).$$
The energy of  $h: V \to \mathbb{R}$ is defined to be the energy of its discrete gradient, i.e.
$$\mathcal{E}(h) = \mathcal{E}(\sigma dh) =  \sum_{vw \in E} \sigma(vw) \left[ h(w) - h(v) \right]^2.$$

We are now able to explain the close connection between the discrete 2-modulus and discrete harmonic functions.
Suppose $G$ is network with disjoint nonempty sets $A,B \subset V$, and edge-weights $\si\in\R^E_{>0}$.
Let $\Gamma$ be the family of paths in $G$ from $A$ to $B$, and let $h$ be the solution to the discrete Dirichlet problem on $V \setminus (A \cup B)$ with $h |_A = 0$ and $h |_B = 1$.
Then,
\begin{equation}\label{eq:connection-harmod}
\Mod_{2,\si}(\Gamma) = \mathcal{E}(h),
\end{equation}
and the extremal density is related to the solution of the Dirichlet problem as follows
\begin{equation} \label{rhofromharmonic}
\rho^*(vw) = |h(w) - h(v)|
\end{equation}
for all $vw \in E$.
See Theorem 4.2 in \cite{abppcw:ecgd2015}.

\subsection{Examples}

We wish to illustrate the concepts from the previous sections and our main results by taking a look at two simple examples.
\begin{example}\label{ex:horiz-rect}
For the first example, let $R$ be the rectangle $[0,L] \times [0,1]$, where $L \in \mathbb{N}$.  
Let $G$ be the square grid in $R$ with lattice size $1/n$ and weights $\sigma \equiv 1$. 
Let $\Gamma$ be the family of all paths in $G$ from  $A = \{v \in V \, : \, \text{Re}(v) = 0 \}$ to $B = \{v \in V \, : \, \text{Re}(v) = L\}$.
\end{example}
 For this example, we claim the following:
\begin{enumerate}
\item The extremal density is 
$\rho^*(e) = \begin{cases}
	\frac{1}{Ln}					&\quad\text{if } e \text{ is a horizontal edge,}  \\  
	0							&\quad\text{if } e \text{ is a vertical edge.}  \\               
      \end{cases}$

\item The subfamily of horizontal paths $\tilde \Gamma$ is a Beurling subfamily.
\item $\Mod_2(\Gamma) = \Mod_2(\tilde \Gamma) = \frac{1}{L} \left( 1+ \frac{1}{n} \right) $.
\item $\tilde \Gamma$ is  the unique minimal subfamily.
\item The optimal pmf $\mu^*$ is unique and is uniform  on $\tilde \Gamma$.
\end{enumerate}

We will use Beurling's criterion to establish the first two claims.  
Note that each horizontal path $\gamma \in \tilde \Gamma$ has $Ln$ horizontal edges, and so $\ell_{\rho^*}(\gamma) = 1$.
To establish the second condition of Beurling's criterion, assume that $h: E \to \mathbb{R}$ satisfies 
$\sum_{e \in\gamma} h(e)  \geq 0$ for all $\gamma \in \tilde \Gamma$.  
Then, since $\rho^*(e)=0$ when $e$ is vertical,
$$\sum_{e \in E} h(e) \rho^*(e)  = \frac{1}{Ln} \sum_{e \text{ horizontal}}  h(e)  = \frac{1}{Ln} \sum_{\gamma \in \tilde \Gamma} \sum_{e \in \gamma} h(e) \geq 0.$$
Thus Beurling's criterion (Theorem \ref{BeurlingCriterion}) shows that $\rho^*$ is the extremal density and $\Mod_2(\Gamma) = \Mod_2(\tilde \Gamma)$.
We compute
$$\Mod_2(\Gamma) = \sum_{e \in E} \rho^*(e)^2 = \left( \frac{1}{Ln} \right)^2  (\text{\# of horizonal edges} )
= \frac{1}{L} \left( 1+ \frac{1}{n} \right),$$
since there are $Ln(n+1)$ horizontal edges.
Notice that as $n \to \infty$, the discrete modulus converges to $\frac{1}{L}$, which is the corresponding continuous modulus.

An alternate approach to finding $\rho^*$ would be to check that 
$h(v) = \frac{1}{L}\text{Re}(v)$ is discrete harmonic on $V \setminus (A \cup B)$.  
Thus $h$ is the solution to the discrete Dirichlet problem with boundary values $h |_A = 0$ and $h |_B = 1$, and we could obtain $\rho^*$ from \eqref{rhofromharmonic}.

We observe that $\tilde \Gamma$ is a minimal subfamily, since removing any path $\gamma \in \tilde \Gamma$ will decrease the modulus.  To see this, change $\rho^*$ to be zero on the removed path $\gamma$ and use Beurling's criterion to show that the new density is extremal for $\tilde \Gamma \setminus \{ \gamma \}$.
To show that $\tilde \Gamma$ is the unique minimal subfamily, assume to the contrary that $\hat \Gamma$ is another minimal subfamily. 
Then the optimal pmf $\hat \mu$ on $\hat \Gamma$ satisfies that the expected edge usage of $e$ under $\hat \mu$ is given by 
$ \rho^*(e)/\Mod_2(\Gamma)$ by \eqref{eq:rhomu}.
This implies that the expected edge usage for any vertical edge must be zero, and the paths in $\hat \Gamma$ cannot use any vertical edges.  Therefore $\hat \Gamma \subset \tilde \Gamma$.  However, since $\Mod_2(\hat \Gamma) = \Mod_2(\tilde \Gamma)$, we must have $\hat \Gamma = \tilde \Gamma$.

Let $\mu^*$ be the optimal pmf on $\tilde{\Gamma}$.  From \eqref{eq:rhomu}, we see that if $e$ is a horizontal edge and $\ga$ is the corresponding horizontal path, then
$$ \mu^*(\ga)=\mathbb{P}_{\mu^*}\left[ \gamma \text{ contains } e\right] = \frac{\rho^*(e)}{\Mod_2(\Gamma)}  = \frac{1}{n+1} .$$
This implies that $\mu^*$  is uniform on $\tilde{\Gamma}$.

We wish to vary $n$, and so we update our notation to show the dependence on $n$, i.e. let $\mu^*_n = \mu^*$ which is supported on $\tilde \Gamma_n = \tilde \Gamma$.  
For this example, the result of Theorem \ref{limitpmf} is fairly intuitive: as $n\to \infty$, $\mu^*_n$ converges to the transverse measure on all horizontal segments in $R$, which are the extremal curves in the continuous setting.

\begin{example}\label{RotatedRectangle}
In our second example, we will look at a rotated rectangle.
Let $R$ be the rectangle  $[0,\sqrt{2} L] \times [0,\sqrt2]$ rotated by $\pi/4$ (i.e.~the corners of $R$ are $0, L+iL, L-1+i( L+1), -1+i$.)  In R, take the  (non-rotated) square grid with lattice size $1/n$ and weights $\sigma \equiv 1$.  Let $\Gamma$ be the family of all paths in $R$ from the left side $A = \{ v \, : \, \text{Im}(v) = -\text{Re}(v) \}$ to the right side $B =  \{ v \, : \, \text{Im}(v) = -\text{Re}(v) +2L\}$.  See Figure \ref{Example2}.
\end{example}

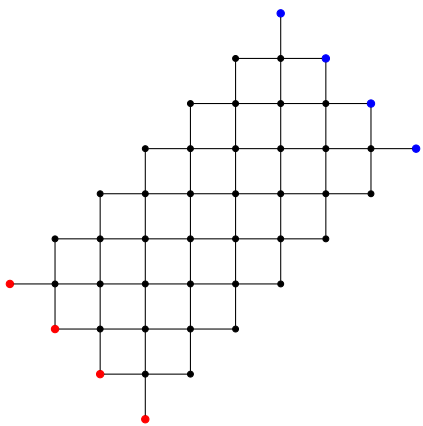
\begin{figure}
 \centering
\begin{tikzpicture}[scale=0.6, auto, node distance=3cm,  thin]
   \begin{scope}[every node/.style={circle,draw=black,fill=black!100!,font=\sffamily\Large\bfseries}]
   \node (v1) [scale=0.2]at (1,1) {};
     \node (v2) [scale=0.2]at (2,2) {};
    \node (v3)[scale=0.2] at (3,3) {};
      \node (v4) [scale=0.2]at (4,4) {};
   \node (vv1) [scale=0.2]at (0,1) {};
     \node (vv2) [scale=0.2]at (1,2) {};
    \node (vv3)[scale=0.2] at (2,3) {};
      \node (vv4) [scale=0.2]at (3,4) {};
    \node (vv5) [scale=0.2]at (4,5) {};
       \node (vv6) [scale=0.2]at (5,6) {};
   \node (w1) [scale=0.2]at (0,2) {};
     \node (w2) [scale=0.2]at (1,3) {};
    \node (w3)[scale=0.2] at (2,4) {};
      \node (w4) [scale=0.2]at (3,5) {};
   \node (ww1) [scale=0.2]at (-1,2) {};
     \node (ww2) [scale=0.2]at (0,3) {};
    \node (ww3)[scale=0.2] at (1,4) {};
      \node (ww4) [scale=0.2]at (2,5) {};
    \node (ww5) [scale=0.2]at (3,6) {};  
        \node (ww6) [scale=0.2]at (4,7) {};  
       \node (x1) [scale=0.2]at (-1,3) {};
     \node (x2) [scale=0.2]at (0,4) {};
    \node (x3)[scale=0.2] at (1,5) {};
      \node (x4) [scale=0.2]at (2,6) {};
   \node (xx1) [scale=0.2]at (-2,3) {};
     \node (xx2) [scale=0.2]at (-1,4) {};
    \node (xx3)[scale=0.2] at (0,5) {};
      \node (xx4) [scale=0.2]at (1,6) {};
    \node (xx5) [scale=0.2]at (2,7) {}; 
        \node (xx6) [scale=0.2]at (3,8) {}; 
           \node (y1) [scale=0.2]at (-2,4) {};
     \node (y2) [scale=0.2]at (-1,5) {};
    \node (y3)[scale=0.2] at (0,6) {};
      \node (y4) [scale=0.2]at (1,7) {};
      \node (v5)[scale=0.2] at (5,5) {};
     \node (w5)[scale=0.2] at (4,6) {};
       \node (x5)[scale=0.2] at (3,7) {};
       \node (y5)[scale=0.2] at (2,8) {};
      \end{scope}
\begin{scope}[every node/.style={circle,draw=red,fill=red!100!,font=\sffamily\Large\bfseries}]
    \node (v0)[scale=0.25] at (0,0) {};
     \node (w0)[scale=0.25] at (-1,1) {};
       \node (x0)[scale=0.25] at (-2,2) {};
       \node (y0)[scale=0.25] at (-3,3) {};
\end{scope}
\begin{scope}[every node/.style={circle,draw=blue,fill=blue!100!,font=\sffamily\Large\bfseries}]
     \node (v6)[scale=0.25] at (6,6) {};
     \node (w6)[scale=0.25] at (5,7) {};
       \node (x6)[scale=0.25] at (4,8) {};
       \node (y6)[scale=0.25] at (3,9) {};
\end{scope}
\begin{scope}[every edge/.style={draw=black,thin}]
    \draw  (w0) edge node{} (v1);
     \draw  (x0) edge node{} (v2);
        \draw  (y0) edge node{} (v3); 
      \draw  (y1) edge node{} (v4);    
   \draw  (y2) edge node{} (v5); 
      \draw  (y3) edge node{} (v6);   
      \draw (y4) edge node{} (w6);     
            \draw (y5) edge node{} (x6);  
      \draw (x0) edge node{} (y1);
      \draw (w0) edge node{} (y2);
       \draw (v0) edge node{} (y3);
         \draw (v1) edge node{} (y4);
           \draw (v2) edge node{} (y5);
           \draw (v3) edge node{} (y6);
            \draw (v4) edge node{} (x6); 
                        \draw (v5) edge node{} (w6); 
            \end{scope}
\end{tikzpicture}           

\caption{The graph for Example \ref{RotatedRectangle} with $L=2$, $n=3$, the vertices of $A$ shown in red, and the vertices of $B$ shown in blue.} \label{Example2}
\end{figure}

In contrast to Example \ref{ex:horiz-rect}, there is not a unique minimal subfamily in this case, and we will describe two  such subfamilies.
Consider first the subfamily $\Gamma^z$ of ``zig-zag" paths, which alternate between increasing the real component and increasing the imaginary component (or vice-versa).  These paths each have $2Ln$ steps, and each edge is used in exactly 1 path.  By arguing as before,
Beurling's criterion can be used to show that the extremal density is 
 $\frac{1}{2Ln}$ on all edges and $\Gamma^z$ is a Beurling subfamily.  Further we see that $\Gamma^z$ is also a minimal subfamily, since removing any path will decrease the modulus.  
 From the extremal density and the fact that there are $4n^2L$ total edges, we see that $\Mod_2(\Gamma) = \frac{1}{L}$, which is equal to the corresponding continuous modulus.

Another minimal subfamily $\Gamma^r$ consists of the ``reflected'' paths which keep increasing the same component (whether real or imaginary) until they hit the boundary of $R$ and then they switch to increasing the other component.  The paths in this family will also have $2Ln$ steps each, and each edge is used in exactly 1 path.  One can check that this will also be  a Beurling subfamily  and  a minimal subfamily.

Each subfamily will have an associated optimal pmf:
let $\mu_n^z$ be the pmf on the zig-zag subfamily $\Gamma^z$, and 
let $\mu_n^r$ be the pmf on the reflected subfamily $\Gamma^r$.
Both measures will be uniform, giving the $2n$ paths in the relevant subfamily weight $\frac{1}{2n}$.  
Further both $\mu_n^z$ and $\mu_n^r$ will have a limit as $n \to \infty$, but the limiting objects will be different.  
The limit for $\mu_n^z$ is the transverse measure on the straight line paths of slope 1 in $R$ (which are the extremal curves in the continuous setting).  
However, the limit for $\mu_n^r$ will be supported on paths that move in a straight line either horizontally or vertically until reflecting off the boundary of the rectangle.

To illustrate the point that not all Beurling subfamilies are minimal, we wish to examine one further subfamily.  Let
$$\hat{\Gamma} = \{ \gamma \in \Gamma \, : \, \gamma \text{ has } 2Ln \text{ steps}\}.$$ 
In other words, this is the collection of all paths that have length 1 under the extremal density. 
This family contains both $\Gamma^z$ and $\Gamma^r$.
In a moment, we will show that this is a Beurling subfamily, and hence it will be the largest Beurling subfamily by construction.

Let $\hat{\mu}$ be the pmf that arises from viewing the paths in $\hat{\Gamma}$ as 1-dim simple random walk (SRW) trajectories started on the left side of $R$ and reflected to stay in the rectangle.  In other words, we can think of creating the paths dynamically as follows: Each of the $2n$ initial edges has equal probability $\frac{1}{2n}$.  Once the first edge is chosen, we chose the next edge with probability  $1/2$ of going up or going to the right.  Then we repeat.  If we happen to be on the boundary, then we take the only available edge with probability 1.  One can show for any fixed edge $e$ that
 $\mathbb{P}_{\hat{\mu}}[e \in \gamma] = \frac{1}{2n}$, see Claim \ref{cl:edge-prob-rrw} below.

 With this understanding of $\hat{\mu}$, we will be able to  show that $\hat \Gamma$ is a Beurling subfamily by verifying Beurling's Criterion.  Let $h \in \mathbb{R}^E$ with $\sum_{e \in \gamma} h(e) \geq 0$ for all $\gamma \in \hat\Gamma $.  
Then
$$  \sum_{e \in E} h(e) \cdot \frac{1}{2nL} 
  =  \sum_{e \in E} h(e) \cdot \frac{1}{L} \sum_{\gamma \in \hat\Gamma} \hat\mu(\gamma)  \mathbbm{1}_{\{e \in \gamma\}}
  = \frac{1}{L} \sum_{\gamma \in \hat\Gamma}   \hat\mu(\gamma)  \sum_{e \in \gamma} h(e) \geq 0.$$

  A short computation will show that $\hat \mu$ is a pmf on $\hat \Gamma$ that minimizes the expected overlap of two independent randomly chosen paths in $\hat{\Gamma}$:
 \begin{align*}
 \mathbb{E}_{\hat{\mu}} |\gamma \cap \gamma'| &=  \mathbb{E}_{\hat{\mu}} \left[ \sum_{e\in E} \mathbbm{1}_{\{e \in \gamma \cap \gamma'\} } \right] \\
 &= \sum_{e \in E} \mathbb{P}_{\hat{\mu}} \left[ e \in \gamma \cap \gamma' \right] \\
 &=  \left( \frac{1}{2n} \right)^2 \cdot 4n^2L \\
 &= L. 
 \end{align*}
By (\ref{eq:prob-dual}), we have
$$L = \Mod_2(\Gamma)^{-1} = \min_{\mu \in \mathcal{P}(\Gamma)} \mathbb{E}_\mu | \gamma \cap \gamma' |.$$
Therefore, $\hat{\mu}$ is a pmf on $\hat{\Gamma}$ that minimizes expected overlap.
Hence, $\hat \mu$ is an optimal pmf on $\hat \Gamma$.
Note that it may not be unique, since we are only guaranteed a unique optimal pmf for a minimal subfamily, as we saw in Proposition \ref{prop:pmf-minimal}

To highlight the fact that the limit in Theorem \ref{limitpmf} depends on the correct choice of subfamily, 
we mention that the limit of the pmf $\hat \mu$ as $n \to \infty$ yields 1-dim reflected Brownian motion trajectories started at a uniform point on the left side of $R$.

\begin{claim}\label{cl:edge-prob-rrw}
  $\mathbb{P}_{\hat{\mu}}[e \in \gamma] = \frac{1}{2n}$ for any edge $e$.
  \end{claim}
\begin{proof}
This claim can be shown by induction.  In particular, we can subdivide the edges into $2nL$ distinct sets of $2n$ edges
$$E_k = \{ \text{ all the edges that could be step $k$ in a SRW trajectory } \},$$
and we do induction on $k$ to show that $\mathbb{P}_{\hat{\mu}}[e \in \gamma]$ is the same for all edges $e$ in $E_k$.
The base case follows by definition, since each of the initial edges has equal probability.  
We will show the basic idea behind the induction proof by explaining the reasoning for $E_2$ and $E_3$.
Consider the vertices that edges from $E_1$ lead into (we consider the edges oriented so that the SRW trajectories flow from the left side to the right side of $R$).  
Each of these vertices has exactly two edges leading into it, meaning that our SRW induces a uniform distribution on these vertices.  
Since each of these vertices has two edges leading out, and we will choose between them with equal probability, we have uniform probability for all edges in $E_2$.  
Now consider the vertices that the edges from $E_2$ lead into.  There are two special vertices, one on the top boundary and one on the bottom boundary, that only have one edge leading into them and one edge leading out.  All the rest have two edges leading in and two edges leading out.  While we no longer have a uniform distribution induced on these vertices, we still obtain the uniform distribution on $E_3$.  (Reason: We have the same chance of arriving at any internal vertex, and we have exactly half this probability of arriving at a boundary vertex.  From the internal vertex, we have two edges to choose from with equal probability, while at the boundary vertex we must take the only available edge.) 
\end{proof}
Note that the paths in $\Ga^z$ are non-crossing, while in $\Ga^r$ and $\hat{\Ga}$ they do cross.
In Section \ref{non-crossing} we will give an algorithm that, in the case of plane graphs, always produces the unique non-crossing minimal subfamily.

\subsection{Orthodiagonal maps} \label{orthodiagIntro}

A finite {\it orthodiagonal map} is a finite connected plane graph that satisfies the following:
\begin{itemize}
\item Each edge is a straight line segment.
\item  Each inner face is a quadrilateral with orthogonal diagonals.
\item  The boundary of the outer face is a simple closed curve. 
\end{itemize}
See the top left image in Figure \ref{orthomap} for an illustration.

We will briefly describe some of the structure associated with an orthodiagonal map $G = (V, E)$.
Since $G$ is a bipartite graph, we obtain a bipartition of the vertices $V = V^\bullet \sqcup V^\circ$.
Each inner face $Q$ is determined by the four vertices on its boundary, which must alternate between   $V^\bullet $ and  $V^\circ$.
We will write  $Q=[v_1, w_1, v_2, w_2]$ with the convention that  $v_1, v_2 \in V^\bullet$, $w_1, w_2 \in V^\circ$, and the vertices are ordered counterclockwise around $Q$.
We can create two additional graphs associated with $G$, the {\it primal graph} $G^\bullet = (V^\bullet, E^\bullet)$ and the {\it dual graph} $G^\circ = (V^\circ, E^\circ)$.
There is one primal edge and one dual edge  for each
 inner face  of $G$.  In particular, $e_Q^\bullet \in E^\bullet$ is an edge between $v_1$ and $v_2$ in $Q$; it is defined to be the line segment $v_1v_2$, if this segment is contained in $Q$, and otherwise it is the union of the line segments  $v_1p$ and $pv_2$, where $p=(w_1+w_2)/2$ is the midpoint of $w_1w_2$.
The edge  $e_Q^\circ \in E^\circ$ between $w_1$ and $w_2$ in $Q$ is defined similarly.
 For an edge $e \in E^\bullet$, the {\it dual edge to} $e$ is the unique edge $e^\circ \in E^\circ$ so that there is an inner face $Q$ with $e=e_Q^\bullet$ and $e^\circ = e_Q^\circ$.  
The {\it primal network} $(G^\bullet, \sigma)$ and the {\it dual network} $(G^\circ, \sigma)$ have the following canonical edge-weights:
\begin{equation}\label{eq:canonical-weights}
\sigma(e_Q^\bullet) = \frac{|w_1w_2|}{|v_1v_2|} \;\;\; \text{ and } \;\;\; \sigma(e_Q^\circ) = \frac{|v_1v_2|}{|w_1w_2|},
\end{equation}
where the notation $|ab|$ is the Euclidean distance between $a$ and $b$.  
Note that this distance is not necessarily the length of the edge.

One way to interpret these weights is through a finite-volume approximation of the Laplace equation.  Let $v_1$ be an interior vertex in the primal graph, and consider the collection of quadrilaterals containing this vertex.  The bottom right picture in Figure~\ref{orthomap} shows an example of such a collection of quadrilaterals (light edges in the figure).  The edges of the dual graph surround a polygonal region $D$ (thick dashed lines in the figure).  If $h$ is the harmonic function for the continuum problem, then the divergence theorem implies that
\begin{equation*}
0 = \int_D\Delta h\;dA = \int_{\partial D}\frac{\partial h}{\partial n}\;dS.
\end{equation*}
The integral on the right is a sum of integrals along the edges of $D$.  Let $w_1w_2$ be some such edge, as indicated in the figure.  There is a corresponding edge $v_1v_2$ in the primal graph.  Since these edges are orthogonal, $(h(v_2)-h(v_1))/|v_1v_2|$ provides an approximation of the normal derivative along $w_1w_2$.  Thus, the contribution of $w_1w_2$ to the boundary integral can be approximated as
\begin{equation*}
\int_{w_1w_2}\frac{\partial h}{\partial n}\;dS \approx \frac{|w_1w_2|}{|v_1v_2|}(h(v_2)-h(v_1)).
\end{equation*}
Performing a similar approximation on the remaining edges of $D$ yields the approximation
\begin{equation*}
0 \approx \sum_{u:v_1u\in E}\frac{|w_1w_2|}{|v_1u|}(h(u)-h(v_1)).
\end{equation*}
Thus, when the continuous harmonic function is restricted to the graph, it approximately satisfies the discrete harmonic equation (\ref{eq:si-harmonic}) at each interior vertex, with the edge conductances defined in (\ref{eq:canonical-weights}).  Moreover, this approximation improves as the lengths of the graph edges converge to zero.

We list some further notation relating to an orthodiagonal map $G$:
\begin{align*}
\partial G & = \text{ the topological boundary of the outer face of } G \\
{\rm int}(G) &= \text{ the open subset enclosed by } \partial G \\
\hat{G} &= {\rm int}(G) \cup \partial G \\
\tilde{G} &= \text{ the union of the convex hulls of the closures of the inner faces of } G
\end{align*}

In \cite{gurel-jerison-nachmias:advm2020}, Gurel-Gurevich, Jerison, and Nachmias show that one can approximate continuous harmonic functions by discrete harmonic functions on an orthodiagonal map, generalizing the work of \cite{skopenkov:advm2013} and \cite{Werness2015}. 
In particular, they are concerned with the
{\it solution $h_c : \overline{\Omega} \to \mathbb{R}$ to the continuous Dirichlet problem on $\Omega$ with boundary data $ g |_{\partial \Omega}$}; this terminology means that $h_c=g$ on $\partial \Omega$, $h_c$ is harmonic on $\Omega$, and $h_c$ is continuous on $\overline{\Om}$.

\begin{thm}[Theorem 1.1 of \cite{gurel-jerison-nachmias:advm2020}]
Let $\Omega \subset \mathbb{C}$ be a bounded simply connected domain, and let $g: \mathbb{C} \to \mathbb{R}$ be a $C^2$-smooth function.
Given $\e, \delta \in (0, \diam(\Omega))$, 
let $G = ( V^\bullet \sqcup V^\circ, E)$ be a finite orthodiagonal map with maximal edge length at most $\e$
such that the Hausdorff distance between $\partial G$ and $\partial \Omega$ is at most $\delta$.
Let $h_c : \overline{\Omega} \to \mathbb{R}$ be
the solution to the continuous Dirichlet problem on $\Omega$ with boundary data $ g |_{\partial \Omega}$,
and let $h_d : V^\bullet \to \mathbb{R}$ be the solution to the discrete Dirichlet problem on $V^\bullet \setminus \partial G$ with boundary data $g |_{\partial G \cap V^\bullet}$.
Set
$$ C_1 = \sup_{z \in \tilde \Omega} | \nabla g(z)| \;\;\; \text{ and } \;\;\; C_2 = \sup_{z \in \tilde \Omega} || H g(z) ||_2,$$
where $\tilde \Omega = \conv(\overline{\Omega} \cup \hat{G})$.
Then there is a universal constant $C < \infty$ such that for all $z \in V^\bullet \cap \overline{\Omega}$,
$$ |h_d(z) - h_c(z) | \leq \frac{ C \diam(\Omega)(C_1 + C_2\epsilon)}{\log^{1/2}(\diam(\Omega)/ (\delta \vee \e))}.$$ 
\end{thm}

In the above result, the boundary condition involves the vertices of the primal graph that are in the topological boundary of $G$.  
In our work, we need a similar result with different boundary conditions.
This leads us to making the following definition.

\begin{defn}
The graph $G= (V^\bullet \sqcup V^\circ, E)$ 
is a finite {\it orthodiagonal map 
with  boundary arcs $S_1, T_1, S_2, T_2$,} if
$G$ is a finite orthodiagonal map with the following boundary condition:
\begin{itemize}
\item(Boundary condition) The topological boundary of the outer face of $G$ is  subdivided into four nonempty  paths $S_1, T_1, S_2, T_2$, listed counterclockwise, with $S_j \cap T_k$ equal to a single vertex, for $j,k=1,2$.

\end{itemize}
\end{defn}

This boundary condition allows us to think of $\hat{G}$ as a conformal rectangle with ``sides" $S_1, S_2$ and ``top/bottom"  $T_1, T_2$.
For the associated primal and dual graphs, we define the (graph) boundaries as follows:
\begin{equation*}
   \partial V^\bullet := V^\bullet \cap \left( S_1 \cup S_2 \right) \;\; \text{ and } \;\;
   \partial V^\circ := V^\circ \cap \left( T_1 \cup T_2 \right).
\end{equation*} 
Note that $\partial V^\bullet$ and $\partial V^\circ$ do not include all vertices in the topological boundary of $G$.
Further, we define
\begin{equation*}
   {\rm int}(V^\bullet) := V^\bullet \setminus \partial V^\bullet \; \text{ and } \;
     {\rm int}(V^\circ) := V^\circ \setminus \partial V^\circ,
 \end{equation*}
and we point out that these are different from the collections of vertices in the topological interior of $G^\bullet$ and $G^\circ$.
See Figure \ref{orthomap}.

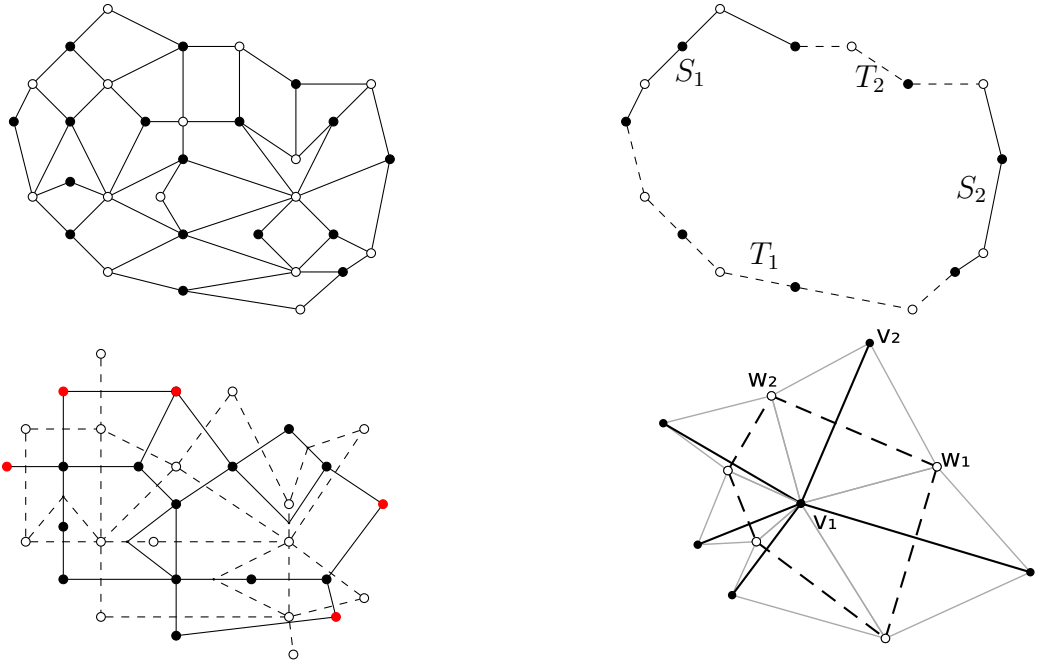
\begin{figure}
\centering
 \begin{minipage}{0.45\textwidth}
\begin{tikzpicture}[scale=0.5, auto, node distance=3cm,  thin]
   \begin{scope}[every node/.style={circle,draw=black,fill=black!100!,font=\sffamily\Large\bfseries}]
   \node (A) [scale=0.3]at (-1.5,0) {};
    \node (B)[scale=0.3] at (0,2) {};
   \node (C)[scale=0.3] at (0,0) {};
   \node (D) [scale=0.3]at (0,-1.6) {};
    \node (E)[scale=0.3] at (0,-3) {};
    \node (F)[scale=0.3] at (2,0) {};
   \node (G)[scale=0.3] at (3,2) {};
   \node (H) [scale=0.3]at (3,-1) {};
    \node (I)[scale=0.3] at (3,-3) {};
    \node (J)[scale=0.3] at (3,-4.5) {};
    \node (K)[scale=0.3] at (4.5,0) {};
   \node (L)[scale=0.3] at (5,-3) {};
   \node (M) [scale=0.3]at (6,1) {};
    \node (N)[scale=0.3] at (7,0) {};
    \node (O)[scale=0.3] at (7,-3) {};
    \node (P)[scale=0.3] at (7.25,-4) {};
      \node (Q)[scale=0.3] at (8.5,-1) {};
       \end{scope}
   \begin{scope}[every node/.style={circle,draw=black,font=\sffamily\Large\bfseries}]
   \node (a) [scale=0.3]at (-1,1) {};
    \node (b)[scale=0.3] at (-1,-2) {};
    \node (c)[scale=0.3] at (1,3) {};
   \node (d) [scale=0.3]at (1,1) {};
    \node (e)[scale=0.3] at (1,-2) {};
    \node (f)[scale=0.3] at (1,-4) {};
    \node (g)[scale=0.3] at (2.4, -2) {};
   \node (h) [scale=0.3]at (3,0) {};
    \node (j)[scale=0.3] at (4.5,2) {};
    \node (k)[scale=0.3] at (6,-1) {};
   \node (l) [scale=0.3]at (6,-2) {};
    \node (m)[scale=0.3] at (6,-4) {};
    \node (o)[scale=0.3] at (8,1) {};
    \node (p)[scale=0.3] at (8,-3.5) {};
         \node (q)[scale=0.3] at (6.117647,-5) {};
       \end{scope}
 \begin{scope}[every edge/.style={draw=black,thin}]
    \draw  (A) edge node{} (a);
     \draw  (A) edge node{} (b);
     \draw  (a) edge node{} (B);
     \draw  (a) edge node{} (C);
         \draw  (b) edge node{} (C);
         \draw  (b) edge node{} (D);
        \draw  (b) edge node{} (E);
        \draw  (B) edge node{} (c);
        \draw  (B) edge node{} (d);
        \draw  (C) edge node{} (d);
        \draw  (C) edge node{} (e);
       \draw  (D) edge node{} (e);
      \draw  (E) edge node{} (e);
      \draw  (E) edge node{} (f);
    \draw  (c) edge node{} (G);
    \draw  (d) edge node{} (G);
    \draw  (d) edge node{} (F);    
     \draw  (e) edge node{} (F);
     \draw  (e) edge node{} (I);
     \draw  (f) edge node{} (I);
     \draw  (f) edge node{} (J);
     \draw  (G) edge node{} (j);
     \draw  (G) edge node{} (h);
     \draw  (F) edge node{} (h); 
     \draw  (H) edge node{} (h);   
   \draw  (H) edge node{} (g);
     \draw  (I) edge node{} (g);
    \draw  (J) edge node{} (m);
     \draw  (e) edge node{} (H);
     \draw  (I) edge node{} (m); 
     \draw  (I) edge node{} (l);
  \draw  (L) edge node{} (l);
     \draw  (L) edge node{} (m);
     \draw  (H) edge node{} (l);
     \draw  (h) edge node{} (K);
     \draw  (K) edge node{} (j);
     \draw  (j) edge node{} (M);
     \draw  (M) edge node{} (k); 
     \draw  (M) edge node{} (o);
    \draw  (k) edge node{} (N);
     \draw  (K) edge node{} (k);
     \draw  (K) edge node{} (l);
     \draw  (N) edge node{} (l);
     \draw  (N) edge node{} (o);
     \draw  (Q) edge node{} (o);
     \draw  (Q) edge node{} (l); 
     \draw  (O) edge node{} (l);
        \draw  (O) edge node{} (m);
          \draw  (P) edge node{} (m);
       \draw  (P) edge node{} (p);
         \draw  (Q) edge node{} (p);
         \draw  (O) edge node{} (p);
          \draw  (J) edge node{} (q);
          \draw  (P) edge node{} (q);
           \end{scope}
 \end{tikzpicture}
\end{minipage}
\begin{minipage}{0.45\textwidth}
\begin{tikzpicture}[scale=0.5, auto, node distance=3cm,  thin]
   \begin{scope}[every node/.style={circle,draw=black,fill=black!100!,font=\sffamily\Large\bfseries}]
   \node (A) [scale=0.3]at (-1.5,0) {};
    \node (B)[scale=0.3] at (0,2) {};
    \node (E)[scale=0.3] at (0,-3) {};
    \node (G)[scale=0.3] at (3,2) {};
    \node (J)[scale=0.3] at (3,-4.4) {};
   \node (M) [scale=0.3]at (6,1) {};
    \node (P)[scale=0.3] at (7.25,-4) {};
      \node (Q)[scale=0.3] at (8.5,-1) {};
       \end{scope}
   \begin{scope}[every node/.style={circle,draw=black,font=\sffamily\Large\bfseries}]
   \node (a) [scale=0.3]at (-1,1) {};
    \node (b)[scale=0.3] at (-1,-2) {};
    \node (c)[scale=0.3] at (1,3) {};
    \node (f)[scale=0.3] at (1,-4) {};
    \node (j)[scale=0.3] at (4.5,2) {};
    \node (o)[scale=0.3] at (8,1) {};
    \node (p)[scale=0.3] at (8,-3.5) {};
             \node (q)[scale=0.3] at (6.117647,-5) {};
       \end{scope}
 \begin{scope}[every edge/.style={draw=black,thin}]
    \draw  (A) edge node{} (a);
     \draw  (a) edge node{} (B);
       \draw  (B) edge node{} (c);
             \draw  (c) edge node{} (G);
     \draw  (Q) edge node{} (o);
       \draw  (P) edge node{} (p);
         \draw  (Q) edge node{} (p);
           \end{scope}
 \begin{scope}[every edge/.style={draw=black,dashed}]    
                   \draw  (A) edge node{} (b);
                   \draw  (b) edge node{} (E);
                  \draw  (G) edge node{} (j);
        \draw  (j) edge node{} (M);
              \draw  (E) edge node{} (f);
     \draw  (f) edge node{} (J);
    \draw  (J) edge node{} (q);
              \draw  (P) edge node{} (q);
                   \draw  (M) edge node{} (o);
 \end{scope} 
 \node[right] at (-0.5,1.3) {$S_1$};
\node[above] at (2.2,-4.2) {$T_1$};
\node[left] at (8.4,-1.8) {$S_2$};
\node[above] at (5,0.5) {$T_2$};      
\end{tikzpicture}
\end{minipage}
\begin{minipage}{0.45\textwidth}
\begin{tikzpicture}[scale=0.5, auto, node distance=3cm,  thin]
  \begin{scope}[every node/.style={circle,draw=black,fill=black!100!,font=\sffamily\Large\bfseries}]
   \node (C)[scale=0.3] at (0,0) {};
   \node (D) [scale=0.3]at (0,-1.6) {};
       \node (E)[scale=0.3] at (0,-3) {};
    \node (F)[scale=0.3] at (2,0) {};
         \node (G)[scale=0.3] at (3,2) {};
   \node (H) [scale=0.3]at (3,-1) {};
    \node (I)[scale=0.3] at (3,-3) {};
    \node (J)[scale=0.3] at (3,-4.5) {};
    \node (K)[scale=0.3] at (4.5,0) {};
   \node (L)[scale=0.3] at (5,-3) {};
             \node (M) [scale=0.3]at (6,1) {};
    \node (N)[scale=0.3] at (7,0) {};
    \node (O)[scale=0.3] at (7,-3) {};
        \node (egg)[scale=0.05] at (1.7,-2) {};
        \node (lkk)[scale=0.05] at (6,-1.5) {};
       \end{scope}
   \begin{scope}[every node/.style={circle,draw=red,fill=red!100!,font=\sffamily\Large\bfseries}]
       \node (A) [scale=0.3]at (-1.5,0) {};
       \node (B)[scale=0.3] at (0,2) {};
          \node (G)[scale=0.3] at (3,2) {};
           \node (P)[scale=0.3] at (7.25,-4) {};
      \node (Q)[scale=0.3] at (8.5,-1) {};
   \end{scope}
  \begin{scope}[every node/.style={circle,draw=black,font=\sffamily\Large\bfseries}]
   \node (a) [scale=0.3]at (-1,1) {};
    \node (b)[scale=0.3] at (-1,-2) {};
    \node (c)[scale=0.3] at (1,3) {};
   \node (d) [scale=0.3]at (1,1) {};
    \node (e)[scale=0.3] at (1,-2) {};
    \node (f)[scale=0.3] at (1,-4) {};
    \node (g)[scale=0.3] at (2.4, -2) {};
   \node (h) [scale=0.3]at (3,0) {};
    \node (j)[scale=0.3] at (4.5,2) {};
    \node (k)[scale=0.3] at (6,-1) {};
   \node (l) [scale=0.3]at (6,-2) {};
    \node (m)[scale=0.3] at (6,-4) {};
    \node (o)[scale=0.3] at (8,1) {};
    \node (p)[scale=0.3] at (8,-3.5) {};
             \node (q)[scale=0.3] at (6.117647,-5) {};
     \node (ok)[scale=0.05] at (6.5,0.5) {};
          \node (ILL)[scale=0.05] at (4,-3) {};
               \node (DC)[scale=0.05] at (0,-0.8) {};
       \end{scope}
 \begin{scope}[every edge/.style={draw=black,thin}]
    \draw  (A) edge node{} (C);
     \draw  (C) edge node{} (B);
     \draw  (B) edge node{} (G);
     \draw  (G) edge node{} (F);
         \draw  (G) edge node{} (K);
         \draw  (K) edge node{} (M);
        \draw  (M) edge node{} (N);
        \draw  (N) edge node{} (Q);
        \draw  (Q) edge node{} (O);
        \draw  (O) edge node{} (P);
        \draw  (C) edge node{} (D);
       \draw  (C) edge node{} (F);
      \draw  (E) edge node{} (D);
      \draw  (E) edge node{} (I);
    \draw  (F) edge node{} (H);
    \draw  (H) edge node{} (K);
    \draw  (H) edge node{} (I);    
     \draw  (I) edge node{} (J);
     \draw  (L) edge node{} (O);
     \draw  (L) edge node{} (I);
       \draw  (J) edge node{} (P);
    \draw  (egg) edge node{} (H);
     \draw  (egg) edge node{} (I);
     \draw  (lkk) edge node{} (K);   
   \draw  (lkk) edge node{} (N);
           \end{scope}
  \begin{scope}[every edge/.style={draw=black,dashed}]          
     \draw  (a) edge node{} (b);
     \draw  (a) edge node{} (d);
     \draw  (c) edge node{} (d);
    \draw  (d) edge node{} (e);
     \draw  (e) edge node{} (f);
     \draw  (b) edge node{} (e); 
     \draw  (e) edge node{} (g);
   \draw  (g) edge node{} (l);
    \draw  (f) edge node{} (m);
     \draw  (m) edge node{} (l);
     \draw  (m) edge node{} (p);
     \draw  (p) edge node{} (l);
     \draw  (l) edge node{} (o); 
     \draw  (l) edge node{} (k);
    \draw  (k) edge node{} (j);
     \draw  (j) edge node{} (h);
     \draw  (h) edge node{} (d);
     \draw  (h) edge node{} (e);
     \draw  (h) edge node{} (l);
          \draw  (m) edge node{} (q);
     \draw  (o) edge node{} (ok);
     \draw  (k) edge node{} (ok); 
     \draw  (m) edge node{} (ILL);
        \draw  (l) edge node{} (ILL);
          \draw  (b) edge node{} (DC);
      \draw  (e) edge node{} (DC);
           \end{scope}
 \end{tikzpicture}
 \end{minipage}
\begin{minipage}{0.45\textwidth}
\includegraphics[width=0.8\textwidth, trim = 2cm 0 0 0, clip]{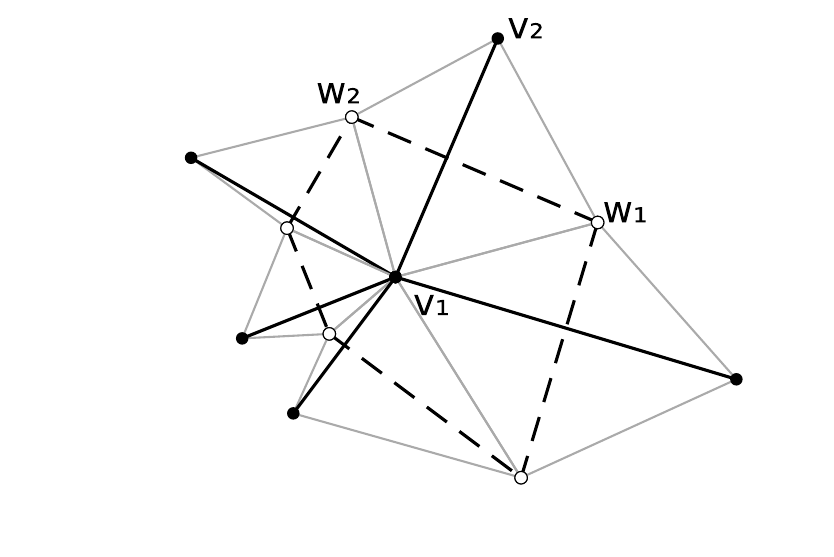}
 \end{minipage}
\caption{Top left: An example of an orthodiagonal map with the bipartation of the vertices. 
Top right: The topological boundary of the map decomposed into four paths: $S_1, S_2$ (solid) and $T_1,T_2$ (dashed).
Bottom left: The primal graph $G^\bullet$ (black/red vertices and solid edges) with $\partial V^\bullet$ shown in red and  the dual graph $G^\circ$ (white vertices and dashed edges).
Bottom right: Finite volume interpretation of the edge weights assigned to an orthodiagonal map.} \label{orthomap}
\end{figure}

\begin{defn}\label{def:smooth-quad}
Let $\Omega$ be a simply connected Jordan domain in the plane. We say $\Omega$ is  an {\it analytic quadrilateral}, if $\partial\Om$  can be decomposed into four closed analytic arcs $A, \tau_1, B, \tau_2$, listed in counter-clockwise order, which do not form a cusp at their intersection points.
Recall that
a {\it closed analytic arc} is the image of the interval $[0,1] \times \{0\}$ under a conformal map defined in an open set containing $[0,1] \times \{0\}$.  
Moreover, if $F : \mathbb{D}:=\{|z|<1\} \to \Omega$ is a conformal map, it will extend as a homeomorphism $F : \overline{\mathbb{D}} \to \overline{\Omega}$, by Carath\'eodory's Theorem.
Then, $\partial \Omega$ has a {\it cusp} at $F(e^{it})$, if
$$\lim_{s \to t^+} \arg [F(e^{is}) - F(e^{it})] - \lim_{s \to t^-} \arg [F(e^{is}) - F(e^{it})] = 0 \text{ mod}(2\pi).$$
\end{defn}
We wish to approximate an analytic quadrilateral with an orthodiagonal map.  To define this, we must specify how we will measure distances.
For $U \subset \mathbb{C}$, the notation dist$_{U}(x, y)$ means the minimum length of all paths between $x$ and $y$ that remain in $U$.  
The notation dist$_U(x, P) = \inf_{y \in P} \dist_U(x,y)$.
When $U=\mathbb{C}$, we omit the subscript and just write dist.
Finally, the notation dist$_{\hat{G}}(S_1, A) < \delta$ means that  dist$_{\hat{G}}(x, A) <\delta$ for all $x \in S_1$ 
       and dist$_{\hat{G}}(x, S_1) <\delta$ for all $x \in A$.
(i.e. for all $x \in S_1$, there is a point $y \in A$ with dist$_{\hat G}(x, y) < \delta$ and vice versa).  

\begin{defn}\label{def:approx-smooth-quad}
Let $G$ be a finite orthodiagonal map with boundary arcs $S_1, T_1, S_2, T_2$, and let $\Omega$ be an analytic quadrilateral with boundary arcs $A, \tau_1, B, \tau_2$.
We say that  {\it $(G;S_1, T_1, S_2, T_2)$   approximates  $ (\Omega;A, \tau_1, B, \tau_2)$ within distance $\delta$}, when
$\overline{\Omega} \subset \hat G$
and $\dist_{\hat G}(S_1, A)$, $\dist_{\hat G}(S_2, B)$, and $\dist_{\hat G}(T_k, \tau_k)$, for $k=1,2$, are all less than $\delta$.
\end{defn}

\subsection{Linearly connected  and John domains} \label{LinearlyConnected}

A simply connected domain $\Omega$ is {\it linearly connected}  if there exists $K>0$ so that any two points $z, w \in \Omega$ can be connected by a curve $\gamma \subset \Omega$ with diam$(\gamma) \leq K |z-w|$. 
On the other hand, a bounded simply connected domain $\Omega$ is a {\it John domain}  if there exists $J>0$ so that every crosscut $\kappa$ of $\Omega$ satisfies $\diam(H) \leq J \diam(\kappa)$, where $H$ is one of the two components of $\Omega \setminus \kappa$.
The inner domain of a piecewise smooth Jordan curve is linearly connected if and only if it has no inward-pointing cusps, and it is a John domain if and only if there are no outward-pointing cusps.
(See Chapter 5 of \cite{pommerenke1992}.)
Thus an analytic quadrilateral $\Omega$, as defined in Definition \ref{def:smooth-quad}, is a linearly connected John domain.  This geometric control gives us the H\"older continuity of some relevant functions.
Note that all the domains we consider are Jordan domains.  Therefore by Carath\'eodory's Theorem, all the associated conformal maps extend as homeomorphisms of the closures, and in particular they are well-defined on the boundary.

\begin{lemma} \label{Holdercontinuity}
Let $\Omega_1, \Omega_2$ be two Jordan domains with $\Omega_1$ linearly connected and $\Omega_2$ a John domain.
Let $\Phi : \Omega_1 \to \Omega_2$ be a conformal map.
Then there exists $a >0$ and $C>0$ so that 
$$ | \Phi(z) - \Phi(w) | \leq C |z-w|^{a} \;\;\; \text{ for all } z, w \in \overline{\Omega}_1.$$

\end{lemma}

\begin{proof}

Let $F_k : \mathbb{D} \to \Omega_k$ be conformal so that $\Phi = F_2 \circ F_1^{-1}$.
Recall that every John domain is a H\"older domain, 
which by definition means that $F_2$ is H\"older continuous on $\overline{\mathbb{D}}$.
Since $\Omega_1$ is a linearly connected domain,
Theorem 5.7 of \cite{pommerenke1992} implies that $F_1$ satisfies a lower H\"older condition with exponent $b<2$ on $\partial \mathbb{D}$, that is,
$$  |F_1(z) - F_1(w) | \geq C |z-w|^{b}$$
for $z,w \in \partial \mathbb{D}$.
This is equivalent to a H\"older condition on $F_1^{-1}$ with exponent $1/b > 1/2$ on $\partial \Omega_1$.
Hence there exists $a>0$ so that 
$$ |\Phi(z) - \Phi(w) | \leq C |z-w|^{a} \,\,\, \text{  for } z,w \in \partial \Omega_1.$$
Then Theorem 1 in \cite{hinkkanen1988} applied to $\Re \Phi$ and $\Im \Phi$ proves the result.

\end{proof}

In our work, we will need to extend some of our functions over a boundary arc.  We accomplish this in the following result.

\begin{lemma}\label{reflection}
Let $\Omega$ be an analytic quadrilateral with boundary arcs $A, \tau_1, B, \tau_2$, as in Definition \ref{def:smooth-quad}.
Let $\psi$ be a conformal map of $\Omega$ onto the interior of the rectangle $R=[0,1] \times [0,m]$ taking the arcs $A$ and $B$ to the vertical sides of the rectangle.
Let  
$$P_\lambda = \{ z \, : \, \dist(z, \overline{\Omega}) \leq \lambda \text{ and } \dist(z, A\cup B) \geq \lambda^{1/4} \}.$$
Then  the map $\psi$ can be extended to be conformal on a domain $\mathcal{O} \supset \Omega$ such that 
 for $\lambda$ small enough (depending on the analytic quadrilateral $\Omega$)  
 we have that $\mathcal{O}$ contains $P_\lambda$ and $\dist(P_\lambda, \partial \mathcal{O}) \geq \frac{5}{16}\l^{1/4}$.  (See the right panel in Figure \ref{SubOrthoPic} for an illustration.)
Further, there exist $a, b, c, C>0$ so that  the extended map $\psi$ satisfies 
\begin{equation}\label{hcont}
c |z-w|^b \le |\psi(z) - \psi(w)| \leq C |z-w|^{a} \,\,\, \text{ for all } z,w \in \overline{\mathcal{O}}.
\end{equation}
\end{lemma}

\begin{proof}
By Schwarz reflection, we may
extend $\psi $ across the arcs $\tau_1$ and $\tau_2$.
More precisely, recall from Definition \ref{def:smooth-quad}, that since  $\tau_k$ is an analytic arc, it is the image of the interval $[0,1] \times \{0\}$ under a conformal map $F_k$ which is defined in a domain $\mathcal{D}_k$ containing $[0,1] \times \{0\}$.  
Note that $\phi_k= \psi\circ F_k$ is defined in the domain $\mathcal{D}_k\cap F_k^{-1}(\Omega)$ and extends continuously to $[0,1] \times \{0\}$ mapping it to either the top or the bottom boundary arc of $R$ (i.e.~$[0,1] \times \{0\}$ or $[0,1] \times \{m\}$).
Hence by Schwarz reflection, we can extend $\phi_k$ over $(0,1) \times \{0\}$, and from this we obtain an extension of $\psi$, which is equal to $\phi_k \circ F_k^{-1}$ near $\tau_k$.  
In particular, we have extended 
 $\psi$ to be conformal in an open set containing $\Omega$ and the relative interior of $\tau_k$.  
 
 In a moment we will specify the domain $\mathcal{O}$ and show that it satisfies the desired list of properties, but 
 before getting to the details, we want to mention the key idea.  The proof relies on the fact that  the arcs $A, B$ meet $\tau_1, \tau_2$ non-tangentially.  This means that under the reflection, we can extend to a domain whose boundary is not tangent with $\tau_1, \tau_2$.

 To get started, we take a look at the ``corners" of $\Omega$.
  Consider the corner point $w=A\cap \tau_1$.
 By Definition \ref{def:smooth-quad} the arcs $A$ and $\tau_1$ are not tangent at $w$.
Then,  the interior angle $\theta$ of $\Omega$ at $w$ is in $(0, 2\pi)$. 
By conformality, $\theta$ is also the angle from $[0,1] \times \{0\}$ to $F_1^{-1}(A)$ at $F_1^{-1}(w) = 0$.
First, assume that $\theta < \pi$.  Then, under the Schwarz reflection, $\phi_1$ extends to a domain that has interior angle $2 \theta$ at 0.
Under conformality again, $\psi$ extends to a domain with interior angle $2\theta$ at $w$.
If $\theta \in [\pi, 2\pi)$, then $\phi_1$ extends to a domain that has interior angle $2 \pi$ at 0
and hence $\psi$ extends to a domain with interior angle $2\pi$ at w.

The desired domain $\mathcal{O}$ will be the interior domain of a Jordan curve, and we will construct this curve in stages.  First, we note that the arcs $A$ and $B$ will be part of the boundary.  It remains to specify two connecting boundary arcs, $\eta_1, \eta_2$.
Next we will specify $\eta_1$ locally near the corner $w$. By rotating if necessary, we may assume $\arg(z-w) \to 0$ as $z$ approaches $w$ along $\tau_1$.  
Let $D_{r}(w)$ be the disc of radius $r$ centered at $w$.
Then for $\rho >0$, there exists $r>0$ so that $A \cap D_r(w)$ is in the sector 
$\{ w+se^{it} \, : \, s \geq 0, \, t \in  (\theta - \rho, \theta + \rho) \}$
and $\tau_1 \cap D_r(w)$ is in the sector
$\{ w+se^{it} \, : \, s \geq 0, \, t \in  (- \rho, \rho) \}$.
When $\theta < \pi$, we define $\eta_1$ in $D_r(w)$ to be 
the segment $\{se^{i (-\theta+ \rho)} \, : s \in [0,r] \}$.  
By taking $\rho$ small enough, we have $\dist(z, \eta_1 \cap D_r(w)) \geq \frac{1}{2} \dist(z, A)$ for $z \in \tau_1 \cap D_r(w) $.
If $\theta \in [\pi, 2\pi)$, we take $\eta_1$ in $D_r(w)$ to be the segment $\{se^{i \frac{1}{2}(\theta-2\pi)} \, : s \in [0,r] \}$. 
As in the other case, we may again take $\rho$ small enough so that $\dist(z, \eta_1 \cap D_r(w)) \geq \frac{1}{2} \dist(z, A)$ for $z \in \tau_1 \cap D_r(w)$.

To finish defining $\partial \mathcal{O}$, we follow the same procedure at each of the four corners.   
To complete $\eta_k$, we connect the arc at the corner $A \cap \tau_k$ to the arc at the corner $B \cap \tau_k$ with an analytic arc that lies  in the extended domain of $\psi$ outside of $\overline{\Omega}$, making sure that there are no cusps formed at the connection points.  In this way, we have created the domain $\mathcal{O}$ containing $\Omega$ so that its boundary  is a piecewise smooth Jordan curve with no cusps and $\psi$ is defined on $\mathcal{O}$.


Let $w_1, w_2, w_3, w_4$ be the four corner points of $\Omega$ and let $r_k$ be the radius identified in the construction of $\partial \mathcal{O}$ at $w_k$.
Set  $D= \cup_{k=1}^4 D_{r_k}(w_k) $.
Then 
by taking $\lambda$ small enough, 
we may assume that 
\begin{equation} \label{gammaAssumption}
 \text{dist} \left(  x, \, \partial \mathcal{O} \right) \geq\lambda^{1/4} \;\; \text{for all } x \in \left(\tau_1 \cup \tau_2 \right) \setminus D,
\end{equation}
and
\begin{equation} \label{gammaAssumption2}
 \text{dist} \left(  x, \, \partial \mathcal{O} \setminus D \right) \geq\lambda^{1/4} \;\; \text{for all } x \in \left(\tau_1 \cup \tau_2 \right) \cap D.
\end{equation}
We also assume that $\lambda \leq 1/16$, which implies that $\lambda \leq \frac{1}{8} \lambda^{1/4}$.

We wish to show that for all $ p \in P_\lambda$
\begin{equation} \label{dist2bdry}
\text{dist}(p, \partial \mathcal{O}) \geq \frac{5}{16} \lambda^{1/4}.
\end{equation}
Let $p \in  P_\l$.
If the closest point to $p$ in $\partial \mathcal{O}$ is in $A \cup B$, then 
\eqref{dist2bdry} holds since dist$(p, A \cup B) \geq \lambda^{1/4}$.
So we will assume that the closest point to $p$ in $\partial \mathcal{O}$ is not contained in $A \cup B$.
First consider the case that $p \in \overline{\Omega}$.
The line segment connecting $p$ to a closest point in $\partial \mathcal{O}$ must cross out of $\Omega$ at some point $z_0 \in  \tau_1 \cup \tau_2$.  
We have two possible cases: 
    (1) $z_0 \in \left( \tau_1 \cup \tau_2 \right)  \setminus D$
    or    (2) $z_0 \in \left( \tau_1 \cup \tau_2 \right) \cap D $.
In case (1), we use \eqref{gammaAssumption} to obtain
 $$\text{dist}(p, \partial\mathcal{O}) \geq  \text{dist}(z_0, \partial\mathcal{O}) \geq \lambda^{1/4}.$$
In case (2), we may assume that the closest point to $z_0$ in $\partial \mathcal{O}$ is in $D$ (for otherwise, we are done by \eqref{gammaAssumption2}.)  Recall that by our construction $\dist(z_0, \partial \mathcal{O} \cap D) \geq \frac{1}{2} \dist(z_0, A \cup B)$.
Therefore
\begin{align*}
\text{dist}(p, \partial\mathcal{O})  &= \dist(p, z_0) +  \text{dist}(z_0, \partial\mathcal{O}) \\
     &\geq \text{dist}(p,z_0) + \frac{1}{2} \text{dist}(z_0, A \cup B)    \\
     &\geq \frac{1}{2} \text{dist}(p, A \cup B)   \\
     &\geq \frac{1}{2} \lambda^{1/4}. 
\end{align*}

    Next we assume that $p \notin \Omega$.  Then there is      
some $z_0 \in \tau_1 \cup \tau_2$ 
so that $\text{dist}(p, z_0)  \leq \l$.
If $z_0 \notin D$ or if $z_0 \in D$ but the closest point to $z_0$ in $\partial \mathcal{O}$ is not in $D$, then by  \eqref{gammaAssumption} or  \eqref{gammaAssumption2}
 $$\text{dist}(p, \partial\mathcal{O}) \geq  \text{dist}(z_0, \partial\mathcal{O})  - \dist(p, z_0) \geq \lambda^{1/4} - \l \geq \frac{7}{8} \lambda^{1/4},$$  
where we used that $\lambda \leq \frac{1}{8} \lambda^{1/4}$. 
Assume $z_0 \in D$ and the closest point to $z_0$ in $\partial \mathcal{O}$ is in $D$.
Then 
$$ \text{dist}(z_0, \partial \mathcal{O}) \geq \frac{1}{2} \text{dist}(z_0, A \cup B)
\geq \frac{1}{2} \left[ \text{dist}(p, A \cup B) - \text{dist}(p, z_0) \right]
\geq \frac{1}{2} \left( \lambda^{1/4} - \l \right) \geq \frac{7}{16} \lambda^{1/4}, $$
and so
$$ \text{dist}(p, \partial \mathcal{O}) \geq \text{dist}(z_0, \partial \mathcal{O}) - \text{dist}(p, z_0) \geq \frac{7}{16} \lambda^{1/4} - \l \geq \frac{5}{16} \lambda^{1/4}.$$
 This establishes \eqref{dist2bdry}.

By our construction above, $\mathcal{O}$ is the inner domain of a piece-wise smooth Jordan curve with no cusps, and so  $\mathcal{O}$ is a linearly connected John domain.  Similarly, the Schwarz reflection construction also gives that $\psi(\mathcal{O})$ is a linearly connected John domain. 
Then \eqref{hcont} follows from Lemma \ref{Holdercontinuity} applied to the extended maps $\psi$ and $\psi^{-1}$.
\end{proof}

\section{Non-crossing minimal subfamilies}\label{non-crossing}

In this section we assume that the finite network $G=(V, E, \sigma)$ is a {\it plane} graph (i.e.~$G$ is embedded in the plane with no edges crossing.)
Further, we assume that there are nonempty sets of  boundary vertices $A$ and $B$  defined as follows:
   $A = S_1 \cap V$ and $B= S_2 \cap V$, 
  where $S_1, S_2$ are two disjoint arcs from a Jordan curve $\mathcal{C}$ that encircles $G$ in the closed outer face of $G$.
We will refer to $G=(V, E, \sigma, A, B)$ as a {\it plane network with boundary}. 
Note that when we specialize to orthodiagonal maps later, we take the primal graph as our network and the boundary of the outer face of the map as the Jordan curve. 

Let $F$ be the conformal map from the interior of $\mathcal{C}$ to a rectangle $(0,L) \times (0,1)$ that takes $S_1$ to the left side $\{0\}\times [0,1]$ and $S_2$ to the right side $\{L\} \times [0,1]$.  
Given a path $\gamma$ from $A$ to $B$, the vertical strip $\{ z \, : \, \text{Re}(z) \in [0, L] \}$ is split by $F(\gamma)$ into two connected components.  Let $U_\gamma$ be the connected component containing $-i$ (i.e. the component ``under" $F(\gamma)$).
Then for a subgraph $H$ of $G$ and a path $\gamma$ in $H$ from $A$ to $B$, 
we say that $\gamma$ is the {\it top path} in $H$ from $A$ to $B$ if
$F(H) \subset \overline{U_\gamma}$. 
We say that two paths from $A$ to $B$ are {\it non-crossing} when one path is the top path in the subgraph induced by these two paths.
A family of paths is {\it non-crossing} when every pair of paths in the family is non-crossing.
We note that two non-crossing paths are allowed to share vertices and edges.
Also we say that an edge $e$ is {\it below} a path $\ga$ if $\ga$ is the top path of the graph induced by $\{e\}\cup\ga$.
\begin{remark}\label{rem:cross}
Two paths $\ga_1$ and $\ga_2$ are crossing if one can find $e_1\in\ga_1\setminus \ga_2$ and $e_2\in\ga_2\setminus\ga_1$, so that $e_1$ is not below $\ga_2$ and $e_2$ is not below $\ga_1$.
\end{remark}

\subsection{Non-crossing minimal subfamily algorithm}\label{ssec:noncross-minalg}
To establish the existence and uniqueness of the non-crossing minimal subfamily in Theorem \ref{existence}, we will use Algorithm \ref{alg:non-cross-min-subfam} below.  We summarize the properties of Algorithm \ref{alg:non-cross-min-subfam} in the following statement.
\begin{prop}\label{prop:non-cross-min-subfam}
 Given a plane network with boundary $G = (V, E, \sigma, A, B)$, the family $\Ga(A,B)$ of paths from $A$ to $B$ in $G$, and the extremal density $\rho^*$ for $\Mod_{2,\si}\Ga(A,B)$,
Algorithm \ref{alg:non-cross-min-subfam} outputs a subfamily $\Ga\subset \Ga(A,B)$ and a set of positive weights $m\in\R^\Ga_{>0}$.
Then, 
\bi
\item[(i)] the sum $M:=\sum_{\ga\in\Ga} m(\ga)$ is equal to  $\Mod_{2,\si}\Gamma(A,B)$;
\item[(ii)] $\mu^*:=m/M$ is a pmf on $\Ga$ that is optimal for $\Mod_{2,\si}\Ga(A,B)$;
\item[(iii)] $\Ga$ is a minimal subfamily of non-crossing paths for $\Ga(A,B)$.
\ei
In particular,  $\mu^*$ is the unique optimal pmf on $\Ga$ that is guaranteed by Proposition \ref{prop:pmf-minimal}.  

\end{prop}

\begin{figure}[H]
\begin{minipage}{0.05\textwidth}
\hspace*{1pt}
\end{minipage}
\begin{minipage}{0.9\textwidth}
\begin{algorithm}[H]
\smallskip

\noindent {\it Input:} 
a plane network with boundary $ G = (V, E, \sigma, A, B)$ 
and the corresponding extremal density  $\rho^*$ for $\Mod_{2,\si}\Ga(A,B)$, where $\Ga(A,B)$ is the family of paths from $A$ to $B$ in $G$.

\smallskip

\noindent {\it Initialization:}  set $r(e) = \sigma(e) \rho^*(e)$ for each edge $e \in E$.

\smallskip

\noindent {\it Iteration:}
\begin{enumerate}
\item Remove any edges $e$ for which $r(e) = 0$ from $G$ (call such edges {\it zero edges}).
\item If there is one, consider the top path $\gamma$ in $G$ from $A$ to $B$ and add it to a growing family $\Gamma$; if not, terminate.
\item For the top path $\gamma$ identified in step 2, set $m(\gamma) =\min_{e \in \gamma} r(e)$.
\item For each edge $e$ in the top path $\gamma$ identified in step 2, update $r(e)$ to be  $r(e) - m(\gamma)$.
\item Repeat.
\end{enumerate}

\noindent {\it Output:} path family $\Gamma$ and the positive weights $m\in\R_{>0}^\Ga$.

\caption{Non-crossing Minimal Subfamily Algorithm}
  \label{alg:non-cross-min-subfam}
\end{algorithm}
\end{minipage}
\end{figure}

We postpone the proofs of Proposition \ref{prop:non-cross-min-subfam} and Theorem \ref{existence} to the end of the section, in order to first establish some preliminary lemmas.

To get started, we show that we can direct the non-zero edges of $G$ to obtain a directed acyclic graph (once we remove any edges with $\rho^*(e) = 0$).
Informally, the edges are directed by the current flow defined by the solution of the discrete Dirichlet problem.
\begin{lemma} \label{dag}
Let $G=(V, E, \sigma, A, B)$ be a plane network with boundary.
Let $\rho^*$ be the unique extremal density for the family of paths from $A$ to $B$ in $G$, 
and let  $E_0 = \{ e\in E \, : \, \rho^*(e) = 0 \}$. 
Then there exists an orientation of all edges in $ E \setminus E_0$ so that the following hold:
\begin{enumerate}
\item[(i)] At any vertex $v$  not in $A\cup B$, the sum of $\sigma(e)\rho^*(e)$ over the edges $e$ leading into $v$ equals the sum of $\sigma(e)\rho^*(e)$ over the edges $e$ leading out of $v$.

\item[(ii)] All edges in $E \setminus E_0$ with exactly one vertex in $A$ (resp. B) are oriented away from $A$ (resp. towards $B$).

\item[(iii)] Let $\gamma$ be a path in  $G\setminus E_0$ from $A$ to $B$.
Then $\gamma$ is a directed path from $A$ to $B$ if and only if $\gamma$ has $\rho^*$-length 1, i.e. $\ell_{\rho^*}(\gamma) = 1$.

\item[(iv)] $G \setminus E_0$ is a directed acyclic graph.
\end{enumerate}
\end{lemma}

\begin{proof}
Let  $h$ be the solution to the discrete Dirichlet problem on $V \setminus (A \cup B)$ with $h |_A = 0$ and $h |_B = 1$,
and recall, from (\ref{rhofromharmonic}), that $\rho^*(vw) = |h(w) - h(v)|$.
For each edge $e$ incident to vertices $v,w$ with $\rho^*(e) \neq 0$, we orient $e$ from $v$ to $w$ when $h(w) -h(v) >0$.

Let $f$ be  the corresponding current flow: $ f(vw) = \sigma(vw) \left[ h(w) - h(v) \right]$ for  $vw \in \vv{E},$ which we extend to $V\times V$, by setting $f$ to be zero on pairs of nodes not connected by an edge.  
For $v \in V \setminus (A \cup B)$,  let $E_v^{\text{in}}$ be the edges directed into $v$ and $E_v^{\text{out}}$ be the edges directed out of $v$.
Then,  by harmonicity of $h$ at $v$,
\begin{align*}
0 & = \sum_{w: \, vw \in E} \si(vw)[h(w)-h(v)]= \sum_{w: \, vw \in E} f(vw)\\& = \sum_{e \in E_v^\text{out}}\sigma(e) \rho^*(e) - \sum_{e \in E_v^\text{in}}\sigma(e) \rho^*(e),
\end{align*}
which proves (i).

Statement (ii) follows from the discrete maximum principle, which implies that $0\le h\le 1$.
To show (iii), let $\gamma$ be a path from $A$ to $B$ in $G \setminus E_0$ given by vertices $v_0, v_1, \cdots, v_k$.  Then
$$  \ell_{\rho^*}(\gamma) = 
 \sum_{j =1}^k \left| h(v_j) - h(v_{j-1}) \right|  \geq 
\sum_{j =1}^k (h(v_j) - h(v_{j-1})) = 1.$$
There is equality in the above equation if and only if $\gamma$ is a directed path from $A$ to $B$.

To show that $G \setminus E_0$ has no cycles (respecting the edge directions), we will
assume for the sake of contradiction that there is a cycle of directed edges $e_1, e_2, \cdots , e_k$ with each $e_j$ directed from $v_{j-1}$ to $v_j$. 
Then 
$$0 < \sum_{j=1}^k \rho^*(e_j) =  \sum_{j=1}^k (h(v_j) -h(v_{j-1})) = h(v_k) -h(v_0) = 0,$$
 since $v_k = v_0$. This is a contradiction.
\end{proof}

\begin{lemma}\label{toppath}
Let $G=(V, E, \sigma, A, B)$ be a plane network with boundary, and let $E\setminus E_0$ be oriented as in Lemma \ref{dag}.
Assume that $H$ is a connected subgraph of $G\setminus E_0$ so that all the source vertices of $H$ are in $A$ and all the sink vertices of $H$ are in $B$.
Then the following hold:
\begin{enumerate}
\item[(i)] The top path of $H$ is a directed path from $A$ to $B$.
\item[(ii)] Let $\tilde \Gamma$ be a family of non-crossing paths in $H$ that does not contain the top path of $H$.  Then there is an edge of the top path that is not contained in any path in $\tilde \Gamma$.
\end{enumerate}
\end{lemma}

\begin{proof}
Let $\gamma$ be the top path of $H$.
Suppose $\gamma$ is not a directed path from $A$ to $B$, and let $e$ be an edge in $\gamma$ that is oriented ``to the left".  
Since all the sources (resp. sinks) are in $A$ (resp. $B$), there must be some directed path in $H$ from $A$ to $B$ that contains $e$.  However, the planarity of $H$ and the fact that $\gamma$ is the top path imply that $H$ must contain a directed cycle.  
This contradicts the fact that $H$ is a directed acyclic graph from Lemma \ref{dag} (iv). 
Therefore, $\gamma$ is a directed path from $A$ to $B$, which establishes (i).

We will prove (ii) by contradiction.  Assume that every edge of $\gamma$ is contained in some path in $\tilde \Gamma$.
Let $\tilde \gamma$ be a path from $\tilde \Gamma$ that contains the longest initial segment in common with $\gamma$.  
Let $e_1$ be the edge in $\gamma$ immediately after this initial segment (such an edge must exist since $\tilde \gamma \neq \gamma$).  
Then $e_1$ is not below $\tilde \gamma$.
Now there must be some other path $\tilde \gamma'$ in $\tilde \Gamma$ that contains  $e_1$, and by the choice of $\tilde \gamma$, the initial segment of $\tilde \gamma$ has an edge $e_2$ not in $\tilde \gamma'$.  Then $e_2$ is not below $\tilde \gamma'$.  
By Remark \ref{rem:cross}, this implies that $\tilde \Gamma$ contains two paths that cross, which is not possible.  

\end{proof}

\begin{lemma}\label{components}
Let $G=(V, E, \sigma, A, B)$ be a plane network with boundary, and let $E\setminus E_0$ be oriented as in Lemma \ref{dag}.
After each completion of step 1 in the Non-crossing Minimal Subfamily Algorithm \ref{alg:non-cross-min-subfam},
 the connected components of the resulting graph are  isolated vertices or subgraphs with all their source vertices in $A$ and all their sink vertices  in $B$.
\end{lemma}

\begin{proof}
The first time we complete step 1, we remove all edges with $\rho^*(e) =0$.
Let $H$ be one of the resulting connected components with at least one edge.
Since $H$ is a directed acyclic graph by Lemma \ref{dag}, 
$H$ must contain at least one source vertex  and at least one sink vertex.\footnote{The fact that a connected directed acyclic graph with at least one edge contains a source and sink is straightforward to prove by contradiction.  For instance, suppose there is no sink.  Then every vertex has at least one edge oriented away from the vertex.  Since the graph is acyclic, we can then create a directed path that contains every vertex.  However, we obtain a contradiction since the last vertex in the path has an edge oriented outward, allowing us to create a cycle.  The argument showing there is a source vertex is similar.}
Properties (i) and (ii) of Lemma \ref{dag} imply that
all the source vertices of $H$ must be in $A$ and all the sink vertices of $H$ must be in $B$.

Now we proceed by induction. Suppose that the statement is true for the first $k$ times we complete step 1, and consider a connected component  $H$ obtained after the $(k+1)$th iteration of step 1.
We assume that $H$ has at least one edge, and we will use contradiction to show that all the source vertices are in $A$ and all the sink vertices are  in $B$.
For the sake of simplicity, assume that $H$ contains a source vertex $v$ that is not in $A$ (as a similar argument applies in the other case).
In completing the first $k$ iterations of the algorithm, we have identified the top paths $\gamma_1, \cdots, \gamma_{k}$.
By the induction hypothesis and Lemma \ref{toppath}\,(i), these paths are all directed paths from $A$ to $B$.
Let $\Gamma_v \subset \{ \gamma_1, \cdots, \gamma_k \}$ be the collection of these paths that pass through $v$.
Let $E_v^{\text{in}}$ be the edges directed into $v$ and $E_v^{\text{out}}$ be the edges directed out of $v$.
Then, since all of the edges directed towards $v$ have been removed by the end of the $(k+1)$th iteration of step 1, 
$$\sum_{\gamma \in \Gamma_v} m(\gamma) = \sum_{e \in E_v^\text{in}}\sigma(e) \rho^*(e). 
$$
Namely, an edge $e$ gets removed once the sum of the weights of all the paths in $\Ga_v$ that are using $e$ has exhausted the initial quantity $\si(e)\rho^*(e)$.
Moreover, by Lemma \ref{dag}\,(i),
$$ \sum_{e \in E_v^\text{in}}\sigma(e) \rho^*(e) = \sum_{e \in E_v^\text{out}}\sigma(e) \rho^*(e).$$
However, this means that we have also removed all the edges directed out from $v$, contradicting the fact that $H$ is a connected component larger than an isolated vertex.
\end{proof}

\begin{proof}[Proof of Proposition \ref{prop:non-cross-min-subfam}]
We begin by proving (i).
Let $\Ga\subset\Ga(A,B)$ and $m\in\R_{>0}^{\Ga}$ be, respectively, the subfamily and the set of positive weights that are output by the Non-crossing minimal subfamily Algorithm \ref{alg:non-cross-min-subfam}. The algorithm
must terminate, because we remove at least one edge from the graph during every iteration (except possibly the first one). 
When the algorithm terminates,  there are no edges remaining in the graph by Lemma \ref{components}.  Thus, each edge  $e\in E \setminus E_0$ belongs to at least one path in $\Ga$. Furthermore, 
\begin{equation}\label{eq:mrhostar}
\sum_ {\gamma \in \Gamma} m(\gamma)  \mathbbm{1}_{\{e \in \gamma\}}=\sigma(e) \rho^*(e), \qquad\text{for every $e\in E$.}
\end{equation}
Let $E\setminus E_0$ be oriented as in Lemma \ref{dag}. Then, from Lemmas \ref{components} and  \ref{toppath}\,(i), each path in $\Gamma$ is a directed path from $A$ to $B$.
Therefore,  Lemma \ref{dag}\,(iii) implies that the paths in $\Gamma$ all have $\rho^*$-length equal to $1$. In particular, if we multiply Equation  (\ref{eq:mrhostar}) by $\rho^*(e)$ and sum over $E$, we get
\begin{align*}
\Mod_{2,\si}\Ga(A,B) & = \sum_{e\in E} \si(e)\rho^*(e)^2 = \sum_{e\in E}\rho^*(e)\sum_ {\gamma \in \Gamma} m(\gamma)  \mathbbm{1}_{\{e \in \gamma\}}  & \text{(by Equation (\ref{eq:mrhostar}))}\\
& = \sum_ {\gamma \in \Gamma} m(\gamma) \sum_{e\in E}\rho^*(e) \mathbbm{1}_{\{e \in \gamma\}} =  \sum_ {\gamma \in \Gamma} m(\gamma) \ell_{\rho^*}(\ga) & \text{(switching sums)} \\
&=  \sum_ {\gamma \in \Gamma} m(\gamma). & \text{(because $\ell_{\rho^*}(\ga)=1$)}
\end{align*}
This shows that (i) holds, namely 
\begin{equation}\label{eq:sum-m-gamma}
M:=\sum_ {\gamma \in \Gamma} m(\gamma)=\Mod_{2,\si}\Ga(A,B).
\end{equation}
Next, we prove (ii). Define a pmf $\mu^*\in\cP(\Ga)$ by setting
\begin{equation}\label{eq:mustarm}
\mu^*(\gamma) := \frac{m(\gamma)}{\sum_ {\gamma \in \Gamma} m(\gamma)}, \quad \text{ for all }  \gamma \in \Gamma.
\end{equation}
Then, since $\mu^*=m/M$, for every edge $e\in E$, we have
\begin{align*}
\bP_{\mu^*}(e\in\underline{\ga}) & = \sum_{\ga\in\Ga}\mu^*(\ga)\ones_{\{e\in\ga\}}  = \frac{1}{M}\sum_{\ga\in\Ga}m(\ga)\ones_{\{e\in\ga\}} \\
& = \frac{\si(e)\rho^*(e)}{M} & \text{(by Equation \ref{eq:mrhostar})} \\
& = \frac{\si(e)\rho^*(e)}{\Mod_{2,\si}\Ga(A,B)} & \text{(by Equation \ref{eq:sum-m-gamma})} \\
\end{align*}
This means that $\mu^*$ satisfies (\ref{eq:rhomu}). Hence, by Corollary \ref{cor:prob}, $\mu^*$ is optimal for  $\Mod_{2,\si}\Gamma(A,B)$.
This shows that (ii) holds.

Finally, to prove (iii) we must prove that $\Ga$ is a minimal subfamily of $\Ga(A,B)$. First, we show that $\Mod_{2,\si}(\Gamma)=\Mod_{2,\si}(\Ga(A,B))$. 
To that end, we apply Beurling's Criterion (Theorem \ref{BeurlingCriterion}) to show that  $\Gamma$ is a Beurling subfamily of $\Ga(A,B)$. 
We have already seen that, by Lemma \ref{dag}\,(iii), $\ell_{\rho*}(\ga)=1$ for every $\ga\in\Ga$.
So the first condition for a Beurling subfamily is satisfied.
Next, we check the second condition.  Let $h \in \mathbb{R}^E$ with 
\begin{equation}\label{eq:h-non-neg}
\sum_{e \in \gamma} h(e) \geq 0
\end{equation}
for all $\gamma \in \Gamma$.
We apply this assumption and (\ref{eq:mrhostar}) to obtain the needed condition:
\begin{align*}
\sum_{e \in E} h(e) \rho^*(e) \sigma(e)
   &=  \sum_{e \in E} h(e)  \sum_{\gamma \in \Gamma} m(\gamma)  \mathbbm{1}_{\{e \in \gamma\}} & \text{(by Equation (\ref{eq:mrhostar}))} \\
   &=  \sum_{\gamma \in \Gamma}   m(\gamma)  \sum_{e \in \gamma} h(e) \geq 0.& \text{(by Equation (\ref{eq:h-non-neg}))}
\end{align*}
Thus, $\Gamma$ is a Beurling subfamily of $\Ga(A,B)$, which implies that $\Mod_{2,\si}(\Gamma)$ is the same as $\Mod_{2,\si}(\Ga(A,B))$. 

Next, suppose $\Ga$ is not minimal. Then, we can remove a path from $\Ga$ without affecting its modulus. Continue removing paths from $\Ga$ until we reach a strictly smaller family $\tilde{\Ga}$ with the same modulus, that cannot be made any smaller. Then, $\tilde{\Ga}$ also consists of non-crossing paths and it is a minimal subfamily for  $\Ga(A,B)$. By Proposition \ref{prop:pmf-minimal}, there is a unique pmf $\tilde\mu$  supported on $\tilde\Ga$, which is optimal for $\Mod_{2,\si}(\Ga(A,B))$, and is strictly supported on $\tilde\Ga$. Lemma \ref{minimalfamilyuniqueness}  below shows that we then must have  $\Gamma =\tilde \Gamma$,  but this is a contradiction. Therefore, $\Ga$ is also minimal and we are done with (iii).

Finally, this shows that the pmf $\mu^*$ defined in (\ref{eq:mustarm}) must be the optimal pmf guaranteed by Proposition \ref{prop:pmf-minimal} for $\Ga$.
\end{proof}
\begin{lemma} \label{minimalfamilyuniqueness}
Let $G=(V, E, \sigma, A, B)$ be a plane network with boundary and let $\Ga(A,B)$ be the family of all paths from $A$ to $B$ in $G$.
Consider the path family $\Gamma\subset\Ga(A,B)$ and the pmf $\mu^*$ on $\Ga$ defined in (\ref{eq:mustarm}) that can be computed using the output of the Non-crossing Minimal Subfamily Algorithm \ref{alg:non-cross-min-subfam}.
Let $\tilde \Gamma\subset\Ga(A,B)$ be any subfamily with non-crossing paths, such that there is a pmf $\tilde \mu$ strictly supported on $\tilde{\Ga}$ that is optimal for $\Mod_{2,\si} \Gamma(A,B)$.

Then $\tilde{\Gamma} = \Gamma$ and $\tilde{\mu}=\mu^*$.
\end{lemma}

\begin{proof}
Recall that $M:=\Mod_{2,\si}\Ga=\Mod_{2,\si}\Ga(A,B)$, by Proposition \ref{prop:non-cross-min-subfam} (i), which has already been proved.
Let $\gamma_1, \gamma_2, \cdots, \gamma_n$ be the paths of $\Gamma$ in the order that they are identified in the algorithm, and $m(\ga_1), m(\ga_2),\dots, m(\ga_n)$ the corresponding weights defined on step 3 of the algorithm.  
We will begin by proving that $\gamma_1$ must be in $\tilde \Gamma$.
Since the edges where $\rho^*=0$ are removed in step 1 of the algorithm, all the edges of $\gamma_1$ have nonzero $\rho^*$. Therefore, by \eqref{eq:rhomu} applied to $\tilde{\mu}$, each edge of $\gamma_1$  must be contained in at least one path in $\tilde \Gamma$.
Thus, Lemmas  \ref{components} and \ref{toppath} (ii)  show that $\gamma_1 \in \tilde \Gamma$. Furthermore,  Lemma \ref{toppath} (ii) implies
that there must be some edge $e_1$ in $\gamma_1$ that is contained in no other path in $\tilde \Gamma$. Again, by (\ref{eq:rhomu}) applied to $\tilde{\mu}$ on the edge $e_1$, we must have
\begin{equation}\label{eq:mutilde-eone}
\tilde{\mu}(\gamma_1)M   = \sigma(e_1) \rho^*(e_1).
\end{equation}
Also, for an edge $a\in\ga_1$ we have
\begin{align}
\sigma(a) \rho^*(a) & =M\sum_{\tilde{\ga}\in\tilde{\Ga}}\tilde{\mu}(\tilde{\ga})\cN(\tilde{\ga},a) & \text{(by (\ref{eq:rhomu}) applied to $\tilde{\mu}$)}\notag\\
& \ge M\tilde{\mu}(\ga_1) & \text{(since $\ga_1\in\tilde{\Ga}$)}\notag\\
&  \ge \sigma(e_1) \rho^*(e_1) & \text{(by (\ref{eq:mutilde-eone}))}.\label{eq:eone-min}
\end{align}
In conclusion,
\begin{align} 
\tilde{\mu}(\gamma_1)M  & = \sigma(e_1) \rho^*(e_1) &\text{(by (\ref{eq:mutilde-eone}))}\notag\\
& = \min_{a \in \gamma_1} \sigma(a) \rho^*(a) &\text{(by (\ref{eq:eone-min}))}\notag\\
& =m(\ga_1),&\text{(by Step 3 in Algorithm \ref{alg:non-cross-min-subfam})}
\label{minedge}
\end{align}
and hence $ \tilde{\mu}(\gamma_1) = m(\gamma_1)/M$. In particular, this implies that, if $e_1'\in\ga_1$ is an edge where the minimum in Step 3 of Algorithm \ref{alg:non-cross-min-subfam} is attained, then $e_1'$ does not belong to any other path in $\tilde{\Ga}\setminus\{\ga_1\}$ and so when $e_1'$ is removed in Step 1 of the following iteration, the remaining paths in $\tilde{\Ga}\setminus\{\ga_1\}$ are not affected. Therefore, we can now set up an induction.

Assume that $k\ge 2$ and $\gamma_1, \cdots, \gamma_{k-1}$ are in $\tilde \Gamma$ with
$\tilde{\mu}(\gamma_j) = \frac{m(\gamma_j)}{M}$ for $j = 1, \cdots, k-1$. 
Consider the point in the algorithm where we have just identified $\gamma_k$ as the top path, 
and consider an edge $e \in \gamma_k$.  
Since $e$ was not removed yet,
\begin{equation}\label{eq:re-positive-ind}
 r(e) :=  \sigma(e) \rho^*(e) - \sum_{j<k: \,e \in \gamma_j} m(\gamma_j) >0.
 \end{equation}
Therefore,
\begin{align*}
\sum_{j<k: \,e \in \gamma_j} \tilde{\mu}(\gamma_j) & =\frac{1}{M} \sum_{j<k: \,e \in \gamma_j} m(\gamma_j) & \text{(by induction hypothesis)}\\
& < \frac{\sigma(e) \rho^*(e)}{M}  & \text{(by (\ref{eq:re-positive-ind}))}\\
& =\sum_{\gamma: \,e \in \gamma} \tilde{\mu}(\gamma),& \text{(by  \eqref{eq:rhomu} applied to $\tilde{\mu}$)}
\end{align*}
This means that there must be a path in
 $\tilde \Gamma \setminus \{ \gamma_1, \cdots, \gamma_{k-1} \} $ that contains $e$, since we have not accumulated enough mass from the paths $\gamma_1, \cdots, \gamma_{k-1}$. 
Thus, Lemmas \ref{components} and \ref{toppath}\,(ii) show that $\gamma_k \in \tilde \Gamma$, and they further imply that
 there must be some edge  $e_k\in\gamma_k$ that is contained in no other path in $\tilde \Gamma \setminus \{ \gamma_1, \cdots, \gamma_{k} \}$.
 Such edge $e_k \in \gamma_k$ will satisfy
\begin{equation*} 
r(e_k) = \min_{a \in \gamma_k} r(a),
\end{equation*}
by the same argument as in (\ref{eq:eone-min}).
Therefore,  
\begin{equation}\label{eq:mutildegak}
 \tilde{\mu}(\gamma_k) = \frac{m(\gamma_k)}{M},
 \end{equation}
by the same argument as in (\ref{minedge}).

This completes the induction proof. Therefore, we have shown that $\Ga\subset\tilde{\Ga}$. Moreover, (\ref{eq:mutildegak}) implies that $\tilde{\mu}(\ga)=0$ for $\ga\in\tilde{\Ga}\setminus\Ga$, and therefore, $\Ga=\tilde{\Ga}$. Finally, $\tilde{\mu}=\mu^*$ by (\ref{eq:mutildegak}) and Proposition \ref{prop:non-cross-min-subfam} (ii), which has already been proved.
\end{proof}

\begin{proof}[Proof of Theorem \ref{existence}]
Apply the Non-Crossing Minimal Subfamily Algorithm \ref{alg:non-cross-min-subfam} to $\Ga(A,B)$ and use  Proposition \ref{prop:non-cross-min-subfam}. 
\end{proof}

\subsection{Scaling limit of non-crossing minimal subfamilies}

In this section,  we will prove Theorem \ref{limitpmf} (modulo Proposition  \ref{subdiagHarmConv}, which is a version of Theorem \ref{HarmConvThm} that applies to the analytic extension provided by Lemma \ref{reflection}).
Assume the hypotheses of Theorem \ref{limitpmf}  hold.
Then by Proposition \ref{prop:pmf-minimal} there is a unique extremal density $\rho_n$ on the edges of $G_n^\bullet$ and a unique optimal pmf $\mu_n$ on $\Gamma_n$, which are related by
\begin{equation}\label{rhoandmu}
 \frac{\sigma_n(e)\rho_n(e)}{\Mod_{2,\si_n}(\Gamma_n)} = \mathbb{P}_{\mu_n}\left[ e\in\underline{\gamma} \right]
\end{equation}
for all  $e \in G^\bullet_n$. 
We will prove that $\mu_n$ converges to the transverse measure on the family of extremal curves in $\Omega$ as $\epsilon \to 0$.

We begin with a lemma about a useful dual minimal subfamily.
Since this lemma only relates to the orthodiagonal map (and not to the underlying domain),  for ease of reading, we omit the subscript/superscript $n$.

\begin{lemma}\label{lem:fulkerson-duality}
Let $G$ be an orthodiagonal map with boundary arcs $S_1, T_1, S_2, T_2$. Let $\sigma$ be the canonical weights on the edges of $G^\bullet$ as in Equation (\ref{eq:canonical-weights}).
Consider the family $\Ga(A,B)$ of all paths in $G^\bullet$ from $A= V^\bullet \cap S_1 $ to $B= V^\bullet \cap S_2$, and let $\Gamma\subset\Ga(A,B)$ be the unique non-crossing minimal subfamily produced by Theorem \ref{existence}, and let $\mu$ by the corresponding optimal pmf defined in Proposition \ref{prop:non-cross-min-subfam} (ii). Likewise,  consider the dual graph $G^\circ$ with edge-weights as defined in Equation (\ref{eq:canonical-weights}), and let  $\Gamma^\circ$ be the unique non-crossing minimal subfamily of all paths in $G^\circ$ from $\alpha = V^\circ \cap T_1$ to $\beta = V^\circ \cap T_2$. 
 Then
 \begin{enumerate}
\item[(i)] 
$\Mod_{2,\si}(\Gamma) \cdot \Mod_{2,\si}(\Gamma^\circ) = 1.$
\item[(ii)] 
Let $\rho$ be the extremal density for $\Gamma$ and let $\eta$ be the extremal density for $\Gamma^\circ$.  Then
\begin{equation}\label{eq:dual-edge-probability}
\eta(e^\circ) = \frac{\sigma(e) \rho  (e)}{\Mod_{2}(\Gamma  )} = \mathbb{P}_{\mu  }\left[ \gamma \text{ contains } e \right] = \sum_{\gamma\in\Ga  : \,e \in \gamma} \mu  (\gamma), 
\end{equation}
for all pairs of dual edges $(e, e^\circ) \in E^\bullet \times E^\circ  $.
\item[(iii)] Fix a path $\gamma \in \Gamma$.  Then for every edge $e$ in $\gamma$, there is a path $\gamma^\circ \in \Gamma^\circ$ that contains the dual edge $e^\circ$.
\item[(iv)] Fix a path $\gamma^\circ \in \Gamma^\circ$.  Then, every path $\gamma \in \Gamma$ contains exactly one edge $e$ whose dual edge $e^\circ$ is in  $\gamma^\circ.$
\end{enumerate}
\end{lemma}

\begin{proof}
From Fulkerson duality, see \cite[Section 4.2]{acfpc:ampa2019}, we have that 
\begin{equation}\label{eq:fulkerson-duality}
\Mod_{2,\si}(\Gamma  ) \cdot \Mod_{2,\si^{-1}}(\hat{\Gamma}  ) = 1,
\end{equation}
where $\hat{\Ga}  $ is the family of all $A  B  $-cuts, and $\si^{-1}(e):=1/\si(e)$, for every $e\in E(G^\bullet)$. Recall that $S\subset V^\bullet$ is an $A  B  $-cut if $A  \subset S$ and $B  \cap S=\emptyset$. Moreover, the usage of $S$ is defined by $\cN(S,e)=\ones_{\{e\in\delta S\}}$ where $\delta S=\{e=xy\in E^\bullet: |S\cap \{x,y\}|=1\}$. 
Note that  we can define a bijective map $r: E^\bullet \to E^\circ$ by setting $r(e) := e^\circ$ which is the dual edge to $e$.  
Therefore, applying the transformation $r$ to each edge in $\delta S$, we obtain an object $r(\de S)$ on the dual graph $G^\circ  $. 
We want to show that $r(\de S)$ contains a path connecting $\alpha  $ and $\beta  $ and conversely every such path comes from a cut $S$. 
First, every node in $V^\bullet \setminus \partial G  $ is contained in the interior of a closed  dual face for $G^\circ  $. 
In the case when the node $v$ is in $\partial G $, let $C$  be the closure of the component of $v$ in the intersection of the unbounded face for $G^\circ  $ with $\hat{G}  $. Take a small  ball $D_v$ centered at $v$ that is disjoint from any other vertex and define $C\cup D_v$ to be the dual face containing $v$. 
By taking the union of all the dual faces containing a node in $S$ and then taking the interior of this set, we obtain an open set $U$ containing $S$ (and hence $A  $) in its interior, such that $B  \subset {\rm int}(U^c)$, where $U^c$ is the complement of $U$.
Now we can apply Lemma 2.4 in \cite{ebpc}, where the quadrilateral is our set $\hat{G}  $, and obtain a path $\ga$ connecting $\alpha  $ and $\beta  $ with $\ga\subset \bd U\cap\hat{G}  $. By construction, we see that $\ga$ is a union of dual edges, and hence is a path in $G^\circ  $. Conversely, every such path can be continued to get a Jordan domain with $A  $ in its interior and $B  $ in its exterior. 

Finally, note that, by Equation (\ref{eq:canonical-weights}), for every pair of dual edges $(e,e^\circ)$ we have $\si(e^\circ)=\si(e)^{-1}$.
Therefore,  we conclude that $\Mod_{2,\si^{-1}}(\hat{\Ga}  )=\Mod_{2,\si}(\Ga^\circ  )$ and (i) follows from (\ref{eq:fulkerson-duality}).
Moreover, (\ref{eq:dual-edge-probability}) follows from (\ref{eq:rhomu}) and   \cite[Theorem 4 and Theorem 7]{acfpc:ampa2019}, by letting $\eta^\bullet  $ be the extremal density for $\hat{\Ga}  $ and setting
$\eta  (e^\circ)=\eta^\bullet  (r(e^\circ))$ (thinking of $r$ as an idempotent operation).

For (iii), fix $\gamma \in \Gamma$ and let $e $ be an edge in  $\gamma$ with dual edge $e^\circ$.
Suppose that $e^\circ$ is not contained in any path in $\Gamma^\circ$.  
By applying \eqref{eq:dual-edge-probability} with the roles of $\Gamma$ and $\Gamma^\circ$ reversed, we find that $\eta(e^\circ) = \rho(e) = 0$.
However, by \eqref{eq:dual-edge-probability} again, this implies that $e$ cannot be contained in any path in $\Gamma$, yielding a contradiction.

For (iv), fix $\gamma^\circ \in \Gamma^\circ$.
As in the first paragraph, $\gamma^\circ$ can be continued outside $\hat{G}$ to get a Jordan domain with $A  $ in its interior and $B  $ in its exterior. 
Then each $\gamma \in \Ga$ must cross this Jordan curve through an edge in $\gamma^\circ$, implying that 
$\gamma$ must contain at least one such edge $e$ with $e^\circ \in \gamma^\circ.$
Since every minimal subfamily is also a Beurling subfamily, $\gamma^\circ$ has length 1 under $\eta$.
Thus by (\ref{eq:dual-edge-probability}),
\begin{align*}
 1 = \sum_{e\, :\, e^\circ \in \gamma^\circ} \eta(e^\circ) 
 &= \sum_{e \,: \,  e^\circ \in \gamma^\circ} \sum_{\gamma \in \Ga \,: \, e \in \gamma} \mu (\gamma) \\
&= \sum_{\gamma \in \Ga} \mu(\gamma) \left[\text{ \# of edges } e \in \gamma \text{ with } e^\circ \in \gamma^\circ\right]\ge 1.
\end{align*}
Therefore, the number of edges $e\in\ga$ with $e^\circ\in\ga^\circ$ must be identically equal to one.

\end{proof}

The non-crossing condition for the non-crossing minimal subfamily $\Gamma_n$ allows us to order the paths in $\Gamma_n$:  
 $\gamma < \gamma' $ means $ \gamma \neq \gamma' $ and  $\gamma' $  is the top path in the subgraph consisting of these two paths.
 Equivalently,  $\gamma < \gamma'$ means that $\gamma$ is identified after $\gamma'$ in the Non-crossing Minimal Subfamily Algorithm.
We use this ordering in the next lemma.

\begin{lemma}\label{GammaHatLemma}
Assume the hypotheses of Theorem \ref{limitpmf} hold, and let $\phi$ be the conformal map from $\Omega$ onto the rectangle $R=(0,L)\times (0,1)$ so that (after extending $\phi$ to the boundary) $\phi(A)$ is the left side of $R$ and $\phi(B)$ is the right side of $R$. 
 Then $\phi$ has a conformal extension to a domain which contains
 \begin{equation}\label{gammahat}
\hat{\Gamma}_n = \{ \gamma \in \Gamma_n \, : \, \dist(\gamma, \tau_1 \cup \tau_2) \geq 4\epsilon_n^{1/4} \}
\end{equation}
when $n$ is large enough.
Moreover, there exists $\delta_n>0$ with $\delta_n \to 0$ as $n \to \infty$ so that the following hold.
\begin{enumerate}
\item[(i)] Let $\gamma' \in \hat\Gamma_n$ and define
              $\displaystyle p(\gamma') = \sum_{\gamma \in \Gamma_n, \gamma \leq \gamma'} \mu_n(\gamma)$.
Then for all points $u$ on $\gamma'$,              
\begin{equation} \label{ImagOsc}
 \Im  \phi(u) \in [p(\gamma')-\delta_n, \, p(\gamma')+\delta_n].
\end{equation}
\item[(ii)]  For $q \in [0,1]$ there exists $\gamma' \in \hat \Gamma_n$ with $|p(\gamma') -q| \leq 2\delta_n$ and $\Im \phi(u) \in [q-3\delta_n, q+3\delta_n]$ for any point $u$ on $\gamma'$.
\item[(iii)] $\mu_n(\hat \Gamma_n) \geq 1-4\delta_n$.
\end{enumerate}
\end{lemma}

The proof of this lemma will rely on Proposition  \ref{subdiagHarmConv}, which is established in Section \ref{harmonic}.
We also note that we will use $C$ to refer to constants that are independent of $\epsilon_n$ but may depend on the analytic quadrilateral $\Omega$.  The exact value of these constants may change from line to line.

\begin{proof}
By applying Lemma \ref{reflection} (with the roles of $A \cup B$ and $\tau_1 \cup \tau_2$ switched) and taking $\epsilon_n$ small enough,  we extend $\phi$ to be conformal on a domain $\mathcal{O}$ containing 
$$P_n = \{ z \, : \, \dist(z, \overline{\Omega}) \leq \epsilon_n \text{ and } \dist(z, \tau_1 \cup \tau_2) \geq \epsilon_n^{1/4} \}$$
so that 
\begin{equation} \label{phiHolder}
 |\phi(z) - \phi(w)| \leq C |z-w|^{a} \;\; \text{ for } z,w \in \overline{\mathcal{O}}.
\end{equation}

Let  $h_n$ be the solution to the discrete Dirichlet problem on ${\rm int}(V_n^\bullet)$ with $h_n |_{A_n} = 0$ and $h_n |_{B_n} = L$.
Similarly let  $\tilde{h}_n$ be the solution to the discrete Dirichlet problem on ${\rm int}(V_n^\circ)$ with $\tilde{h}_n |_{\alpha_n} = 0$  and $\tilde{h}_n |_{\beta_n}= 1$
for $\alpha_n = V^\circ \cap T_1^n$ and $\beta_n = V^\circ \cap T_2^n$.
Let $\Gamma_n^\circ$ be the unique non-crossing minimal subfamily of all paths in $G_n^\circ$ from $\alpha_n$ to $\beta_n$, and let $\rho_n$ (resp. $\eta_n$) be the extremal density for $\Gamma_n$ (resp. $\Gamma_n^\circ$).
Note that  by \eqref{rhofromharmonic},
$$\rho_n(v_1v_2) = \frac{1}{L} |h_n(v_1) - h_n(v_2)|  \;  \text{  and  }   \;  \eta_n(w_1w_2) =  |\tilde{h}_n(w_1) - \tilde{h}_n(w_2)|$$
for all $v_1v_2 \in E^\bullet_n$ and $w_1w_2 \in E^\circ_n$.

Fix $\gamma' \in \hat\Gamma_n$, and let
$$ p=p(\gamma') = \sum_{\gamma \in \Gamma_n, \gamma \leq \gamma'} \mu_n(\gamma),$$
recalling that  $\gamma \leq \gamma' $ means $\gamma' $  is the top path in the subgraph consisting of these two paths.
Let $e$ be an edge in $\gamma'$, with dual edge $e^\circ $.
Consider a path $\gamma^\circ \in \Gamma_n^\circ$ that contains $e^\circ$ (which exists by Lemma \ref{lem:fulkerson-duality}(iii)).
Under the orientation of edges induced by $\tilde h_n$ in Lemma \ref{dag}, $\gamma^\circ$ is an oriented path from $\alpha_n$ to $\beta_n$ by Lemma \ref{dag}\,(iii).  
 Let $e^\circ = xy$, oriented with respect to this orientation.
 Consider the portion of $\gamma^\circ$ before $x$ and let this be given by vertices $u_0, \cdots, u_k =x$.
  Recall that by Lemma \ref{lem:fulkerson-duality}(iv), every path $\gamma \in \Gamma_n$ contains exactly one edge $a$ with $a^\circ \in \gamma^\circ.$
 This along with  \eqref{eq:dual-edge-probability} implies that 
 $$ \tilde h_n (x) = \sum_{j=1}^k \tilde h_n(u_j) - \tilde{h}_n(u_{j-1}) = \sum_{j=1}^k \eta_n(u_{j}u_{j-1}) \leq \sum_{\gamma \leq \gamma'} \mu_n(\gamma) = p. $$
 Similarly, 
 $$ \tilde h_n (y) =  \sum_{j=1}^k \eta_n(u_{j}u_{j-1}) + \eta_n(xy) \geq \sum_{\gamma \leq \gamma'} \mu_n(\gamma)  = p.$$
Next we apply
 Proposition \ref{subdiagHarmConv} (which extends Theorem  \ref{HarmConvThm} to points in $P_n$)  to $\tilde h_n$ and $\Im \phi$ (noting that the roles of $A,B$ and $\tau_1, \tau_2$ are switched).  This gives that
 \begin{align}\label{htildebound}
|\tilde{h}_n(y) - \tilde{h}_n(x)| &\leq |\tilde{h}_n(y) - \text{Im } \phi(y) | + |\text{Im } \phi(y) - \text{Im } \phi(x)| + |\text{Im } \phi(x) - \tilde{h}_n(x)|  \nonumber \\
 &\leq \frac{C}{\log^{1/2}\left(\diam(\Omega)/\epsilon_n\right)} + C\epsilon_n^{a} \\
 &\leq \delta_n, \nonumber
\end{align}
where $\delta_n = \frac{C}{\log^{1/2}\left(1/\epsilon_n \right)}$ for some constant $C$ depending on $\Omega$.
Now let $u$ be a point on $e$.
Then $|u-x| \leq 2\epsilon_n$ and by arguing as in \eqref{htildebound} we obtain
\begin{align*}
\text{Im }\phi(u) 
  &\leq \tilde h_n(x) +  |\text{Im } \phi(x) - \tilde h_n(x)| + |\text{Im }\phi(u) - \text{Im }\phi(x)|  \\
  &\leq p + \delta_n.
\end{align*}
Similarly,
\begin{align*}
\text{Im }\phi(u) 
  &\geq   \tilde h_n(y) -  |\text{Im } \phi(y) - \tilde h_n(y)| - |\text{Im }\phi(u) - \text{Im }\phi(y)|  \\
  &\geq p - \delta_n.
\end{align*}
Since this hold for all edges of $\gamma'$, we have established \eqref{ImagOsc}.

To prove (ii) and (iii), we begin by looking at the curve $\gamma_{\text{first}} \in \hat \Gamma_n$ that comes first in our ordering 
(i.e.~$ \gamma_{\text{first}} \leq \gamma$ for all $\gamma \in \hat \Gamma_n$).
Let $\gamma_0 \in \Gamma_n$ be the path directly before $\gamma_{\text{first}}$.
Then since $\gamma_0$ is not in $\hat \Gamma_n$, we have that $\dist(\gamma_0, \tau_1 ) < 4\epsilon_n^{1/4}$.
In particular, there must be a point $z$ from an edge $e$ in $\gamma_0$ with $\dist(z, \tau_1 ) < 4\epsilon_n^{1/4}$.
By Lemma \ref{lem:fulkerson-duality}(iii), there exists $\gamma^\circ \in \Gamma_n^\circ$ containing the dual edge $e^\circ$.
Let $a^\circ$ be the edge after $e^\circ$ in $\gamma^\circ$ under the orientation of $\gamma^\circ$ from $\alpha_n$ to $\beta_n$
given by Lemma \ref{dag}.
We claim that $\gamma_{\text{first}}$ contains $a$ (the dual edge to $a^\circ$).
To see why, first note that $\gamma_{\text{first}}$ cannot contain $e$, because $\dist(\gamma_{\text{first}}, \tau_1) \geq 4\e_n^{1/4}$ but $\dist(e, \tau_1 ) < 4\epsilon_n^{1/4}$.
By Lemma \ref{lem:fulkerson-duality}(iii) applied to $\gamma^\circ$, there is some path in $\Gamma_n$ containing $a$.  
Since there are no paths between $\gamma_{\text{first}}$ and $\gamma_0$ in the ordering,  $\gamma_{\text{first}}$ must contain $a$.

Since each quadrilateral in $G_n$ has diameter bounded by $2\epsilon_n$, we have
$$\dist(\gamma_{\text{first}}, \tau_1)  \leq \dist(a, \tau_1) \leq \dist(a, z) + \dist(z, \tau_1) < 4\epsilon_n + 4 \epsilon_n^{1/4} \leq 8 \epsilon_n^{1/4}. $$
Then by \eqref{phiHolder},
\begin{equation}\label{first}
\dist(\phi(\gamma_{\text{first}}), [0,L]) \leq C \epsilon_n^{a/4} \leq \delta_n,
\end{equation} 
where for the last inequality, we adjust the $\Omega$-dependent constant in the definition of $\delta_n$, if needed.
By \eqref{ImagOsc} this implies that $p(\gamma_{\text{first}}) \leq 2\delta_n$.
By a similar argument, the path $\gamma_{\text{last}} \in \hat \Gamma_n$ that comes last in the ordering must satisfy $p(\gamma_{\text{last}}) \geq 1-2\delta_n$.
This shows that (iii) holds, since
$$\mu_n(\hat\Gamma_n) = p(\gamma_{\text{last}})-p(\gamma_0) \geq 1-4\delta_n.$$

Suppose that $\gamma'< \gamma''$ are two paths in $\hat\Gamma_n$ that are adjacent in the ordering.  Let $e$ be an edge in $\gamma''$ with dual edge $e^\circ = xy$.
Then 
$$p(\gamma'') - p(\gamma') = \mu_n(\gamma'')  \leq \sum_{\gamma: e \in \gamma} \mu_n(\gamma) 
 =\eta_n(e^\circ)   =  |\tilde{h}_n(y) - \tilde{h}_n(x)| \leq  \delta_n, $$
by \eqref{eq:dual-edge-probability} and \eqref{htildebound}.
Therefore for $q \in [0,1]$, there exists $\gamma \in \hat \Gamma_n$ with $|p(\gamma) -q| \leq 2\delta_n$.
By \eqref{ImagOsc}, this proves (ii).

\end{proof}

We are now ready to prove the convergence of the probability measures $\mu_n$.  

\begin{proof}[Proof of Theorem \ref{limitpmf}.]

We begin by using the previous lemma to identify 
 the limit points of $\hat{\Gamma}_n$ (defined in \eqref{gammahat}) as $n \to \infty$,
which are the curves $\gamma_{\star}$ such that there is an increasing sequence $n_k$ and a sequence $\gamma_k \in \hat\Gamma_{n_k}$ with $\gamma_k$ converging to $\gamma_{\star}$ (using the Hausdorff distance as the metric on curves).  
We claim that these limiting curves
are the  extremal curves in $\Omega$:
\begin{equation}\label{extremalcurves}
\Gamma := \bigl\{ \phi^{-1}\left(\{z \in R \, : \, \text{Im}(z) = p\}\right) \, : \, p\in [0,1] \bigr\},
\end{equation}
where $\phi:\Omega \to R= (0,L)\times (0,1)$ is conformal with  $\phi(A)$ equal to the left side of $R$ and $\phi(B)$ the right side of R.
Assume first  that $\gamma_k$ converges to $\gamma_{\star}$ with $\gamma_k \in \hat\Gamma_{n_k}$ for some some increasing sequence $n_k$.  
Then $\phi(\gamma_k)$ converges to $\phi(\gamma_{\star})$. By Lemma \ref{GammaHatLemma}(i), 
the oscillation of the imaginary part of $\phi(\gamma_k)$ must converge to zero, implying that $\Im\left[ \phi(\gamma_{\star}) \right]$ must be constant. This shows that $\gamma_{\star} \in \Gamma$,
and
 the limit points of $\hat\Gamma_n$ are contained in $\Gamma$.
 Combining this with Lemma \ref{GammaHatLemma}(iii) shows that any subsequential limit of $\mu_n$ must be suported on $\Gamma$.
 Although this is all we need for our proof, we point out that one could use 
 Lemma \ref{GammaHatLemma}(ii) to show that every curve of $\Gamma$ is a limit point of $\hat\Gamma_n$.

Recall that $X = \left(\bigcup_n \Gamma_n \right) \cup \Gamma$, 
and $\mu_n$ are elements of the set $\cP(X)$ of probability measures on $(X, \mathcal{B}(X))$.
We will show next that
the set of probability measures $\{ \mu_n \}$ is {\it tight}, meaning that,
for each $\delta >0$ there exists a compact subset $K_\delta$ of $X$ such that $\mu_n(K_\delta) \geq 1-\delta$ for all $n$.
Fix $\delta >0$. Let $N_\delta \in \mathbb{N}$ be large enough so that for $n \geq N_\delta$ the conclusion of Lemma \ref{GammaHatLemma} holds  and $\delta_n \leq \delta /4$.
With $\hat{\Ga}_n$ as in Lemma \ref{GammaHatLemma}, define $K_\delta \subset X$ by 
$$K_\delta := \left( \bigcup_{n=0}^{N_\delta-1} \Gamma_n \right) \cup \left( \bigcup_{n=N_\delta}^\infty \hat\Gamma_n \right)
 \cup \Gamma.$$
We will show $K_\delta$ is compact, meaning every sequence in $K_\delta$ has a subsequence converging to an element in $K_\delta$.
A sequence in $K_\delta$ must satisfy one of the following two cases.  
In the first case, there is a subsequence contained in
$ \left( \bigcup_{n=0}^{N_\delta-1} \Gamma_n \right) \cup \left( \bigcup_{n=N_\delta}^M \hat\Gamma_n \right)$
 for some $M \in (N_\delta, \infty)$.  
Since  this set has a finite number of elements, this means there must be a sub-subsequence that is constant, and hence it converges to an element of $K_\delta$.
In the second case, there is a subsequence $\gamma_{n_k}$ with 
$$\gamma_{n_k} \in \left(\bigcup_{n=k}^\infty \hat\Gamma_n \right) \cup \Gamma.$$  
Hence by \eqref{ImagOsc}, there exists $p_{k} \in [0,1]$  so that  Im$\,\phi(u) \in [p_{k}-\delta_{k}, \, p_{k}+\delta_{k}]$ for all $u \in \gamma_{n_k}$. 
Since $[0,1]$ is compact, we can chose a subsequence of $p_k$ so that 
$p_{{k_j}} \to p \in [0,1]$.  
This implies that $\phi(\gamma_{n_{k_j}})$ converges to $\{z \in R \, : \, \Im(z) = p\}$, 
and hence the sub-subsequence $\gamma_{n_{k_j}}$ converges to an element of $\Gamma \subset K_\delta$.

Note that $\mu_n(K_\delta) = 1$ for $n<N_\delta$ 
and $\mu_n(K_\delta) =  \mu_n(\hat\Gamma_n) \geq 1-4\delta_n \geq 1-\delta$ for $n \geq N_\delta$
by Lemma \ref{GammaHatLemma}(iii).
Therefore, $\{ \mu_n \}$ is tight, and subsequently
Prohorov's Theorem (Theorem 83.10 in \cite{rogers_williams_2000}) implies that $\{ \mu_n \}$ is conditionally compact 
(i.e.~its closure is compact) in the space $\cP(X)$ equipped with the topology of weak convergence.
This implies the existence of subsequential limits of $\mu_n$.

It remains to show that the only possible subsequential limit of $\mu_n$ is the transverse measure on $\Gamma$.  
To that end, let $\mu$ be a limit of $\mu_{n_k}$.
Fix $r \in (0,1)$, and let $N$ be large enough so that for $n \geq N$, the conclusion of Lemma \ref{GammaHatLemma} holds and $\delta_n < r/3$.
Set $$H_r = \{ \gamma \in \left( \cup_{n=N}^\infty \hat{\Gamma}_n \right) \cup \Gamma \, : \,  \phi(\gamma) \subset \mathbb{R} \times (-\infty, r] \},$$
To characterize $\mu$ as the transverse measure on $\Gamma$, we must simply show that 
$\mu(H_r) = r$ for all $r \in (0,1)$.
Let $q = r-3\delta_{n}$. 
By Lemma \ref{GammaHatLemma}(ii) there exist $\gamma_{n}$
with $|p(\gamma_n) - q| \leq 2\delta_n$ and
$\Im \phi (u) \in [q - 3\delta_n, q+3 \delta_n]$ for $u$ in $\gamma_n$. 
Then $H_r$ contains the set
$$  \{ \gamma \in \hat{\Gamma}_n  \, : \, \gamma \leq \gamma_n \}.$$
Thus 
$$ \mu_n(H_r) \geq  p(\gamma_n) - 4\delta_n\geq  q -6\delta_n = r-9\delta_n.$$
Since $H_r$ is a closed set and $\mu$ is the limit of $\mu_{n_k}$ under weak convergence,
$$ \mu(H_r) \geq \limsup_{k \to \infty} \mu_{n_k}(H_r) \geq r.$$ 
Let $\delta>0$ with $r+ \delta < 1$ and set $U_{r+\delta} = \{ \gamma \in \left( \cup_{n=N}^\infty \hat{\Gamma}_n \right) \cup \Gamma \, : \, \phi(\gamma) \subset \mathbb{R} \times [r + \delta, \infty) \}.$
Note that the same argument as above shows that $\mu(U_{r+\delta}) \geq 1-(r+\delta)$.  
Further since $X \setminus H_r$ contains $U_{r + \delta}$, we have
$\mu(X \setminus H_r) \geq \mu(U_{r+ \delta})$.
Therefore $\mu(H_r) = 1-\mu(X \setminus H_r) \leq r+ \delta$.
Since this is true for all $\delta >0$,
 $ \mu(H_r) =  r$ for all $r \in (0,1)$.  This characterizes the limit measure as the transverse measure on $\Gamma$.

\end{proof}

\subsection{Further consequences of the minimal non-crossing subfamily}

As mentioned in the Introduction, Theorem \ref{existence} has an interpretation in terms of current flow. 

Let $G=(V,E, \sigma)$ be a finite weighted connected graph. 
Assume that there are disjoint nonempty sets of  boundary vertices $A$ and $B$.
Solve the (weighted) Dirichlet problem, as in Proposition \ref{discreteDP}, and find a function $h: V\rightarrow\R$, so that $h|_A=0$, $h|_B=1$, and $h$ is discrete harmonic on $V\setminus(A\cup B)$, as in (\ref{eq:si-harmonic}). The function $h$ produces a current flow $f$ via Ohm's law:
\begin{equation}\label{eq:ohmslaw}
f(xy) = \si(xy)\left(h(y)-h(x)\right) \;\;\; \text{ for all } xy \in \vv{E}.
\end{equation}
In particular,  in the notations of Section \ref{sec:discrete-harm-fnct}, $f$ is a flow from $A$ to $B$, meaning that it satisfies the node law in $V\setminus (A\cup B)$. Moreover, it is well-known that ${\rm strength}(f)=M$ in this case, where $M:=\Mod_{2,\si}(\Ga(A,B))$ and $\Ga(A,B)$ is the family of  all paths from $A$ to $B$ in $G$. To see this, use for instance \cite[Lemma 2.9]{LP:book} with their $\theta$  replaced by our flow $f$ and their gradient flow $df$ replaced by our $dh$, then recall the connection (\ref{eq:connection-harmod}). Finally, Thomson's principle \cite[Page 40]{LP:book} states that among all flows from $A$ to $B$ with strength $M$, $f$ is the unique energy minimizer, where $\cE(f):=\sum_{e\in E}\si(e)^{-1}|f(e)|^2$.

Lemma  \ref{lem:connect-flow-mu} below describes the connection between current flow and the probabilistic interpretation of modulus. Given a path $\gamma$ from A to B define $x_\gamma$ to be the unit flow along $\gamma$. Then, the current flow $f$ has a {\it path flow decomposition}, if it can be written as a sum $\sum_{\ga\in\Ga} a_\gamma x_\gamma$ of path flows, with coefficients $a_\ga> 0$, for $\ga$ in some family $\Ga\subset\Ga(A,B)$ and the paths in $\Ga$ all have the same direction of flow on every edge $e\in E$. 
\begin{lemma}\label{lem:connect-flow-mu}
Let $f$ be the current flow defined in (\ref{eq:ohmslaw}). Assume $\Ga\subset\Ga(A,B)$ strictly supports an optimal pmf $\mu^*$ for the probabilistic interpretation of $\Mod_{2,\si}(\Ga(A,B))$.
Then,  the paths in $\Ga$ all have the same direction of flow on every edge $e\in E$ and $f=\sum_{\ga\in\Ga} M\mu^*(\gamma) x_\gamma$ is a path flow decomposition for $f$, where $M:=\Mod_{2,\si}(\Ga(A,B))$.
Conversely, if  $\sum_{\ga\in\Ga} a_\gamma x_\gamma$ is a path flow decomposition for $f$, then
\[
\sum_{\ga\in\Ga} a_\gamma={\rm strength}(f)=M,
\]
Moreover, $\tilde{\mu}(\ga):=a_\ga/M$ is an optimal pmf strictly supported on $\Ga$.
\end{lemma}
\begin{proof}
First, we want to show that if $\mu^*$ is an optimal pmf strictly supported on $\Ga\subset\Ga(A,B)$, then the paths in $\Ga$ all have the same direction of flow on each edge. Suppose not, then we can find an edge $e=xy$ and two paths $\ga,\ga'\in\Ga$, such that, $\ga=\ga_{ax}\cup xy \cup \ga_{yb}$ for some $a\in A$, $b\in B$ and paths $\ga_{ax}$ from $a$ to $x$ and $\ga_{yb}$ from $y$ to $b$. Likewise, $\ga'=\ga'_{a'y}\cup yx \cup \ga'_{xb'}$ for some $a'\in A$, $b'\in B$ and paths $\ga'_{a'y}$ from $a'$ to $y$ and $\ga'_{xb'}$ from $x$ to $b'$.  Then, find paths $\ga_1 \subset\ga_{ax}\cup\ga'_{xb'}$ and $\ga_2\subset \ga'_{a'y}\cup\ga_{yb}$ connecting $A$ and $B$, and define a new family of paths $\Ga'$ by adding $\ga_1$ and $\ga_2$  to $\Ga$. If $0<t<\min(\mu^*(\ga),\mu^*(\ga'))$, define a measure $\nu$ on $\Ga'$, by setting $\nu(\ga):=\mu^*(\ga)-t$, $\nu(\ga'):=\mu^*(\ga')-t$, $\nu(\ga_1)=\nu(\ga_2):=t$, and keeping $\nu=\mu^*$ otherwise.
Then, $\nu$ is a pmf on $\Ga'$, and has the same edge probabilities as $\mu^*$ at every edge 
except the edges in the set $H:=(\ga\cup\ga')\setminus(\ga_1\cup\ga_2)$ where they are smaller. In particular,
 $xy$ is guaranteed to be in $H$ and 
\[
\bP_\nu(xy\in\underline{\ga})=\bP_{\mu^*}(xy\in\underline{\ga})-2t<\bP_{\mu^*}(xy\in\underline{\ga}).
\]
This contradicts the optimality of $\mu^*$, by the second equality in (\ref{eq:prob-dual}).

Consider $g:=\sum_{\ga\in\Ga} M\mu^*(\gamma) x_\gamma$. Then $g$ satisfies the node law on $V\setminus (A\cup B)$ and ${\rm strength}(g)=M$.  Thus, if $\rho^*$ is the extremal density for $\Mod_{2,\si}(\Ga(A,B))$, then
\begin{align*}
|g(e)| & = \sum_{\ga\in\Ga} M\mu^*(\gamma) \ones_{\{e\in\ga\}} & \text{(all paths flow in same direction)}  \\
& =M\bP_{\mu^*}(e\in\underline{\ga}) & \text{(by definition of probability)}  \\
&  =\si(e)\rho^*(e)  & \text{(by (\ref{eq:rhomu}))} \\
& =\si(e)|h(y)-h(x)| & \text{(by (\ref{eq:connection-harmod}))}\\
& =|f(e)|. & \text{(by (\ref{eq:ohmslaw}))}
\end{align*}
Therefore, $g$ is also a flow energy minimizer, and by uniqueness in the Thomson's principle we must have $f=g$. So, we got a path flow decomposition for $f$.

Conversely,  if  $\sum_{\ga\in\Ga} a_\gamma x_\gamma$ is a path flow decomposition for $f$, then
$\sum_{\ga\in\Ga} a_\gamma={\rm strength}(f)=M$. Therefore, setting $\tilde{\mu}(\ga):=a_\ga/M$ yields a pmf on $\Ga$. Moreover, every path in $\Ga$ flows in the same direction as $f$ on every edge. Thus,
\[
M\bP_{\tilde\mu}(e\in\underline{\ga})= \sum_{\ga\in\Ga}M \tilde{\mu}(\gamma) \ones_{\{e\in\ga\}}=\sum_{\ga\in\Ga}a_\gamma \ones_{\{e\in\ga\}}=|f(e)|=\si(e)\rho^*(e).
\]
Hence, by (\ref{eq:rhomu}), $\tilde\mu$ is an optimal pmf for the probabilistic interpretation of $\Mod_{2,\si}(\Ga(A,B))$, that is strictly supported on $\Ga$.
\end{proof}
Theorem \ref{existence} and the non-crossing minimal subfamily algorithm imply that there is a unique way of decomposing the current flow $f$ so that the paths $\gamma$ in the path decomposition do not cross. Moreover,  the corresponding coefficients $a_\gamma$ are computed by the algorithm.
\begin{prop}
Let $G=(V, E, \sigma, A, B)$ be a plane network with boundary and let $f$ be the current flow from $A$ to $B$ as described above. Then, $f$ can be decomposed as $\sum_{\ga\in\Ga} m_\gamma x_\gamma$, where the family $\Ga$ and the weights $\{m_\ga\}$ are output by the Non-Crossing Minimal Subfamily Algorithm \ref{alg:non-cross-min-subfam}.

Moreover, if $\sum_{\tilde{\ga}\in\tilde{\Ga}} a_{\tilde{\ga}}x_{\tilde{\ga}}$ is another path decomposition, and  $\tilde{\Ga}$ consists of non-crossing paths from $A$ to $B$, then $\tilde{\Ga}=\Ga$ and the $a_\ga$-weights and $m_\ga$-weights coincide. 
\end{prop}
\begin{proof}
Let $\Ga$  be the subfamily given by the non-crossing minimal subfamily algorithm.
Then, 
$\sum_{\ga\in\Ga} m_\gamma x_\gamma$ coincides with the current flow $f$ by construction. Indeed,
the algorithm routinely picks the top path $\ga$ and finds a weight $m_\ga$ so that the corresponding path flow satisfies $0\le m_\ga x_\ga \le f$ on every edge. By Lemma \ref{components}, when the algorithm terminates the flow $f$ on every  edge has been exhausted.

Now suppose $\sum_{\tilde{\ga}\in\tilde{\Ga}} a_{\tilde{\ga}}x_{\tilde{\ga}}$ is another path decomposition for $f$,  where $\tilde{\Ga}$ consists of non-crossing paths from $A$ to $B$. 
By Lemma \ref{lem:connect-flow-mu}, defining $\tilde{\mu}(\tilde{\ga}):=a_{\tilde{\ga}}/M$, with $M:=\Mod_{2,\si}(\Ga(A,B))$, yields an optimal pmf that is strictly supported on $\tilde{\Ga}$. 
Then, by Lemma \ref{minimalfamilyuniqueness}, we obtain that $\Ga=\tilde{\Ga}$, and $\tilde{\mu}(\ga)=m_\ga/M$, for every $\ga\in\Ga$. Thus, the $a_\ga$-weights and $m_\ga$-weights coincide.
\end{proof}
Another consequence of the non-crossing minimal subfamily algorithm is that it gives an alternative way to draw a rectangle-packing given a planar electrical network, as in the famous result from \cite{brooks-smith-stone-tutte:duke1940}. Namely, given a weighted plane graph and two arcs of nodes $A$ and $B$ bordering the unbounded face, the resulting current flow from $A$ to $B$, interpreting the edge-weights as electric conductances, gives rise to a tiling of a rectangle by a collection of smaller rectangles that have certain precise combinatorial properties derived from the graph, see \cite[Theorem 6.1]{lovasz2019} for a precise description of these rectangle tilings. See also \cite[Lemma 5.4]{prasolov-skopenkov:jcombthsera2011} for similar tilings of more general polygons in the plane. We state the result informally in the following remark and will sketch a proof for it below.
\begin{remark}\label{rem:rect-pack}
Let $G=(V, E, \sigma, A, B)$ be a plane network with boundary. Let $\Ga(A,B)$ be the family of all paths from $A$ to $B$ in $G$.
Set $M:=\Mod_{2,\si} \Ga(A,B)$ and consider the rectangle $\mathcal{R}=[0,1]\times[0,M]$.
Then, the non-crossing minimal subfamily Algorithm \ref{alg:non-cross-min-subfam} can be used to tile $\mathcal{R}$ with a collection $\{R_e\}_{e\in E}$ of smaller rectangles  (possibly degenerating to a point), so that each rectangle corresponds to an edge $e\in E$, and every maximal vertical segment composed of sides of rectangles corresponds to one or more nodes in the graph.
 
The reason why we cannot guarantee that the maximal vertical segments composed of sides of rectangles are in bijection with the nodes of the graph, is because such segments might be a union of abutting subsegments, each corresponding to a node. This will be explained in the sketched proof explained below. See also \cite[Section 6.1]{lovasz2019}, in particular \cite[Figure 6.2]{lovasz2019}, where the expression ``resolving $4$-fold corners'' is used.
\end{remark}
\begin{figure}[H]
\begin{minipage}{0.45\textwidth}
\includegraphics[scale=0.5,trim={0 0 0 0},clip]{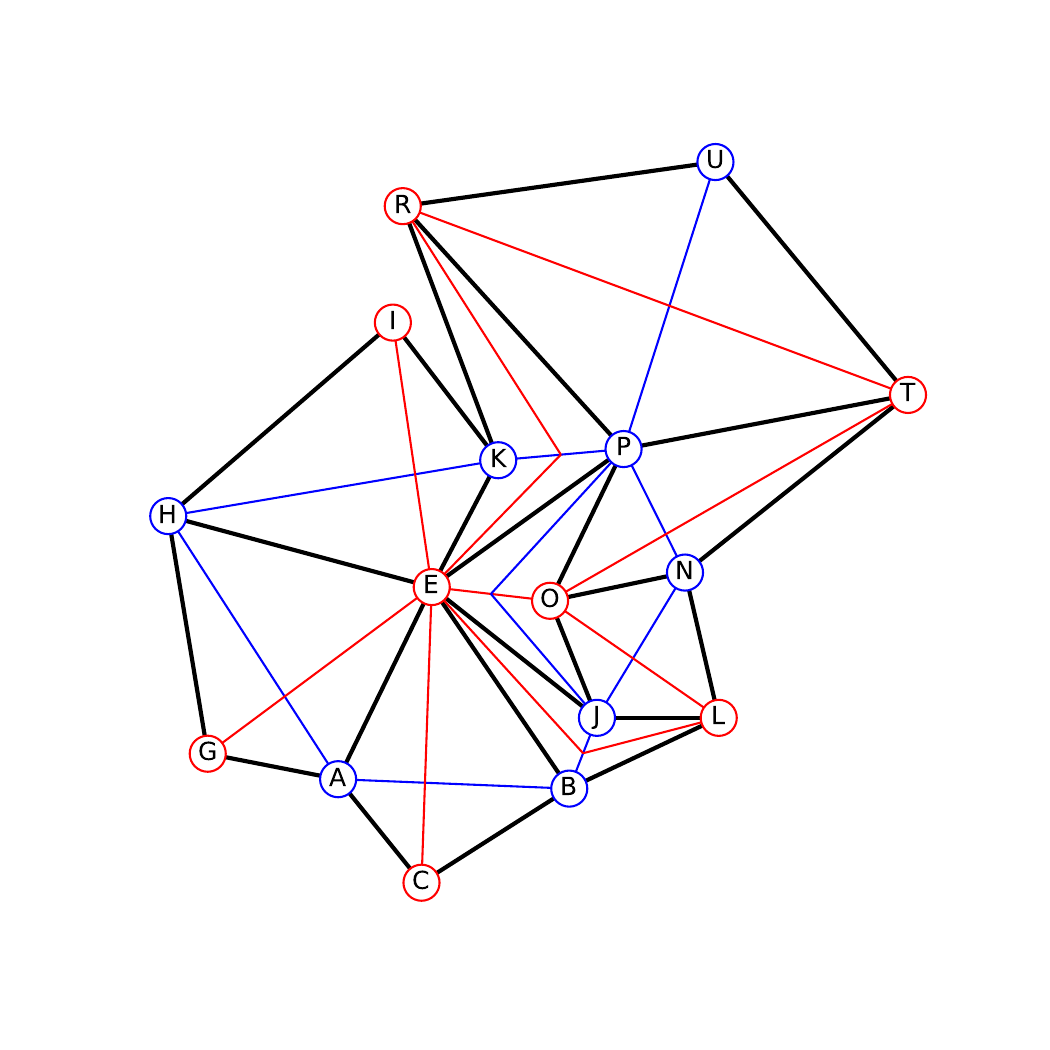}
\end{minipage}
\begin{minipage}{0.45\textwidth}
\includegraphics[scale=0.2,trim={5cm 5cm 5cm 5cm},clip]{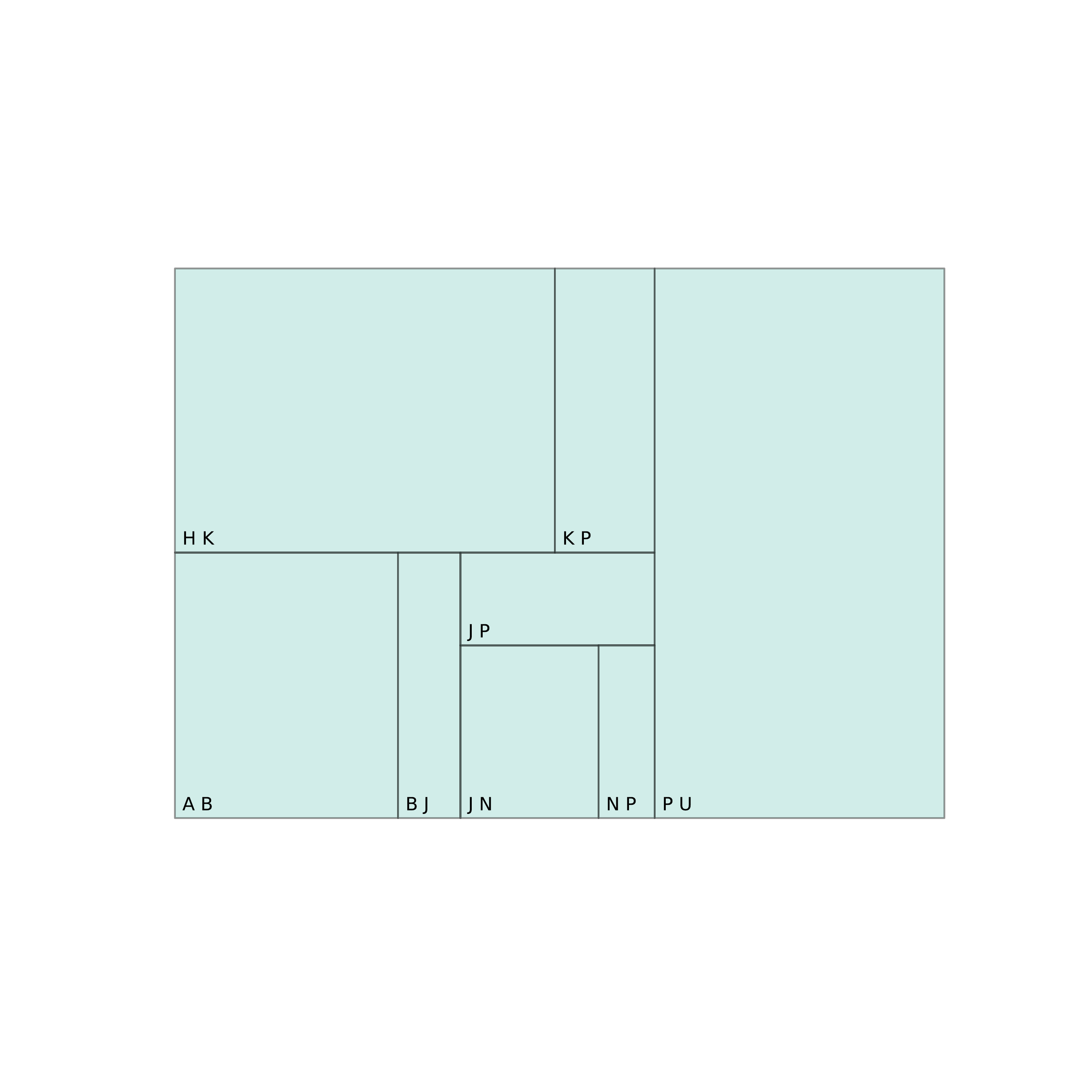}
\end{minipage}
\caption{This is an example of Remark \ref{rem:rect-pack} in the case of orthodiagonal maps. For the map on the left, the primal graph (in blue) is given by the nodes $A,B,H,J,K,N,P,U$. Here we compute the modulus $M$ of the family of curves connecting the set $\{A,H\}$ to the set $\{U\}$. In this case the quantity $\tilde{h}(e)$ coincides with the harmonic function defined on the dual graph with $\tilde{h}=0$ on the set $\{C,L,T\}$ and $\tilde{h}=M$ on the set $\{I,R\}$.}\label{fig:rectpack}
\end{figure}
\begin{proof}[Sketch of proof for Remark \ref{rem:rect-pack}]
Let  $h: V \to \mathbb{R}$ be the discrete harmonic on $V \setminus ( A \cup B )$ with 
$h|_A=0$ and $h|_B=1$. Recall from (\ref{eq:connection-harmod}) and  (\ref{rhofromharmonic}), that $\cE(h)=\Mod_{2,\si}\Ga(A,B)$, and that, for each edge $e=xy \in E$, the extremal density $\rho(e)$ for $\Mod_{2,\si}\Ga(A,B)$, satisfies $\rho(e)=|h(x)-h(y)|$.

For each edge $e= xy \in E$, define a rectangle $R_e$ is of the form 
\begin{equation}\label{eq:rectangles-re}
R_e:=[h(e),h(e)+\rho(e)]\times[\tilde{h}(e),\tilde{h}(e)+\si(e)\rho(e)],
\end{equation}
where $h(e):=\min\{h(x),h(y): e=xy\}$, and $\tilde{h}(e)$ is a height that we will compute using the non-crossing minimal subfamily Algorithm \ref{alg:non-cross-min-subfam}.

Note  that the rectangle $R_e$ degenerates to a point when the corresponding value of $\rho(e)$ is equal to zero. Furthermore, the area of each rectangle $R_e$ is equal to $\sigma(e) \rho^2(e)$, which is the contribution of the edge $e$ to the modulus $\Mod_{2,\si}\Ga(A,B)$, since 
$$\sum_{e \in E} \sigma(e) \rho^2(e) = \Mod_{2,\si} \Gamma(A,B).$$

To show how $\tilde{h}(e)$ is defined and deduce that the rectangles $R_e$ defined in (\ref{eq:rectangles-re}) form a tiling, we perform the non-crossing minimal subfamily Algorithm \ref{alg:non-cross-min-subfam} from the bottom up and write $\ga_1,\dots,\ga_n$ for the output paths in the order that they are added to $\Ga$. So, $\ga_1$ is the bottom path. We also write $m_j:=m(\ga_j)$ for the corresponding weights computed by the algorithm.

We begin by associating to $\ga_1$  the strip $(0,1)\times(0,m_1)$. Writing $\ga_1=x^1_0\ \cdots\ x^1_{k_1}$ we trace the horizontal segments $[h(x^1_j),h(x^1_{j+1})]\times\{0\}$, for $j=0,\dots,k_1-1$, as well as the vertical segments $\{h(x^1_j)\}\times [0,m_1]$, for $j=0,\dots,k_1$.  For every edge $e= x^1_j x^1_{j+1}$,  $j=0,\dots,k_1-1$, set $\tilde{h}(e)=0$.
Finally, for every edge $e= x^1_j x^1_{j+1}$ that attains the minimum in step 3 of the algorithm, trace the horizontal segment $[h(x^1_j),h(x^1_{j+1})]\times\{m_1\}$, thus closing up the rectangle corresponding to these edges.

Rather than a formal proof by induction, we now describe the second iteration. Consider the second path $\ga_2$ and add the strip $(0,1)\times(m_1,m_1+m_2)$. Note that if an edge $e$ was in the bottom path $\ga_1$ in the previous iteration and did not attain the minimum in step 3, then it will again be in the bottom path $\ga_2$.  Moreover, if we consider a maximal set of consecutive edges, $e_1,\cdots, e_r$ 
in $\ga_1$,  that attained the minimum in step 3, they
form a subpath $\ga_1':=x_{j_1}^1 x_{j_1+1}^1 \cdots x_{j_1+r}^1$ of $\ga_1$.  Also, the tops of the rectangles corresponding to the edges $e_i$ in $\ga_1'$ were closed up in the previous iteration, and their union forms a connected horizontal cut $K$. Back on the graph $G$ the edges $e_1,\cdots,e_r$ are removed, so the second path $\ga_2$ replaces $\ga_1'$ with a new path of the form $\ga_2':=x^2_{k_1},\dots,x^2_{k_l},\dots,x^2_{k_L}$, where the first and last node belong to $A\cup B\cup \{\ga_1\}$ (depending on whether the face corresponding to $\ga_1'\cup\ga_2'$ abuts $A\cup B$ or not). In particular, we will have $h(x^2_{k_1})<h(x^2_{k_l})<\cdots<h(x^2_{k_L})$.  
Also, for every edge $e=x^2_j x^2_{j+1}$, $j=k_1,\dots,k_L-1$, we set $\tilde{h}(e)=m_1$.

Now, further subdivide the cut $K$ and add vertical segments of the form $\{h(x^2_{k_l})\}\times [m_1,m_1+m_2]$. Further, extend the vertical segments from the previous step that were not bounding an edge realizing the minimum. Finally, again close up the rectangles corresponding to the edges in $\ga_2$ that attain the minimum in step 3 of the algorithm. Then, we are ready to repeat this process.

Note that each face in the plane network $G$ corresponds in the tiling to a horizontal cut which is a union of sides of smaller rectangles. On the other hand, every node in $G$ gives rise to a vertical segment in the tiling, which is also a union of sides of smaller rectangles. These vertical segments may abut each other in some special cases, namely when $h$ has the same value at a node in the subpaths $\ga_1'$ and $\ga_2'$. That's when the ``$4$-fold corners'' situation mentioned above arises.
  
When the algorithm ends we will have packed the rectangle $[0,1]\times[0,\tilde{M}]$, where
\[
\tilde{M}:=\sum_{j=1}^n m_j.
\]
However, the fact that $\tilde{M}=M=\Mod_{2,\si} \Gamma(A,B)$ follows from Proposition \ref{prop:non-cross-min-subfam} (i).
\end{proof}
Both rectangle tilings in Figure \ref{fig:rectpack} and Figure \ref{fig:rect-pack} were obtained using code implementing  the non-crossing minimal subfamily Algorithm \ref{alg:non-cross-min-subfam} and Remark \ref{rem:rect-pack}.
Note that Remark \ref{rem:rect-pack} is stated for general planar graphs $G$. When, $G$ is the primal graph of an othodiagonal map, then the quantity $\tilde{h}(e)$ in the sketched proof can be computed using the harmonic function defined on the dual graph with $\tilde{h}=0$ on the bottom side and $\tilde{h}=M$ on the top side.

\subsection{A numerical example}\label{sec:numerical-example}

We have implemented the non-crossing minimal subfamily algorithm, and some of its consequences, numerically. Consider the family of paths connecting the arcs $ab$ and $cd$ of the cross domain on the left of Figure \ref{fig:cross_domain}.
\begin{figure}[H]
\includegraphics[width=0.33\textwidth]{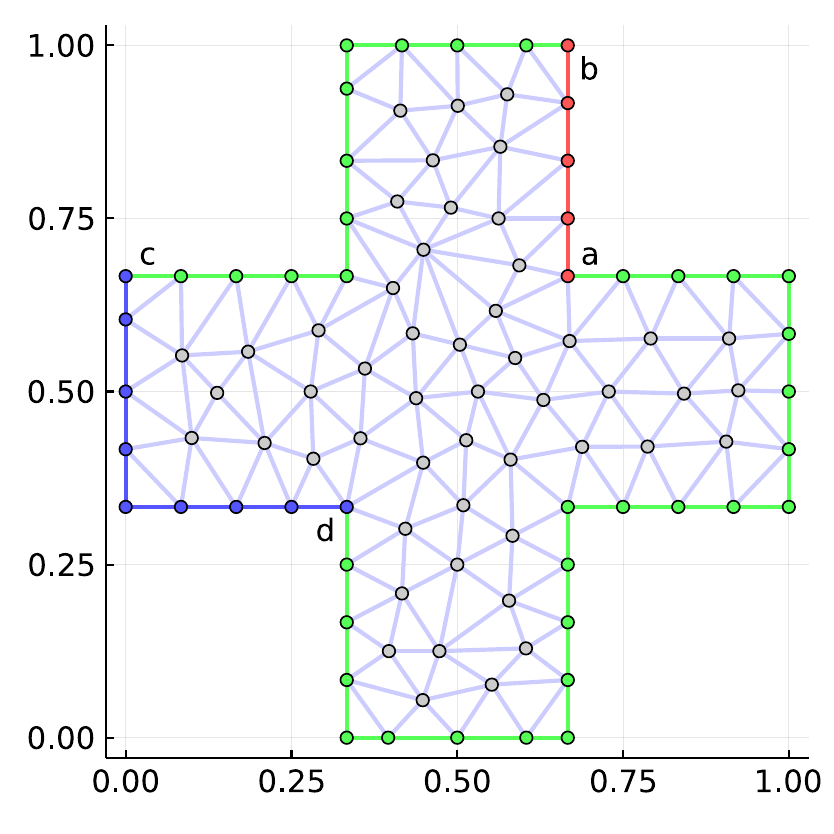}%
\includegraphics[width=0.33\textwidth]{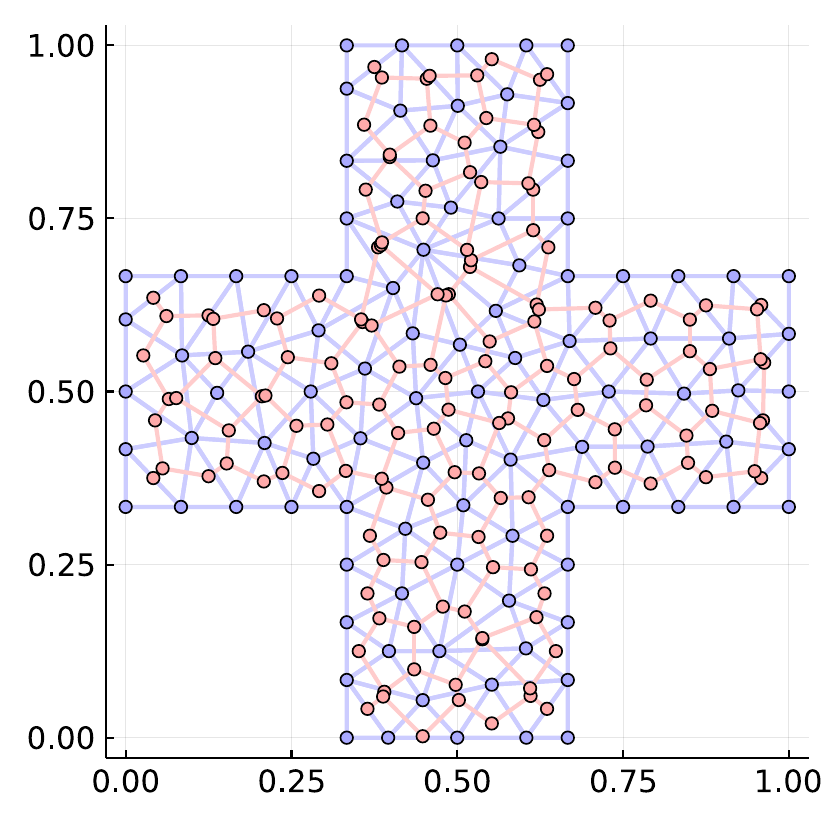}%
\includegraphics[width=0.33\textwidth]{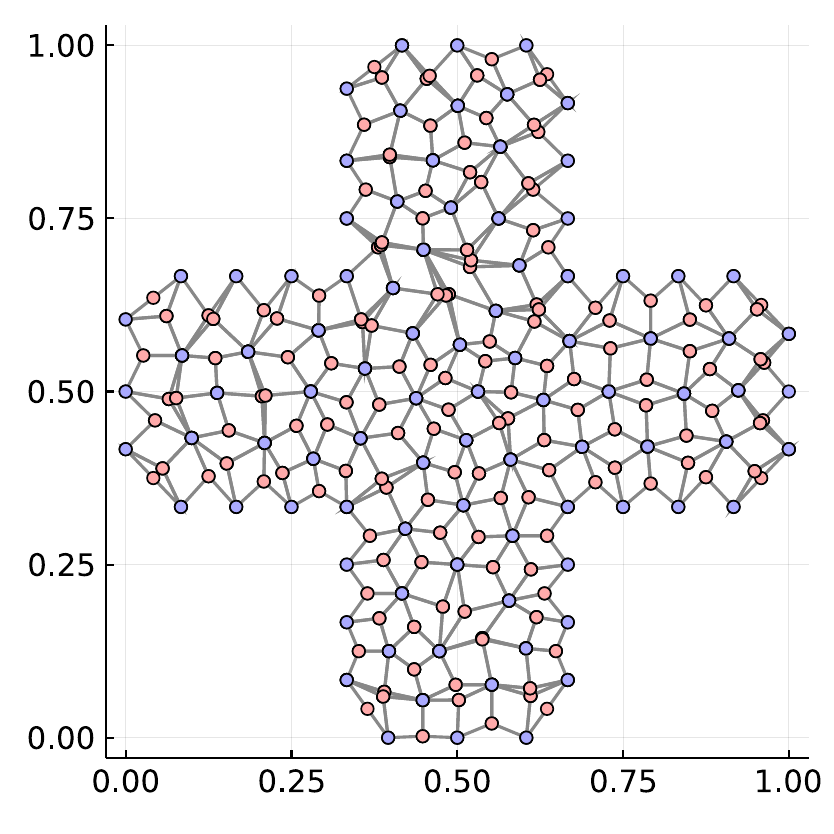}
\caption{Left: a Delaunay triangulation of a cross domain; Center: the triangulation shown together with its dual Voronoi graph; Right: the corresponding orthodiagonal map.}\label{fig:cross_domain}
\end{figure}
The figure shows a Delaunay triangulation of the domain, produced by the software package Triangle~\cite{shewchuk96b}.  The circumcenters of the triangles form a dual grid known as the  Voronoi diagram. In the center of Figure~\ref{fig:cross_domain}, we show the same triangulation along with its Voronoi dual.
The edges of the triangulation are blue, and the edges connecting two circumcenters are red. Note that the blue and red edges are orthogonal, hence can be thought as the orthogonal diagonals of a quadrilateral. In particular, an orthodiagonal map is obtained by connecting each blue node with a neighboring red node. We have drawn the corresponding orthodiagonal map on the right in Figure \ref{fig:cross_domain}. With some minor caveats (see below), the Delaunay triangulation can be thought as the primal graph of the orthodiagonal map and the Voronoi diagram as the dual graph. See also \cite[Section 5.1]{skopenkov:adv-math2013}.

\subsubsection{Approximating the modulus}

In order to approximate the modulus of the family of curves, we first choose a size parameter, $\epsilon>0$, which controls the size of the triangles in the triangulation.  The triangulation is constructed to satisfy the Delaunay property as well as two additional constraints: no triangle may have area larger than $\epsilon^2$ and no angle in any triangle may be smaller than 20$^\circ$.  These additional constraints ensure that the edges of the triangulation are of order $O(\epsilon)$, which is needed to ensure convergence of the discrete modulus to the continuous modulus.

Next, we need to assign appropriate conductances to the edges in the triangulation (Figure~\ref{fig:cross_domain}, left).  As stated above, the triangulation is \emph{almost} the graph that is obtained from the blue vertices of the quadrilateral mesh (Figure~\ref{fig:cross_domain}, right) by connecting vertices that belong to the same quadrilateral.  This does not quite reproduce the triangulation on the left;  notably absent are the edges on the boundary of the domain.  As a consequence, any vertex of the triangulation that is incident only upon boundary edges is absent from the quadrilateral mesh.  In the example shown, all convex ``corners'' are excluded from the quadrilateral mesh.  Fortunately, there is a relatively simple way to address these issues at the boundary.

Most edges of the triangulation will correspond to diagonals in the quadrilateral mesh.  For these, we can use the formulas specified in Section~\ref{orthodiagIntro}.  Namely, we assign the conductance, $\sigma$, on such an edge to be the ratio of the length of the dual edge to the length of the edge in question.  To all boundary edges, we assign $\sigma=\epsilon/|E_b|$, where $E_b$ is the set of boundary edges.  Since the harmonic function $h$ is bounded between $0$ and $1$, this means, by Ohm's law, that the contribution of all boundary edges to the current flow is $O(\epsilon)$, ensuring that the contribution of the extraneous boundary edges to the approximation is negligible as $\epsilon\to 0$.  With this choice of conductances, we can then solve the discrete Dirichlet problem to approximate the harmonic function $h$.

\subsubsection{Non-crossing paths}

Once the discrete harmonic function is computed, the non-crossing minimal family can be determined using Algorithm \ref{alg:non-cross-min-subfam}.  The left side of Figure~\ref{fig:cross-curves} shows the result of this numerical study for the choice $\epsilon=0.005$.  The algorithm produced a sequence of $34403$  non-crossing paths from the boundary segment $ab$ to the boundary segment $cd$.  We have plotted 21 of these paths, plotting roughly every $1720$-th path in the sequence.  The first and last path trace the boundary segments between $a$ and $d$ and between $b$ and $c$.  In the middle of Figure~\ref{fig:cross-curves}, we show the corresponding curves as computed by Driscoll's Schwarz–Christoffel Toolbox for Matlab.  Using the toolbox code, we approximated a conformal mapping between the cross and a rectangle with $ab$ and $cd$ sent to the vertical sides of the rectangle.  Inverting that map on horizontal lines produced the image.  The image on the right side of the figure shows the superposition of both sets of curves. 

\begin{figure}
\includegraphics[width=0.33\textwidth]{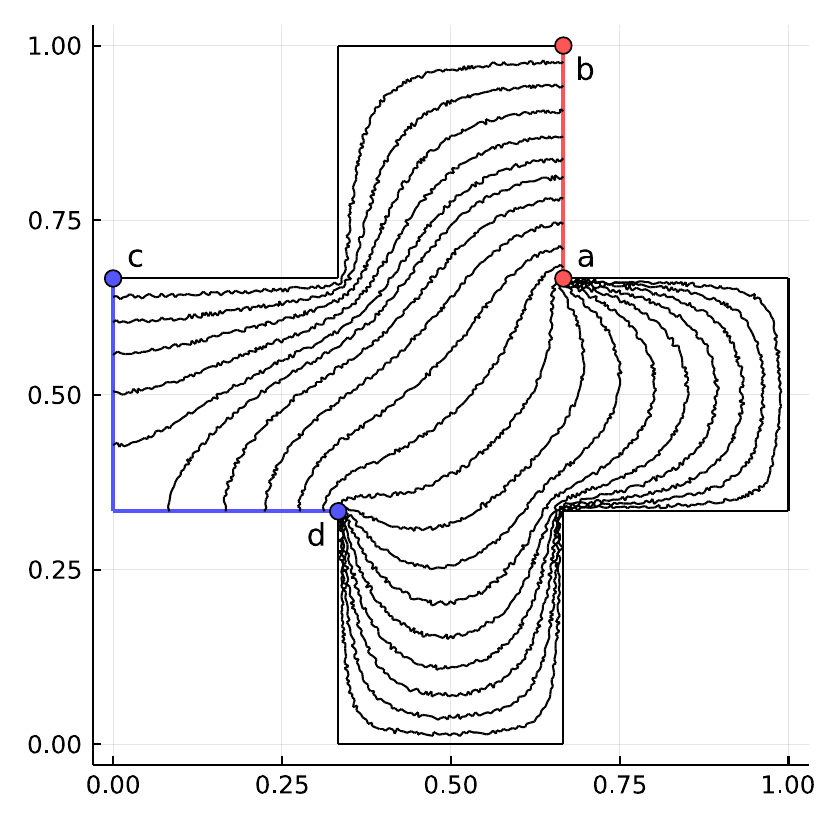}%
\includegraphics[width=0.33\textwidth]{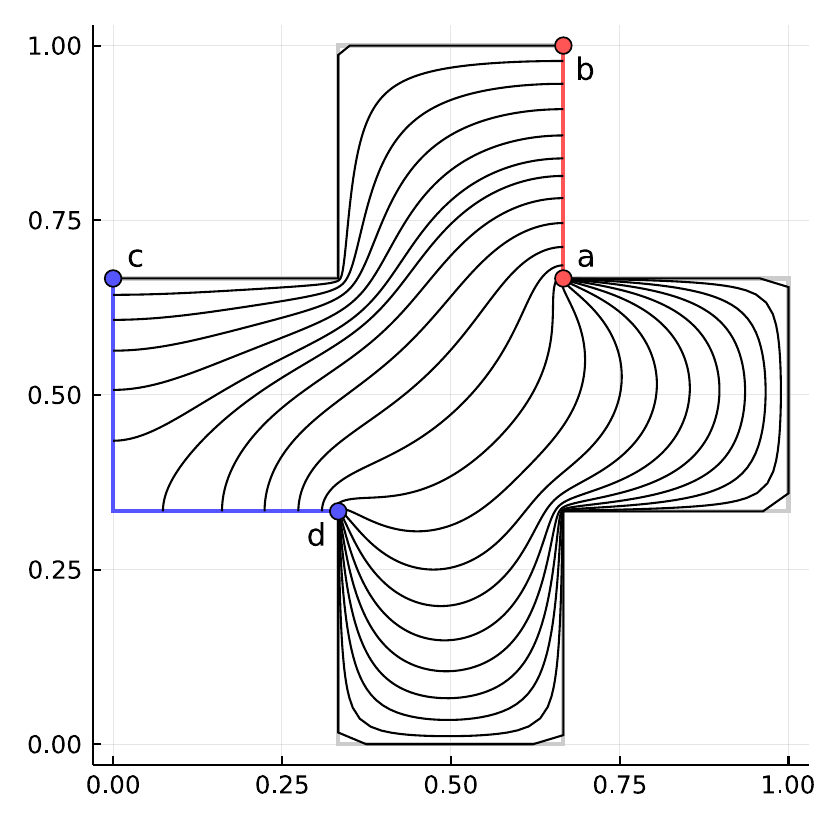}%
\includegraphics[width=0.33\textwidth]{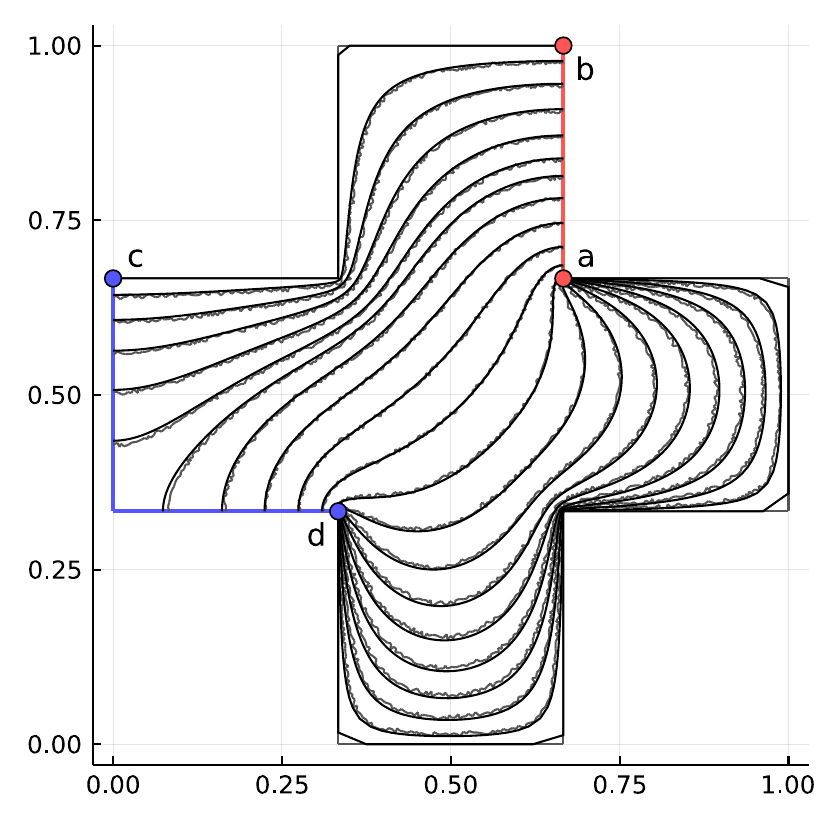}
\caption{Left: non-crossing paths for a triangulation of the cross domain; Center: the corresponding curvess as computed by the Schwarz-Christoffel Toolbox in Matlab; Right: the superposition of both sets of curves.}
\label{fig:cross-curves}
\end{figure}

Using the non-crossing path decomposition, we are able to assign a two-dimensional coordinate to each vertex in the graph.  The $x$-coordinate of each vertex is simply the value of the discrete harmonic function $h$ on that vertex.  To obtain the $y$-coordinate of a vertex $v$, we locate the first path in the sequence passing through $v$ and sum the flows of all previous paths.  The rectangle tiling for this choice of $\epsilon$ is too dense to draw, but the left side of Figure~\ref{fig:rect-pack} shows a rectangle tiling for $\epsilon=0.075$.  By Theorem \ref{HarmConvThm} and Theorem \ref{cor4}, this choice of coordinates provides a mapping on the vertices of the triangulation that approximates the conformal map to the rectangle.  The right side of Figure \ref{fig:rect-pack} shows the non-crossing paths in these coordinates (for the finer, $\epsilon=0.005$, triangulation).  Except near $a$ and $d$, where there is large distortion, these curves are nearly horizontal, as we would expect. Incidentally, the rectangle tiling (Figure \ref{fig:rect-pack}, left) gives an indication of where one may refine the original triangulation, since the largest distortion corresponds to very large rectangles in the tiling. 

\begin{figure}[H]
\includegraphics[width=0.5\textwidth]{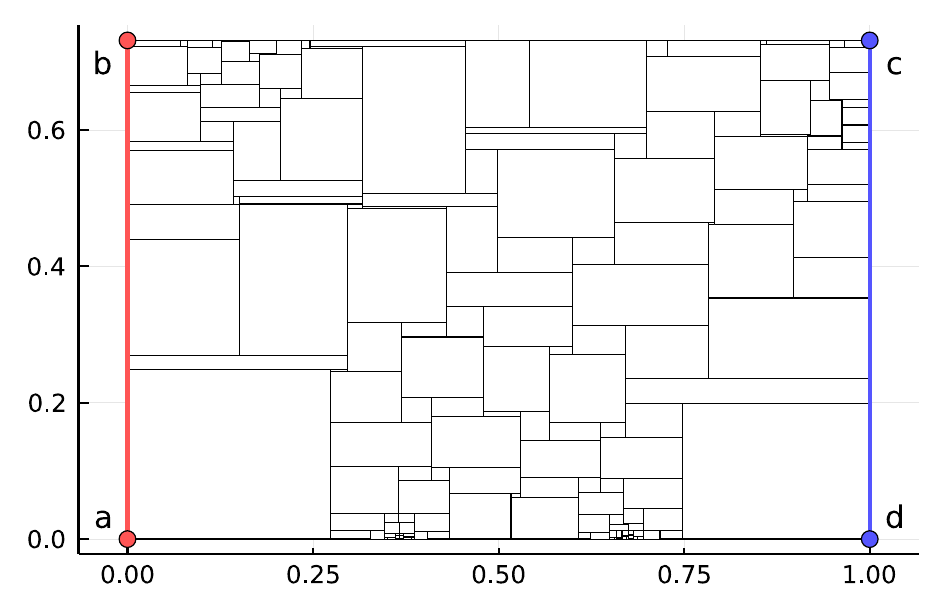}
\includegraphics[width=0.5\textwidth]{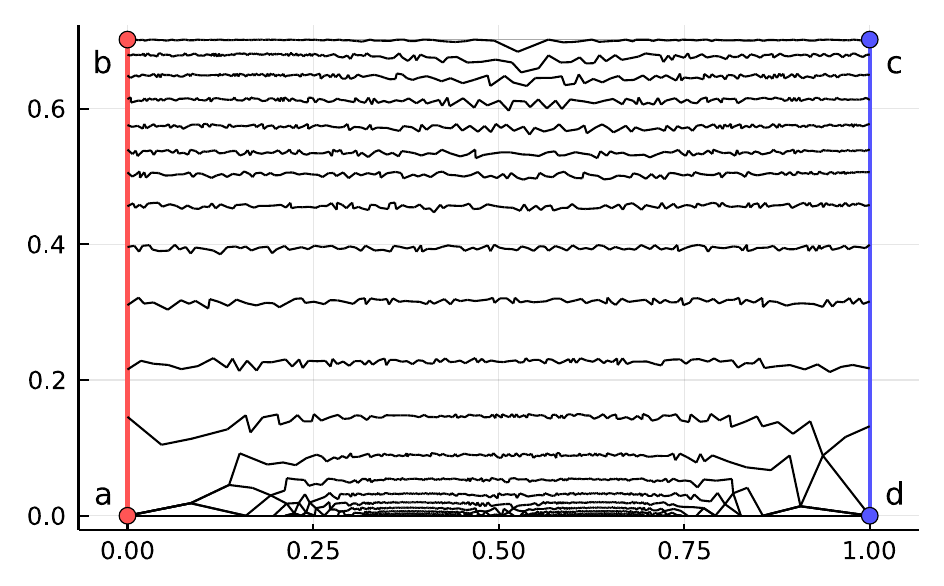}
\caption{Left: the rectangle tiling associated with a triangulation of the cross domain shown in Figure~\ref{fig:cross_domain}; Right: the images of the non-crossing paths from Figure~\ref{fig:cross-curves} in coordinates given by the rectangle tiling.}\label{fig:rect-pack}
\end{figure}

\section{Convergence of discrete harmonic functions}\label{harmonic}

In this section, we prove Theorem \ref{HarmConvThm} by adapting the arguments in \cite{gurel-jerison-nachmias:advm2020} to our setting and then we prove Theorem \ref{cor4} using Fulkerson duality for modulus (Lemma \ref{lem:fulkerson-duality}).
 
We begin by recalling our assumptions:
\begin{itemize}
\item The Jordan domain $\Omega$ is an analytic quadrilateral with boundary arcs $A, \tau_1, B, \tau_2$.  

\item The orthodiagonal map  $G=(V^\bullet \sqcup V^\circ, E)$ with boundary arcs $S_1, T_1, S_2, T_2$ 
has maximal edge length at most $\epsilon$ 
and approximates $\Omega$ within distance $\delta$. Recall that by Definition \ref{def:approx-smooth-quad}, we ask that $\overline{\Om}\subset \hat{G}$.

\item The continuous harmonic function $h_c$ is defined by  $h_c = \text{Re } \psi$, where $\psi$ is a conformal map of $\Omega$ onto the interior of the rectangle $R=[0,1] \times [0,m]$, taking the arcs $A$ and $B$ to the vertical sides of the rectangle.

\item  The function $h_d : V^\bullet \to \mathbb{R}$ is discrete harmonic on ${\rm int}(V^\bullet)$ with 
$h_d(z) = 0$ for $z \in S_1 \cap V^\bullet$ and $h_d(z) = 1$ for $z \in S_2 \cap V^\bullet$.

\item By scale invariance, we may assume diam$(\Omega) = 1$. 
\end{itemize}
Set $\lambda := \epsilon \vee \delta$.  Then under these assumptions, our goal is to show for all $z \in V^\bullet \cap \overline{\Omega}$
\begin{equation} \label{maingoal}
|h_d(z) - h_c(z)| \leq \frac{C}{\log^{1/2}\left(1/ \lambda \right)},
\end{equation}
where  $C$ depends only on the analytic quadrilateral $\Omega$
(i.e.~$C$ depends on $\Omega$ and the boundary arcs $A, \tau_1, B, \tau_2$).
Note that by assumption $\lambda < \diam(\Omega) =1$.
In our proof, we will show that this equation holds whenever $0<\lambda < C'$ for some $C'>0$ which depends only on $\Omega$.  
However, since $h_d$ and $h_c$ are both bounded between 0 and 1 (by the discrete and continuous maximum principles), we may adjust the constant $C$ in \eqref{maingoal} so that it holds for all $\l \in (0,1)$.

Lastly, we note that we will use $C$ to refer to constants that are independent of $\l$ but may depend on the analytic quadrilateral $\Omega$.  The exact value of these constants may change from line to line.
Further,  we will use  the notation $D_r(z)$ to mean the closed disk of radius $r$ centered at $z$.

\subsection{Modified toolbox}

We will be able to use some of the tools of \cite{gurel-jerison-nachmias:advm2020} as they are stated in that paper.  
However, we need to modify a handful of the results  to apply to our setting, and that is the goal of this section. Throughout this section, we will use the terminology of flows as described in Section \ref{sec:discrete-harm-fnct}.

There is one last piece of terminology that we need to introduce before we begin.
Let $G= (V^\bullet \sqcup V^\circ, E)$ be
a finite  orthodiagonal map 
with  boundary arcs $S_1, T_1, S_2, T_2$.
We say that $G_1$ is a {\it sub-orthodiagonal map of} $G$ when 
 $G_1$ is an othordiagonal map contained in $G$ and  
 the topological boundary $\partial G_1$ intersects $T_1\cup T_2$. 
The boundary
$   \partial V_1^\bullet$ are the points from  $V_1^\bullet$  that are in  $\partial G_1 \setminus \left(  T_1 \cup T_2 \right) $ and  the boundary $   \partial V_1^\circ$ are the points in $V_1^\circ \cap (T_1 \cup T_2)$.
Note that for $G_1$ to be a sub-orthodiagonal map, we do not require that $\partial G_1 \cap T_k$ is connected.
When we prove Theorem \ref{HarmConvThm} in Section \ref{HarmConvThmProof}, 
we will work with a sub-orthodiagonal map $G_1$ satisfying
$\dist(G_1, A\cup B) \geq \lambda^{1/4}$.  
With this assumption, Lemma  \ref{reflection} implies that for $\lambda$ small enough $\psi$ can be analytically extended to a domain containing $\tilde{G}_1$ so that $\dist(\hat{G}_1, \partial \mathcal{O}) \geq \frac{5}{16} \lambda^{1/4}$.
Since it may help to have a visual, see Figure \ref{SubOrthoPic} for a sketch of this setting.

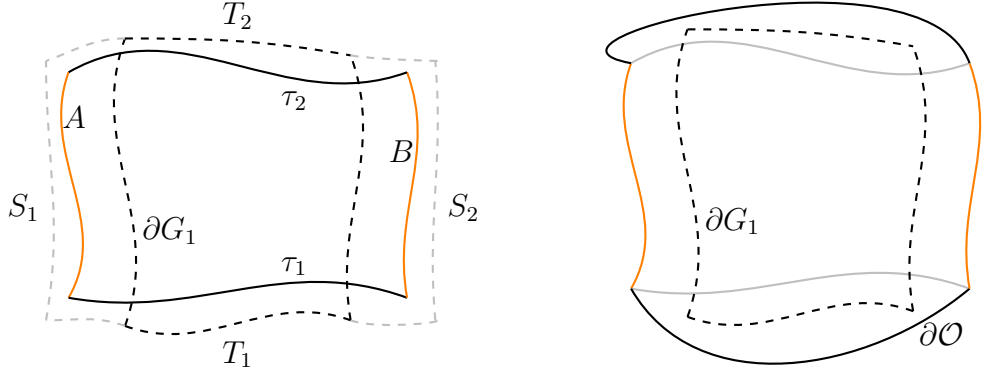
\begin{figure}
\centering
\begin{tikzpicture}[scale=1.5]

\draw[thick,orange] (0,0) to [out=60,in=-110] (0,2);
\draw[thick] (0,0) to [out=-10,in=160] (3,0);
\draw[thick,orange] (3,0) to [out=100,in=-70] (3,2);
\draw[thick] (0,2) to [out=30,in=200] (3,2);

\node at (0.05, 1.6) {$A$};
\node at (2.95, 1.3) {$B$};
\node at (2, 0.28) {$\tau_1$};
\node at (2, 1.75) {$\tau_2$};

\node at (-0.4, 0.8) {$S_1$};
\node at (3.5, 0.8) {$S_2$};
\node at (1.5, -0.5) {$T_1$};
\node at (1.5, 2.5) {$T_2$};

\draw[thick, dashed] (0.5,-0.25) to [out=70,in=-110] (0.5, 2.3);
\draw[thick, dashed] (0.5,-0.25) to [out=-20,in=160] (2.5, -0.2);
\draw[thick, dashed] (2.5, -0.2) to [out=100,in=-70] (2.5, 2.15);
\draw[thick, dashed] (0.5, 2.3) to [out=0,in=170] (2.5, 2.15);

\node at (0.9, 0.6) {$\partial G_1$};

\draw[thick, dashed, lightgray] (0.5,-0.25) to [out=160,in=0] (-0.2, -0.2);
\draw[thick, dashed, lightgray] (0.5,2.3) to [out=180,in=20] (-0.2, 2.1);
\draw[thick, dashed, lightgray] (-0.2,-0.2) to [out=80,in=-90] (-0.2, 2.1);
\draw[thick, dashed, lightgray] (2.5, -0.2) to [out=-20,in=0] (3.2, -0.2);
\draw[thick, dashed, lightgray] (2.5, 2.15) to [out=-10,in=0] (3.2, 2.1);
\draw[thick, dashed, lightgray] (3.25, -0.2) to [out=95,in=-85] (3.25, 2.1);

\node at (0,-0.85){ ~};

\end{tikzpicture}
\begin{tikzpicture}[scale=1.5]

\draw[thick,orange] (0,0) to [out=60,in=-110] (0,2);
\draw[thick, lightgray] (0,0) to [out=-10,in=160] (3,0);
\draw[thick] (0,0) to [out=-60,in=220] (3,0);
\draw[thick,orange] (3,0) to [out=100,in=-70] (3,2);
\draw[thick, lightgray] (0,2) to [out=30,in=200] (3,2);
\draw[thick] (0,2) to [out=170,in=110] (3,2);

\node at (2.75, -0.4) {$\partial \mathcal{O}$};
\node at (0.9, 0.6) {$\partial G_1$};

\draw[thick, dashed] (0.5,-0.25) to [out=70,in=-110] (0.5, 2.3);
\draw[thick, dashed] (0.5,-0.25) to [out=-20,in=160] (2.5, -0.2);
\draw[thick, dashed] (2.5, -0.2) to [out=100,in=-70] (2.5, 2.15);
\draw[thick, dashed] (0.5, 2.3) to [out=0,in=170] (2.5, 2.15);

\end{tikzpicture}

\caption{
A sketch of the setting encountered in the proof of Theorem  \ref{HarmConvThm}.  
Left: An analytic quadrilateral $\Omega$ with boundary arcs $A, B$ (orange) and $\tau_1, \tau_2$ (black)
is approximated by an orthodiagonal map $G$ 
with  boundary arcs $S_1, T_1, S_2, T_2$.
The boundary of the sub-orthodiagonal map $G_1$ is dashed black and $\partial G \setminus \partial G_1$ is dashed gray. 
Right: $G_1$ is contained within the extended domain $\mathcal{O}$ of Lemma \ref{reflection}.
}\label{SubOrthoPic}
\end{figure}

We begin with a modified version of Proposition 5.2 of \cite{gurel-jerison-nachmias:advm2020}, which tells us that the energy of $h_c-h$ is small, for $h_c$ continuous harmonic and $h$ discrete harmonic with matching boundary values.  Note that we we will be working with a sub-orthodiagonal map $G_1$, but for simplified notation, we use the notation $V$ and $E$ for the vertices and edges of $G_1$.

\begin{prop}\label{modified5.2}
Assume that $\Omega, \psi, h_c = \Re \psi$, and $G$ are defined as in  Theorem \ref{HarmConvThm}, 
assume $G_1=(V^\bullet \sqcup V^\circ, E)$ is a sub-orthodiagonal map of $G$,
and assume
$\psi$ can be analytically extended to an open set $\mathcal{O}$ containing $\tilde{G}_1$.
Let $h : V^\bullet \to \mathbb{R}$ be discrete harmonic on ${\rm int}(V^\bullet)$ and satisfy $h = h_c$ on $\partial V^\bullet$.
Set $$L = || \nabla h_c||_{\infty, \hat{G}_1} \;\; \text{ and } \;\;M = || H h_c ||_{\infty, \tilde{G}_1}.$$
Then 
$$\mathcal{E}_1^\bullet(h_c - h) \leq C \left[ M^2 \epsilon^2 + L^2(\epsilon \vee \delta) \right],$$
where $C$ depends only on  $\area(\hat{G})$ and $\Length(\tau_1 \cup \tau_2)$ and $\mathcal{E}_1^\bullet$ is the energy in the graph $G_1^\bullet$.

\end{prop}

\begin{proof}
We begin by following the initial steps in the proof of Proposition 5.2 in \cite{gurel-jerison-nachmias:advm2020}.
In particular,
we wish to define a bijective map $e \to e^\dagger$ from $\vv{E}^\bullet$ to $\vv{E}^\circ$.
If an edge $e=xy$ is oriented from $x$ to $y$, we define $e^- = x$ and $e^+=y$.
For $e \in  \vv{E}^\bullet$, let $Q = [v_1, w_1, v_2, w_2]$ be the inner face of $G_1$ containing $e$.  
We set $e^\dagger = e^\circ_Q$, and we choose its orientation based on the orientation of $e$:
if $e^- = v_1$, then choose the orientation  with $(e^\dagger)^- = w_1$, and in the other case choose the opposite orientation.

We wish to define a flow $\theta$ on $G^\bullet_1$ using the harmonic conjugate $\tilde h_c$  of $h_c$ (i.e. $\tilde h_c = \text{Im}\, \psi$).  However, we must take a little more care near the boundary arcs $\tau_k$.
To this end, let $U_k$ be the component of $\hat{G}_1 \setminus \Omega$ that intersects $\tau_k$, for $k=1,2$, see Figure \ref{SubOrthoPic}.
Next, we define a projection $p: V^\circ\rightarrow \overline{\Om}$ as follows:
for $w \in V^\circ\cap\overline{\Om}$, set $p(w) = w$;
for $w \in U_k \cap V^\circ$, set $p(w)$ to be a point in $\tau_k$ that is closest to $w$.
For $e \in \vv{E}^\bullet$,
 define
$$\theta(e) =\tilde h_c( p((e^\dagger)^+)) - \tilde h_c( p((e^\dagger)^- )).$$
Note that if $e^\dagger$ is not incident to a vertex in $U_1 \cup U_2$, then our definition gives
$$\theta(e) =\tilde h_c( (e^\dagger)^+ ) - \tilde h_c( (e^\dagger)^- )$$
matching the definition in \cite{gurel-jerison-nachmias:advm2020}.  

Now we want to show that $\theta$ satisfies the node law on ${\rm int}(V^\bullet)$.
Let $v \in {\rm int}(V^\bullet)$.
Suppose first $v$ is contained within an interior face $f_v$ of the dual graph $G^\circ_1$, i.e. there is a cycle of dual edges surrounding $v$.
Let $w_1, w_2, \cdots, w_k$ be the vertices incident to $f_v$ in counterclockwise order.
Then (taking indices mod $k$)
$$ \sum_{e\in \vv{E}^\bullet \, :  \,e^- = v} \theta(e) = \sum_{j=1}^k \left[\tilde h_c(p(w_{j+1})) - \tilde h_c(p(w_j)) \right] = 0.$$
If $v$ is not in an interior face $f$, then the edges $e^\dagger$ such that $e\in \vv{E}^\bullet$ with  $e^- = v$ form a path with both endpoints in $\partial V^\circ$ (and in fact both endpoints will be in the same boundary arc $T_1$ or $T_2$, see Figure \ref{SubOrthoPic}). 
Let $w_1, w_2, \cdots, w_k$ be the vertices of this path, indexed so that connecting $w_k$ to $w_1$ in the outer face of $G_1^\circ$ would give a counterclockwise cycle.
Then
$$ \sum_{e\in \vv{E}^\bullet \, :  \,e^- = v} \theta(e) =
\sum_{j=1}^{k-1} \left[\tilde h_c(p(w_{j+1})) - \tilde h_c(p(w_j)) \right] =  \tilde h_c(p(w_k)) - \tilde h_c(p(w_1)) = 0,$$
where we used  that $\tilde h_c = \text{Im} \, \psi$ is constant on $\tau_k$.

Since $\theta$ is a flow on ${\rm int}(V^\bullet)$ and $h$ is discrete harmonic with $h_c = h$ on $\partial V^\bullet$, we may apply Proposition 4.9 in \cite{gurel-jerison-nachmias:advm2020} with $f=h_c$  to obtain that
$$\mathcal{E}_1^\bullet(h_c - h) \leq \mathcal{E}_1^\bullet(h_c - h) + \mathcal{E}_1^\bullet(\theta - \sigma \,d h) 
              = \mathcal{E}_1^\bullet(\phi - \theta),$$
where we define $$\phi(e) := \sigma(e)[h_c(e^+) - h_c(e^-)].$$              
Hence, it suffices to bound 
  $         \mathcal{E}_1^\bullet(\phi - \theta).$      
The contribution to $ \mathcal{E}_1^\bullet(\phi - \theta)$ from a face $Q=[v_1, w_1, v_2, w_2]$ is 
 \begin{equation*}
 \frac{|v_1v_2|}{|w_1w_2|} \left( \frac{|w_1w_2|}{|v_1v_2|} [h_c(v_2) - h_c(v_1)] - [\tilde h_c(p(w_2)) - \tilde h_c(p(w_1))]\right)^2  \;\;\;\;\;\;\;\;\;\;\;\;
 \end{equation*}
 \begin{equation}  \label{RHS}
 \;\;\;\;\;\;\;\;\;\;\;\; = 
|v_1v_2| \cdot |w_1w_2| \cdot \left(\frac{h_c(v_2) - h_c(v_1)}{|v_1v_2|} - \frac{\tilde h_c(p(w_2)) - \tilde h_c(p(w_1))}{|w_1w_2|} \right)^2.
 \end{equation}
Therefore, assuming that neither $w_1$ nor $w_2$ are in $U_1 \cup U_2$,  the quantity in the parentheses (using the Cauchy-Riemann equations) can be rewritten as
 \begin{equation} \label{ugly}
  \frac{h_c(v_2) - h_c(v_1) -  \langle \nabla h_c(q), {\bf v} \rangle |v_1v_2| }{|v_1v_2|} 
 - \frac{\tilde h_c(w_2) - \tilde h_c(w_1)   - \langle \nabla \tilde h_c(q), {\bf w} \rangle |w_1w_2| }{|w_1w_2|}, 
\end{equation}
  where ${\bf v}$ and ${\bf w}$ are the unit vectors in the directions of $\vv{v_1v_2}$ and $\vv{w_1w_2}$, respectively, and $q \in Q$ is the intersection point  of $e_Q^\bullet$ and $e_Q^\circ$.
 Note that  $ || H \tilde h_c ||_{\infty, \tilde{G}_1} =  || H h_c ||_{\infty, \tilde{G}_1} =M$ by the Cauchy-Riemann equations. 
Hence both terms in \eqref{ugly} can be bounded using Lemma 5.3 in \cite{gurel-jerison-nachmias:advm2020} to obtain
\begin{align*}
 &\left| \frac{h_c(v_2) - h_c(v_1) -  \langle \nabla h_c(q), {\bf v} \rangle |v_1v_2| }{|v_1v_2|} \right| \leq 2M \epsilon, \\
 &\left| \frac{\tilde h_c(w_2) - \tilde h_c(w_1)   - \langle \nabla \tilde h_c(q), {\bf w} \rangle |w_1w_2| }{|w_1w_2|} \right| \leq 2M\epsilon.
 \end{align*}
Therefore if $w_1, w_2$ are not in $U_1 \cup U_2$, the contribution to $\mathcal{E}^\bullet ( \phi - \theta)$ is at most 
$$|v_1v_2|\cdot|w_1w_2| \cdot (4M\epsilon)^2 = 32 \,\area(Q) M^2 \epsilon^2. $$

Here, however, we need to deal with some extra cases. Assume that at least one of $w_1$ and $w_2$ are in $U_1 \cup U_2$.  
We will bound the terms in the parentheses of \eqref{RHS} directly.  Note first that
$$\left| \frac{h_c(v_2) - h_c(v_1)}{|v_1v_2|} \right| \leq L.$$
If $w_1$ and $w_2$ are both in $U_k$, then $ \tilde h_c(p(w_2)) - \tilde h_c(p(w_1)) =0$.
If $w_1$ is in $U_k$, but  $w_2$ is not,
then the line segment connecting $w_1$ to $w_2$ must cross $\tau_k$ and $|p(w_1) - w_1 | \leq |w_1-w_2|$, by the definition of $p$.  
(And a similar statement is true if $w_2$ is in $U_{k+1}$, taking the indices mod 2).
Then 
$$|p(w_1) - p(w_2)| = |p(w_1) - w_1| + |w_1 - w_2| + |w_2 - p(w_2)| \leq 3|w_1 - w_2|,$$
which implies that 
$$  \left|  \frac{\tilde h_c(p(w_2)) - \tilde h_c(p(w_1))}{|w_1w_2|} \right|  \leq L \frac{|p(w_2) - p(w_1)|}{|w_2- w_1|} \leq 3L. $$

Thus if $Q$ is incident to a vertex in $U_1 \cup U_2$, then the contribution to $\mathcal{E}_1^\bullet ( \phi - \theta)$ is at most 
$$ |v_1v_2|\cdot|w_1w_2| \cdot ( 4L)^2 \leq 32 \,\area(Q) L^2  .$$
The faces $Q$ that are incident to a vertex in $U_1 \cup U_2$ must be in a band containing $\tau_k$ for either $k=1$ or $k=2$, and the area of these bands is bounded by $8  \, \Length(\tau_1 \cup \tau_2) (\epsilon \vee \delta)$.
Thus summing over all inner faces $Q$ of $G_1$ gives
$$ \mathcal{E}_1^\bullet ( \phi - \theta) \leq 32 \,\area(\hat G) M^2 \epsilon^2 + 256 \, \Length(\tau_1 \cup \tau_2) L^2  (\epsilon \vee \delta). $$

\end{proof}

The next proposition generalizes Proposition 6.1 from \cite{gurel-jerison-nachmias:advm2020} to our setting.  

\begin{prop}\label{modified6.1}
Assume that $\Omega, \psi$ and $G$ are defined as in the setting of Theorem \ref{HarmConvThm}, and let $\lambda = \e \vee \delta$. Let $G_1=(V^\bullet \sqcup V^\circ, E)$ be a sub-orthodiagonal map of $G$ with $\dist(G_1, A\cup B) \geq \lambda^{1/4}$. 
Fix $z_0 \in {\rm int}(V^\bullet)$, and let $r \geq 4 \l$. 
Assume that $D_r(z_0) \cap \partial V^\bullet = \emptyset$.
Then there is a unit flow $\theta$ in $G_1^\bullet$ from $V^\bullet \cap D_r(z_0)$ to $\partial V^\bullet$ so that
$$\mathcal{E}_1^\bullet(\theta) \leq C \log \left(  \frac{\diam(\Omega)}{\l }\right),$$
for $\l \le \l_0$, where $C$ and $\l_0$ depend only on the analytic quadrilateral $\Omega$.
\end{prop}

To prove this proposition, we would like to define the flow over an edge to be the change in argument over its dual edge  (taking care with the edge orientations).  Then take a fixed vertex $v$ in the topological interior, and consider the collection of all the edges incident to $v$; their dual edges form a cycle around $v$.  Thus, the sum of the flow over these edges will be the change in argument over the cycle of dual edges, which will be zero, as long as the cycle does not contain the origin.  However, if $v$ is on $T_1$ or $T_2$, then instead of a cycle, we will obtain a path, and the beginning and ending argument may not be equal.  To compensate for this, we will first transform the region conformally.

\begin{proof}

Recall that $\psi$ is a conformal map of $\Omega$ onto the rectangle $R=[0,1] \times [0,m]$ taking the arcs $\tau_1$ and $\tau_2$ to the horizontal sides of $R$.
Further by Lemma \ref{reflection}, we may take $\lambda$ small enough so that  $\psi$ can be analytically extended to an open set $\mathcal{O}$ containing $\hat{G}_1$ with  $\dist(\hat{G}_1, \partial \mathcal{O}) \geq \frac{5}{16} \la^{1/4} \geq 4\l$.

If $z_0 \notin \overline{\Omega}$, then there is a point $u_0 \in \overline{\Omega}$ with $|u_0 -z_0| < \delta$.  
Further by adjusting $u_0$ slightly, we may assume that $u_0$ is not on any of the dual edges.  
Then $u_0$ lies in one of the dual faces, and subsequently, there is $v_0 \in V^\bullet$ in the same face with $|u_0-v_0| < \epsilon.$
If $z_0 \in \overline{\Omega}$, we take $u_0=v_0 = z_0$.

Consider two triangles in $R$: one with vertices $\psi(u_0), 0, i m$ and the other with vertices $\psi(u_0), 1, 1+i m$.  
Shift both of these triangles by $-\psi(u_0)$, so they both have a corner at 0, and let
$P$ be the union of the regions bounded by these shifted triangles. 
Let $\kappa$ be the sum of the angles at 0 in these shifted triangles, and note that $\kappa$ is bounded below independent of $z_0$ by a constant depending only on $\Omega$. 
(In fact, one can show that $\kappa \geq \tan^{-1}(m)$.)
Let $F(z) = \psi(z) - \psi(u_0)$.
We will define a flow $\phi$ on  $\vv{E}^\bullet$ from $v_0$ to $\partial V^\bullet$ as follows.  
(Recall that   $\vv{E}^\bullet$  is the set of directed edges, which contains both possible orientations for all the edges of $E^\bullet$.)    
Given a face $Q=[v_1, w_1, v_2, w_2]$, let $e$ be the orientation of $e_Q^\bullet$ with $e^- = v_1$.
If $F(e_Q^\circ) \subset \overline{P}$, set
$$\phi(e) = \arg(F(w_2))-\arg(F(w_1)) \;\;\; \text{ and } \;\;\;   \phi(-e) = \arg(F(w_1))-\arg(F(w_2)),$$
where we use a branch of the argument that is well-defined on $P$.
If $F(e_Q^\circ)$ does not intersect $P$, then we define $\phi(e) = 0 = \phi(-e)$.
If $F(e_Q^\circ)$ intersects $P$ and $\C\setminus\overline{P}$, we define $\phi(e)$ to be the change of argument along $F(e_Q^\circ) \cap \overline{P}$, respecting the orientation above.

Let $v \in {\rm int}(V^\bullet)$. 
Consider the set $N$ of edges  $e\in \vv{E}^\bullet$ with  $e^- = v$. The edges dual to the edges in $N$ either form a cycle around $v$ (in the case $v$ is in the topological interior of $\hat G_1$) or a path starting and ending on a boundary arc $T_k$.
In the latter case, we can form a cycle by connecting the endpoints of the path with an arc that is in the complement of $\hat{G}_1$ but still in the domain of $\psi$.
Now in both cases, the image of the cycle under $F$ forms a cycle in $\C$, which we will call $\gamma_v$.
When $v=z_0$, then $\gamma_v$ winds once around the origin, and in all other cases, the winding number of $\gamma_v$ around the origin is zero.
Since
$ \sum_{e\in \vv{E}^\bullet \, :  \,e^- = v} \phi(e) $
gives the sum of the change in argument over the path(s) of $\gamma_v \cap \overline{P}$, this is equal to $\kappa$ when $v=v_0$ and it is equal to 0 otherwise.  
Thus $\phi$ is a flow of strength $\kappa$ from $v_0$ to $\partial V^\bullet$.
Define 
$$\theta(e) = \begin{cases}
	0					&\quad\text{if } e^+ \text{ and } e^- \in D_r(z_0)\\  
	\frac{1}{\kappa} \phi(e)           & \quad\text{else}\\                      
      \end{cases}.$$
Since $v_0 \in D_r(z_0)$, then $\theta$ is a unit flow from from $V^\bullet \cap D_r(z_0)$ to $\partial V^\bullet$.

It remains to bound the energy of $\theta$.
Let $Q=[v_1, w_1, v_2, w_2]$ be a face of $G_1$ so that $v_1, v_2$ are not both in $D_r(z_0)$.
We first consider the case that $F(e_Q^\circ)$ is contained in $\overline{P}$.
  Then the contribution of $e_Q^\bullet$ to $\mathcal{E}_1^\bullet(\theta)$ is
$$\frac{|v_1v_2|}{|w_1w_2|} |\theta(e^\bullet_Q)|^2 = \frac{1}{\kappa^2} |v_1v_2|\cdot |w_1w_2| \left( \frac{\arg(F(w_2)) - \arg(F(w_1))}{|w_1w_2|} \right)^2.$$
Note that $|v_1v_2|\cdot |w_1w_2| = 2 \,\text{area}(Q)$.
Since $Q$ has  diameter at most $2\epsilon$, then $Q$ is contained in a disc $D_t(q)$ with $t=2\epsilon$.
Now $\dist(D_t(q), u_0) \geq \dist(D_t(q), z_0) - |z_0-u_0| \geq 2\lambda-\lambda = \lambda$.
Further since $\dist(\hat{G}_1, \partial \mathcal{O}) \geq 4\lambda$, then $\dist(D_t(q), \partial \mathcal{O}) \geq 2\lambda$.
Hence $Q$ is contained in a disk $D_t(q)$ with the property that  $D_{1.5t}(q)\subset\mathcal{O} \setminus \{u_0\}$. 
Hence, by the Koebe distortion theorem   \cite[Theorem I.4.5 and Equation I.(4.18)]{garnett-marshall2008}, the derivative of $\log \circ F$ in $Q$ is comparable (with absolute constants) to any one of its values in $Q$.
Thus
\begin{align*}
\frac{|v_1v_2|}{|w_1w_2|} |\theta(e^\bullet_Q)|^2 
&=\frac{2}{\kappa^2} \,\text{area}(Q) \left( \frac{\arg(F(w_2)) - \arg(F(w_1))}{|w_2-w_1|} \right)^2 \\
&\leq \frac{2}{\kappa^2} \, \text{area}(Q)  \max_{z \in e^\circ_Q} \left| \frac{d}{dz} \log \circ F (z)\right|^2 \\
&\leq C_0
 \int_Q \frac{1}{|F(z)|^2} |F'(z)|^2 \, dA \\
&= C_0 \int_{F(Q)} \frac{1}{|z|^2} \, dA 
\end{align*}
for a constant $C_0$ depending only on $\Omega$.

If $F(e_Q^\circ)$ intersects  $\overline{P}$ but is not fully contained in it, the above estimate still holds, since there are points $w_1', w_2'$ on the edge $e_Q^\circ$ so that the contribution of 
$e_Q^\bullet$ to $\mathcal{E}_1^\bullet(\theta)$ is
$$ \frac{2}{\kappa^2} \text{area}(Q) \left( \frac{\arg(F(w_2')) - \arg(F(w_1'))}{|w_1w_2|} \right)^2 \leq
 \frac{2}{\kappa^2} \text{area}(Q) \left( \frac{\arg(F(w_2')) - \arg(F(w_1'))}{|w_2'-w_1'|} \right)^2.$$

For $|z-u_0| = \lambda$, we have 
$$|F(z)| \geq \frac{\lambda}{4}  |F'(z_0)| =  \frac{\lambda}{4}  |\psi'(z_0)| \geq \frac{\lambda}{16} \frac{\text{dist}(\psi(z_0), \partial \psi(\mathcal{O}) )}{\diam(\cO)}\geq  \frac{\lambda}{64}   \frac{\text{dist}(\psi(z_0), \partial \psi(\mathcal{O}) )}{\diam(\Omega)}.$$
The first inequality follows from the Koebe One-Quarter Theorem \cite[Theorem I.4.1, Equation I.(4.9)]{garnett-marshall2008}. 
The next  inequality is another application of the Koebe One-Quarter Theorem  \cite[Corollary I.4.4]{garnett-marshall2008}. And the last inequality follows because without loss of generality we can assume that $\diam(\cO)\le 4\diam(\Om)$, since every point in $\hat{G}_1$ is within distance $\de$ of $\Om$ by assumption.
Let $w_0 \in \partial \mathcal{O}$ so that $\text{dist}(\psi(z_0), \partial \psi(\mathcal{O}) ) = |\psi(z_0) - \psi(w_0)|$.
Then by Lemma \ref{reflection}, there is $b >0$ so that
$$ \text{dist}(\psi(z_0), \partial \psi(\mathcal{O}) ) = |\psi(z_0) - \psi(w_0)| \geq C_1 | z_0 - w_0 |^{b} 
\geq C_1 \left( \text{dist}(z_0, \partial \mathcal{O}) \right)^{b} \geq C_1 \lambda^{b}, $$
where $C_1$ is a constant that depends only on the domain $\Om$.
Therefore we have that for $|z-u_0| = \lambda$, then 
$ |F(z)| \geq \frac{C_1\lambda^{1+b} }{64\, \diam(\Omega)} =: d_1$.
This implies $F(Q)$ does not intersect a disc of radius $d_1$ centered at 0.
On the other hand, from Lemma \ref{reflection}, 
$$|F(z)| \leq C\left(\diam(\hat G_1) \right)^a \leq C\left(\diam( \Omega) \right)^a =:d_2 \;\; \text{ for } z \in \hat{G}_1,$$
where the constants $C$ and $a$ depend only on $\Omega$.
Putting this all together, we find that 
\begin{align*}
\mathcal{E}_1^\bullet(\theta) 
  &\leq C_0 \sum_{Q: \text{dist}(Q, u_0) \geq \lambda} \int_{F(Q)} \frac{1}{|z|^2} \, dA \\
  &\leq C_0 \int_0^{2\pi} \int_{d_1}^{d_2} \frac{1}{s^2} s \,ds \, dt \\
  &\leq C \log\left( \frac{\diam(\Omega)}{\lambda } \right),
\end{align*}
for $\l\le\l_0$, where $C$ and $\l_0$ only depend on $\Om$.

\end{proof}

We need modified versions of the three results in \cite[Section 7]{gurel-jerison-nachmias:advm2020}. 
The logical relationship between these results is the following:
$$ \text{Lemma 7.3} \implies \text{Lemma 7.2} \implies \text{Proposition 7.1}.$$
Since these results make use of the fact that the boundary value problem, in their case, is defined on the whole topological boundary, we will need to modify them.  This will have a ripple effect, causing us to update all three results, although the proofs connecting the results will follow the logic of \cite{gurel-jerison-nachmias:advm2020}. 

We begin with stating the modification of \cite[Proposition 7.1]{gurel-jerison-nachmias:advm2020}:

\begin{prop}\label{modified7.1}

Let $G$ 
be a finite orthodiagonal map with  boundary arcs $S_1, T_1, S_2, T_2$ 
with maximal edge length at most $\e$, and let $G_1=(V^\bullet \sqcup V^\circ, E)$ either be equal to $G$ or be a sub-orthodiagonal map of $G$.
Let $x, y \in V^\bullet$.
Assume we can choose $r, R$ so that $r \geq \frac{1}{2}|x-y|$, $R \geq 2r+ 3\epsilon$,
and that there is a path from $x$ to $y$ in $G_1^\bullet \cap D_{r}(z)$, where 
 $z=\frac{1}{2}(x+y)$. 
Further, assume that  
between any two vertices in $\partial V^\bullet$, there is a path $\gamma$  in $\partial G_1^\bullet$ so that $\gamma \cap D_R(z) $ only contains vertices in  $\partial V^\bullet$.

Let $h : V^\bullet \to \mathbb{R}$ be discrete harmonic on ${\rm int}(V^\bullet)$.
Set
$$ \kappa = \max_{v, v' \in \partial V^\bullet \cap D_R(z)} |h(v)-h(v')|,$$
with $\kappa=0$, if $\partial V^\bullet \cap D_R(z)=\emptyset$.
Then there exists a universal constant $C < \infty$ such that
$$| h(x) - h(y) | \leq \frac{C\, \mathcal{E}_1^\bullet(h)^{1/2}}{\log^{1/2}\left[ R/(r+\epsilon) \right]} + \kappa.$$
\end{prop}

The proof of this result is the same as the proof of Proposition 7.1 in \cite{gurel-jerison-nachmias:advm2020}, simply substituting Lemma \ref{modified7.2} in the place of  \cite[Lemma 7.2]{gurel-jerison-nachmias:advm2020}.  However, for the convenience of the reader, we include a proof below, after establishing Lemma \ref{modified7.2}. 
The rough idea is that for a lower bound on the energy, we must be able to find a thick family of paths connecting the red dashed line to the blue dashed line in Figure \ref{PlaneLemmaPic}, and this is what Lemma \ref{modified7.2} does.
For this, we introduce one more definition from \cite{gurel-jerison-nachmias:advm2020}.
For $\rho>0$, a {\it $\rho$-edge} of $G^\bullet$ is an edge  $ e \in E^\bullet$ whose dual edge $e^\circ \in E^\circ$ connects two vertices $w, w' \in V^\circ$ with $|w| < \rho \leq |w'|$.

\begin{lemma} \label{modified7.2}
Under the assumptions of Proposition \ref{modified7.1}, suppose that the origin is located at $\frac{1}{2}(x+y)$ and that $h(x) < h(y)$.
Define
\begin{align*}
X &= \{ v \in V^\bullet \, : \, h(v) \leq h(x) \} \\
Y &= \{ v \in V^\bullet \, : \, h(v) \geq h(y) \}
\end{align*}
and set $Y' = Y \cup (\partial V^\bullet \cap D_R(0))$.
Assume that $X \cap  Y' = \emptyset$.
Then for every $\rho \in (r+\epsilon, R-\epsilon)$, there is a path in $G^\bullet$ from $X$ to $Y'$ consisting entirely of $\rho$-edges.
\end{lemma}
This lemma will follow from the following result about plane graphs. 
The proof of Lemma \ref{modified7.2} from Lemma \ref{modified7.3} follows the same logic that was used in \cite{gurel-jerison-nachmias:advm2020} to prove Lemma 7.2 from Lemma 7.3, but for the convenience of the reader, we include a proof after establishing Lemma \ref{modified7.3}.
This lemma guarantees that when the mesh-size is small enough one can follow the circular arcs in an annulus around $x$ and $y$ to get paths in the graph. The setting is illustrated in Figure \ref{PlaneLemmaPic}.

\begin{lemma}\label{modified7.3}
Let $H$ be a finite connected plane graph with planar dual graph $H^\dagger$.
Let $W, W^\dagger$ be the vertex sets of $H, H^\dagger$.
For each $w \in W^\dagger$, let $f(w)$ denote the face of $H$ that contains $w$.
Fix $\ep>0$ and assume that each inner face $f(w)$ is contained in the interior of $D_\epsilon(w)$.
Fix $0<r<R$ with $r+\epsilon < R-\epsilon$.
Let $\zeta \in W^\dagger$ be the dual vertex contained in the outer face of $H$, and assume that $\zeta \notin D_R=D_R(0)$.

Assume that $\partial H$, the topological boundary of the outer face of $H$,
is a cycle whose vertices are partitioned into two disjoint, nonempty sets
 $\mathcal{S}$ and $\mathcal{T}$, as in Figure \ref{PlaneLemmaPic}. 
Let $x, y \in W$ satisfy $|x|, |y| \leq r$.  
Let $X, Y$ be disjoint subsets of $W$ such that there is a path $\gamma_x$ in $H$ from $x$ to $\mathcal{S}$ whose vertices are all in $X$, and there is a path $\gamma_y$ in $H$ from $y$ to $\mathcal{S}$ whose vertices are all in $Y$.
Assume there is a path in $H$ from $x$ to $y$  that is contained in $D_r=D_r(0)$.
Further assume that 
there is a path $\gamma_\partial$ in $\partial H$ from  $\gamma_x\cap \mathcal{S}$ to  $\gamma_y\cap\mathcal{S}$ so that  $\gamma_\partial \cap D_R$ only contains vertices in $\mathcal{S}$.

Set $Y' = Y \cup [D_R \cap  \mathcal{S} ]$.  
If $X \cap Y' = \emptyset$, then for each $\rho \in (r + \epsilon, R-\epsilon)$, 
there is a path in $H$ from $X$ to $Y'$  consisting entirely of $\rho$-edges that are not in $\partial H$.

\end{lemma}

\begin{figure}
\centering
\begin{tikzpicture}[scale=1.6]

\draw[thick] (0,0) to [out=60,in=-110] (0,2);
\draw[thick] (0,0) to [out=-10,in=160] (3,0);
\draw[thick] (3,0) to [out=100,in=-70] (3,2);
\draw[thick] (0,2) to [out=30,in=200] (3,2);

\draw[fill] (2, 0.2) circle [radius=0.04];
\draw[fill] (2.4,0.4) circle [radius=0.04];

\draw (2.2,0.3) circle [radius=0.4];
\draw (2.2,0.3) circle [radius=1.2];

\node at (-0.15, 1) {$\mathcal{S}$};
\node at (3.3, 1.5) {$\mathcal{S}$};
\node at (0.5, -0.25) {$\mathcal{T}$};
\node at (1.5, 2.25) {$\mathcal{T}$};
\node at (2.6, -0.6) {$D_R$};
\node at (1.95,0.35) {$x$};
\node at (2.3,0.5) {$y$};
\node at (0.55, 1.3) {\color{red}{$\gamma_x$}};
\node at (2.65,1.0) {\color{blue}{$\gamma_y$}};

\draw[thick, dashed, red] (1.97,0.2) to [out=180,in=-20] (-0.1,1.48);
\draw[thick, dashed, black] (2.01, 0.2) to [out=0,in=-100] (2.4, 0.4);
\draw[thick, dashed, blue] (2.42, 0.42) to [out=50,in=180] (3.06, 1);

\end{tikzpicture}
\caption{A sketch of the setting for Lemma \ref{modified7.3}: the topological boundary of $H$ subdivided into sets $\mathcal{S}$ and $\mathcal{T}$ (solid black), the path $\gamma_x$ from $x$ to $\mathcal{S}$ with vertices in $X$ (dashed red), the path $\gamma_y$ from $y$ to $ \mathcal{S}$ with vertices in $Y$ (dashed blue), and the path $\gamma_{xy}$ from $x$ to $y$ in $D_r$ (dashed black).
The cycle $J$ from the proof of Lemma  \ref{modified7.3} consists of $\gamma_x, \gamma_{xy}, \gamma_{y}$ and the arc from  $\partial H$ traced counterclockwise from $\gamma_y$ to $\gamma_x$ .
}\label{PlaneLemmaPic}
\end{figure}
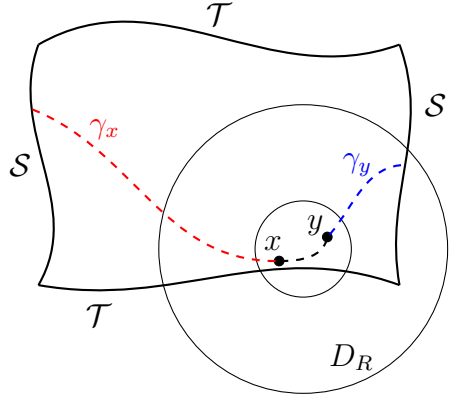

\begin{proof}
We begin by following the proof of Lemma 7.3 in \cite{gurel-jerison-nachmias:advm2020}
to obtain a cycle $K$ in $H$ such that all of its edges are $\rho$-edges.  
For convenience, we briefly repeat that argument here.  
Let $u \in W^\dagger$ be a dual vertex contained in an inner face of $H$ with $x$ incident to this face.  Then $u \neq \zeta$ and $|u| \leq r + \e < \rho$.
Let $\sigma_\rho$ be the set of vertices $w \in W^\dagger$ so that there is a path from $u$ to $w$ in $H^\dagger$ with all vertices of this path in the open disk of radius $\rho$ centered at 0.
Let $H[\sigma_\rho]$ be the subgraph of $H$ whose vertices and edges make up the boundaries of the faces containing $w$ for each $w \in \sigma_\rho$.
The graph $H[\sigma_\rho]$ is connected and contained in $D_R$.
Let $f_\rho^{\text{out}}$ be the outer face of $H[\sigma_\rho]$, which contains the complement of $D_R$.
Let $K$ be the boundary of $f_\rho^{\text{out}}$, which is a connected subgraph.
One can check that all of the edges in $K$ are $\rho$-edges.
Further,
we note that the edges in $K$ are either in $\partial H$, or within distance $\epsilon$ of $C_\rho$, the circle of radius $\rho$ centered at 0.  The former situation can happen when $C_\rho$ extends into the outer face of $H$.  However, when an arc of $C_\rho$ is contained entirely within the interior of $H$, then there will be edges of $K$ within distance $\epsilon$ of this arc.

Recall that $\gamma_x$ is a path in $H$ from $x$ to $ \mathcal{S}$ whose vertices are all in $X$,
$\gamma_y$ is a path in $H$ from $y$ to $ \mathcal{S}$ whose vertices are all in $Y$,
and $\gamma_\partial$ is a path in $\partial H$ from $\gamma_x$ to $\gamma_y$ with $\gamma_\partial \cap D_R \subset  \mathcal{S}$.
Let $\gamma_{xy}$ denote the path in $H$ from $x$ to $y$  that is contained in $D_r$.
Since $\gamma_x \subset X$ is disjoint from $\mathcal{S} \cap D_R \subset Y'$, we must have that $\gamma_x$ connects $x$ to a point in $\partial H$ that is outside $D_R$.  Hence $\gamma_x$ must cross the annulus $D_R \setminus D_r$.
Further  $\gamma_x$ and $\gamma_y \cup (\mathcal{S} \cap D_R)$ must be disjoint within $D_R$, since the vertices of the former are in $X$ and the vertices of the later are in $Y'$.
Hence we can find a cycle $J$ 
composed of three disjoint arcs:
(1) an arc contained in $\gamma_{xy}$ within $D_r$,
(2) an arc  contained in $\gamma_x$ that crosses $D_R \setminus D_r$,
and (3) an arc contained in $\gamma_y  \cup \gamma_\partial$ that crosses $D_R \setminus D_r$ and whose intersection with $D_R$ is contained in $\gamma_y \cup  \mathcal{S}$.

Now the interior of $J$ is contained in the interior of $H$ by construction.  
Further it contains an arc of $C_\rho$.
Hence there must be a path contained in $K$  that crosses the interior of $J$ from a vertex in $\gamma_x$ to a vertex in $\gamma_y \cup \mathcal{S}$.  
This path could potentially contain some edges in $\partial H$, but the only way for this to happen is for the path to contain edges from $\gamma_x$ or $\gamma_y \cup \mathcal{S}$.  
Hence, we can shorten the path by starting  at the last vertex it encounters in $\gamma_x$ and ending at the first vertex it encounters in $\gamma_y \cup \mathcal{S}$.  Then all the edges in the path are strictly within the interior of $J$ and hence in the interior of $H$.  This gives us the desired path.

\end{proof}

\begin{proof}[Proof of Lemma \ref{modified7.2}.]
Assume the hypotheses of Proposition \ref{modified7.1} and  Lemma \ref{modified7.2}.
In order to apply Lemma \ref{modified7.3}, we need to modify $G_1^\bullet$ and $G_1^\circ$ to create graphs $H$ and $H^\dagger$
that are planar duals.
Define $W=V^\bullet$ and $W^\dagger = V^\circ \cup \{ \zeta \}$, where $\zeta$ is a new vertex in the outer face of $G_1$ with $|\zeta|>R$ and with distance at least $10\epsilon$ to the topological boundary of $G$.
As you traverse the topological boundary of $G_1$ clockwise, let $v_1, w_1, v_2, w_2, \cdots, v_n, w_n, v_1$ denote the vertices in order, with $v_k \in V^\bullet$ and $w_k \in V^\circ$.
For $k=1, \cdots, n$, let $e_k^\bullet$ be a new edge in the outer face of $G_1$ between $v_k$ and $v_{k+1}$, taking indices mod $n$.
By drawing these edges appropriately, we may assume that $e_k^\bullet \subset \overline{D_\epsilon (w_k)}$.
For $k=1, \cdots, n$, let $e_k^\circ$ be a new edge in the outer face of $G_1$ between $w_k$ and $\zeta$ that intersects $e_k^\bullet$ exactly once.
Then define $H$ to have vertex set $W$ and edge set $E^\bullet \cup \{e_1^\bullet, \cdots, e_n^\bullet \}$, and similarly define
$H^\dagger$ to have vertex set $W^\dagger$ and edge set $E^\circ \cup \{e_1^\circ, \cdots, e_n^\circ \}$.
Note that $H$ and $H^\dagger$ are planar duals by construction.  Further, for each $w \in W^\dagger \setminus \{\zeta\}$, the inner face $f(w)$ is contained in the interior of $D_\epsilon(w)$ and $\zeta \notin D_R$.

The topological boundary of $H$ consists of the cycle with vertices $v_1, v_2, \cdots, v_n$.
These are precisely the vertices from $V^\bullet$ that are in the topological boundary of $G_1$.
Let $\mathcal{S} = \partial V^\bullet$, and let $\mathcal{T}$ be the remaining vertices in $\partial G_1$, 
i.e.~the vertices from $\partial G_1 \cap \left(T_1 \cup T_2\right)$. 
With this definition, note that the definitions of $Y'$ in Lemma \ref{modified7.2} and Lemma \ref{modified7.3} are in agreement.
To finish establishing the hypotheses of Lemma \ref{modified7.3}, we must show the existence of the paths $\gamma_x, \gamma_y,$ and
$\gamma_\partial$.

Let $X_0$ be the connected component of $X$ in $H$ that contains $x$.  For the sake of contradiction, assume that $X_0$ does not contain any vertices from $\mathcal{S}$.  
This means that $h$ is discrete harmonic on $X_0$.
Further, since $H$ is connected, the set 
$$X_{\text{adj}} = \{ v \in W \, : \,  v\notin X_0 \text{ but } v \text{ is adjacent to a vertex in } X_0\}$$ is nonempty.  
Thus by the discrete maximum principle, there is a vertex $v  \in X_\text{adj}$ with $h(v) \leq h(x).$
But this means $v \in X_0$, which gives a contradiction.
Therefore, $X_0$ contains a vertex from $\mathcal{S}$, which means that there is a path in $H$ from $x$ to $\mathcal{S}$ whose vertices are all in $X$.
The same argument gives that there is a path in $H$ from $y$ to $\mathcal{S}$ whose vertices are all in $Y$.
The path $\gamma_\partial$ exists by the assumption in Proposition \ref{modified7.1} that between any two vertices in $\partial V^\bullet$, there is a path $\gamma$  in $\partial G_1^\bullet$ so that 
$\gamma \cap D_R(z)$  only contains vertices in  $\partial V^\bullet$.

Fix  $\rho \in (r + \epsilon, R-\epsilon)$.
Since we have established all the conditions of Lemma \ref{modified7.3}, it follows that 
there is a path in $H$ from $X$ to $Y'$ consisting entirely of $\rho$-edges that are not in $\partial H$.
Note that this path is also contained in $G_1^\bullet$, since $\partial H$ contains all the extra edges in $H$ that are not in $G_1^\bullet$.
Further, one can check that the definition of $\rho$-edge in $H$ and $\rho$-edge in $G_1^\bullet $ are in agreement for the edges in this path.
\end{proof}

\begin{proof}[Proof of Proposition \ref{modified7.1}.]
Assume the hypotheses of Proposition \ref{modified7.1}.  We may further assume that the origin is located at $\frac{1}{2}(x+y)$ and that $h(x) < h(y)$.
Define $X, Y$ and $Y'$ as in the statement of Lemma \ref{modified7.2}, and let $X' =  X \cup [\partial V^\bullet \cap D_R(0) ]$.
If $X \cap Y' = \emptyset$, then by Lemma \ref{modified7.2} and Proposition 6.2 in \cite{gurel-jerison-nachmias:advm2020},
there is a unit flow $\theta$ in $G_1^\bullet$ from $X$ to $Y'$ with 
\begin{equation*}
\mathcal{E}_1(\theta) \leq \frac{C}{\log\left[(R-\epsilon)/(r+\epsilon)\right]},
\end{equation*}
where $C$ is a universal constant.
Since $R \geq 2r+3\epsilon$, one can show that $\frac{R}{r+\epsilon} \leq \left(\frac{R-\epsilon}{r+\epsilon} \right)^2$.
Therefore if $X \cap Y' = \emptyset$, 
\begin{equation}\label{XY'flow}
\mathcal{E}_1(\theta)  \leq \frac{C}{\log\left[R/(r+\epsilon)\right]}.
\end{equation}
Similarly, if $X' \cap Y = \emptyset$, 
then
there is a unit flow $\phi$ in $G_1^\bullet$ from $X'$ to $Y$ with 
\begin{equation}\label{X'Yflow}
\mathcal{E}_1(\phi) \leq \frac{C}{\log\left[R/(r+\epsilon)\right]}.
\end{equation}

To complete the proof, we will consider three cases.  First assume that $Y=Y'$, which means that $\partial V^\bullet \cap D_R(0)$ is empty or is contained in $Y$.  In this case, $X \cap Y' = \emptyset$, and so \eqref{XY'flow} holds. 
Then Proposition 4.11 in \cite{gurel-jerison-nachmias:advm2020} shows that
$$h(y) - h(x)  = \text{gap}_{X, Y}(h) \leq \mathcal{E}_1(\theta)^{1/2} \mathcal{E}_1(h)^{1/2}
\leq \frac{C\, \mathcal{E}_1^\bullet(h)^{1/2}}{\log^{1/2}\left[ R/(r+\epsilon) \right]},$$
where 
$$\text{gap}_{X, Y}(h) := \min_{v \in Y}h(v) - \max_{w \in X}h(w).$$

Next assume $X=X'$.  In this case, we can repeat the above argument using $\phi$ and \eqref{X'Yflow}  instead of $\theta$ and \eqref{XY'flow}. 

Lastly assume that $\partial V^\bullet \cap D_R(0)$ is nonempty and is not a subset of $X$ or $Y$.
Let $m_1$ and $m_2$ be the minimum and the maximum of $h(v)$ over $v \in \partial V^\bullet \cap D_R(0)$.
Then $\kappa = m_2-m_1$, and our assumption gives that  $m_1< h(y)$ and $m_2 > h(x)$.
This implies that
$$ \text{gap}_{X, Y'}(h) = m_1 -h(x) \; \text{ and } \; \text{gap}_{X', Y}(h) = h(y) - m_2.$$
Then $$ h(y)-h(x) - \kappa = \text{gap}_{X, Y'}(h) + \text{gap}_{X', Y}(h),$$ 
and to finish the proof, it remains to find an upper bound for these two gap terms.
Note that $X \cap Y' = \emptyset $ precisely when $\text{gap}_{X, Y'}(h) > 0$.
Hence, we can either bound $\text{gap}_{X, Y'}(h)$ above by 0, or we can use Proposition 4.11 in \cite{gurel-jerison-nachmias:advm2020} and \eqref{XY'flow}  to show 
$$\text{gap}_{X, Y'}(h) \leq \mathcal{E}_1(\theta)^{1/2} \mathcal{E}_1(h)^{1/2}
\leq \frac{C\, \mathcal{E}_1^\bullet(h)^{1/2}}{\log^{1/2}\left[ R/(r+\epsilon) \right]}.$$
Similarly, $X' \cap Y = \emptyset$ precisely when $\text{gap}_{X', Y}(h) > 0$, allowing us to bound $\text{gap}_{X', Y}(h)$ above by 0 or by $
 \frac{C\, \mathcal{E}_1^\bullet(h)^{1/2}}{\log^{1/2}\left[ R/(r+\epsilon) \right]}$ using Proposition 4.11 in \cite{gurel-jerison-nachmias:advm2020} and \eqref{X'Yflow}.  

\end{proof}

Lastly, we need the following result that allows us to convert a path in $\hat{G}$ to a nearby path in $G^\bullet$.  This result follows exactly the proof of Lemma 3.2 in \cite{gurel-jerison-nachmias:advm2020}, and so we omit the proof.

\begin{lemma}\label{nearbypath}
Let $G$ be a finite orthodiagonal map with edge lengths bounded by $\ep$, let $z,w \in V^\bullet$, and let $\gamma_{zw}$ be a curve from $z$ to $w$ in $\hat G$.
Then there is a path $\gamma$ between $z$ and $w$ in $G^\bullet$ so that $\diam(\gamma) \leq \diam(\gamma_{zw}) + 2 \ep.$
\end{lemma}

\subsection{Proof of Theorem \ref{HarmConvThm}}\label{HarmConvThmProof}
This section is devoted to the proof of Theorem \ref{HarmConvThm},
 where we must show  that \eqref{maingoal}  holds for $\lambda < C'$ for some $C'>0$ which depends only on $\Omega$.
 We split this into two cases, depending on the distance to the boundary arcs $A$ and $B$.   
 Note that each of the two propositions contains slightly more than we need to prove Theorem \ref{HarmConvThm}, but we will use this extra information elsewhere.

\begin{prop}\label{claim1B}
Assume we are in the setting of Theorem \ref{HarmConvThm}, and
let $z \in V^\bullet \cap \overline{\Omega}$ such that dist$(z, A \cup B) \leq 2\lambda^{1/4}$.  
 For $\lambda$ small enough (depending only on $\Omega$),
$|h_d(z) - h_c(z)| \leq \frac{C}{\log^{1/2}\left(\diam(\Omega/ \lambda \right)},$
and further there exists $w \in \partial V^\bullet$ so that 
$| h_d(z) - h_d(w) | \leq \frac{C}{\log^{1/2}\left( \diam(\Omega)/\lambda \right)}$.
\end{prop}

\begin{prop}\label{subdiagHarmConv}
Assume we are in the setting of Theorem \ref{HarmConvThm}, and
let $z_0 \in V^\bullet$ such that dist$(z_0, A \cup B) > 2\lambda^{1/4}$. 
Then 
$h_c$ has a harmonic extension to a region $\mathcal{O}$
so that for $\lambda$ small enough (depending only on $\Omega$),
$z_0 \in \mathcal{O}$ and 
  $ |h_d(z_0) - h_c(z_0)| \leq \frac{C}{\log^{1/2}\left(\diam(\Omega)/ \lambda \right)}$.
\end{prop}

By scale invariance, we will prove both of these in the case that $\diam(\Omega) =1$.

\begin{proof}[Proof of Proposition \ref{claim1B}.]

Let $z \in V^\bullet \cap \overline{\Omega}$ such that dist$(z, A \cup B) \leq 2\lambda^{1/4}$.
Choose $u \in A \cup B$ such that $|z-u| \leq 2\lambda^{1/4}$.
Then there is
 $v \in S_1 \cup S_2$ such that there is a path $\gamma_{uv}$ from $u$ to $v$ in $\hat G$ with length at most $\delta$.
 Further by following the edge containing $v$ to its endpoint in $V^\bullet$, there is 
 $w\in \partial V^\bullet$ such that $|v-w| \leq  \epsilon $ (and let $\gamma_{vw}$ be the path from $v$ to $w$ along the edge).
Then $|z-w| \leq 2\lambda^{1/4} + \delta + \e \leq 4\lambda^{1/4}$.

Since $h_d(w) = h_c(u)$, we have
$$ | h_d(z) - h_c(z) | \leq |h_d(z) - h_d(w) | + |h_c(u) - h_c(z)|.$$
From Lemma \ref{reflection},
$$ | h_c(u) - h_c(z) | \leq C |u-z|^{a} \leq C \lambda^{a/4} \leq  \frac{C}{\log^{1/2}\left(1/\lambda \right)}. $$
Thus it remains to bound $ |h_d(z) - h_d(w) | $.

Our goal will be to use Proposition \ref{modified7.1}, and so we must find a path from $z$ to $w$ in $G^\bullet$.
Since $\Omega$ is a linearly connected domain (with constant $K$), there is a path $\gamma_{zu}$ in $\overline{\Omega} \subset \hat{G}$ from $z$ to $u$ so that 
$$\diam(\gamma_{zu}) \leq K |z-u| \leq 2 K \lambda^{1/4}.$$
Therefore $\gamma_{zw} = \gamma_{zu} \cup \gamma_{uv} \cup \gamma_{vw}$ is a path from $z$ to $w$ in $\hat{G}$ with 
$$\diam(\gamma_{zw}) \leq 2K \lambda^{1/4}+\delta + \e \leq (2K + 2)\lambda^{1/4}.$$
Now by Lemma \ref{nearbypath} the path $\gamma_{zw}$ can be converted into a path $\gamma$ in $G^\bullet$  so that 
$$\diam(\gamma) \leq \diam(\gamma_{zu}) + 2\e \leq (2K+4)\lambda^{1/4}. $$
Now a disc centered at $\frac{1}{2}(z+w)$ will contain $\gamma$ if its radius is at least
$$\frac{1}{2}|z-w| + \diam(\gamma) \leq 2 \lambda^{1/4} + (2K+4)\lambda^{1/4} = (2K+6)\lambda^{1/4}.$$
Set $r= (2K+6)\lambda^{1/4}$ and $R = (4K + 15) \lambda^{1/8}$.
Note that, by taking $\l$ small enough,
$$ 2r+3\e \leq R  \leq \frac{1}{4} \hat{d}, $$
where $\hat{d} := \min\{ \text{\normalfont{dist}}(T_1, T_2),  \text{\normalfont{dist}}(S_1, S_2)\}$ 
and
$$\frac{R}{r+\e} = \frac{(4K + 15 ) \lambda^{1/8}}{(2K+6)\lambda^{1/4}+ \e} \geq \lambda^{-1/8}.$$
The fact that $R \leq \frac{1}{4} \hat{d}$ means that $D_R$, the disc of radius $R$ centered at $\frac{1}{2}|z+w|$ can intersect at most one of $S_1$ or $S_2$ and at most one of $T_1$ or $T_2$.
Because $h_d$ is constant equal to 0 on $S_1 \cap V^\bullet$ and constant equal to 1 on $S_2 \cap V^\bullet$, 
the term $\kappa$ from Proposition \ref{modified7.1} will be zero.
Further since $D_R$ intersects at most one of $T_1$ or $T_2$,
between any two vertices in $S_1 \cup S_2$, there is a path $\gamma_\partial$  in $\partial G^\bullet$ so that the vertices of $\gamma_\partial \cap D_R(z)$ belong to $S_1 \cup S_2$.
Therefore Proposition \ref{modified7.1} gives
$$| h_d(z) - h_d(w) | \leq \frac{C\, \mathcal{E}^\bullet(h_d)^{1/2}}{\log^{1/2}\left( 1/\lambda \right)}.$$

To complete the proof, it remains to bound  $\mathcal{E}^\bullet(h_d)$.
We reflect $\psi $ over $\tau_1$ and $\tau_2$, 
so that $\psi$ is conformal in an open set $\mathcal{O}$ containing $\Omega$, as in Lemma \ref{reflection}.
Define $\mathcal{O}_{\text{mid}} = \{z \in \mathcal{O} \, : \, 1/3 \leq h_c(z) \leq 2/3 \}$.
Consider the distance between $\tau_k \cap \mathcal{O}_{\text{mid}}$ and $\partial \mathcal{O}$.
By taking $\lambda$ small enough, we may assume
 that $3\lambda < \text{dist}\left( \tau_k \cap \mathcal{O}_{\text{mid}}, \, \partial \mathcal{O} \right)$. 
Then the boundary of any connected component of $\hat{G} \setminus \mathcal{O}_{\text{mid}}$ 
    will intersect exactly one of $\psi^{-1}(\{\text{Re}(z) = 1/3\})$ or $\psi^{-1}(\{\text{Re}(z) = 2/3\})$.
Let $\hat{G}_\text{left} $ be the union of all components whose boundary intersects  $\psi^{-1}(\{\text{Re}(z) = 1/3\})$, and similarly define  $\hat{G}_\text{right} $ as the union of components whose  boundary intersects $\psi^{-1}(\{\text{Re}(z) = 2/3\})$.
Define the continuous function $f$ on $\hat{G}$ as follows:
$$f(z) = \begin{cases}
	0					&\quad\text{if } z \in \hat{G}_\text{left}  \\  
	1					&\quad\text{if } z \in \hat{G}_\text{right}  \\  
	3h_c(z) -1                      & \quad\text{else}\\                      
      \end{cases}.$$
Then by Proposition 4.8 in  \cite{gurel-jerison-nachmias:advm2020},
$$\mathcal{E}^\bullet(h_d) \leq \mathcal{E}^\bullet(f).$$
Form $G_{\text{mid}}$  from all the quadrilaterals that intersect $\mathcal{O}_{\text{mid}}$.
Note that the only quadrilaterals in $G$ that contribute to $\mathcal{E}^\bullet(f)$ are those in $G_{\text{mid}}$.  
Further, for these $Q = [v_1, w_1, v_2, w_2]$ in $G_{\text{mid}}$, 
   we have that 
   \begin{equation*}
   |f(v_1) - f(v_2)| \leq 3 \left| h_c(v_1)- h_c(v_2) \right|.
   \end{equation*}
Therefore 
$$ \mathcal{E}^\bullet(f) \leq 9 \mathcal{E}_{\text{mid}}^\bullet(h_c),$$
where the energy on the right is with respect to the graph $G_{\text{mid}}$.
Since $G_{\text{mid}}$ is contained within $\mathcal{O}$, Proposition 5.2 in  \cite{gurel-jerison-nachmias:advm2020} implies that 
$$  \mathcal{E}_{\text{mid}}^\bullet(h_c) \leq 2 \int_{\hat{G}_{\text{mid}}} |\nabla h_c |^2 \, dA + 2\,\text{area}(\hat{G}_{\text{mid}})(10 LM\epsilon + 8M^2\e^2),$$
where $L = || \nabla h_c ||_{\infty, \hat{G}_{\text{mid}}}$ and $M = || H h_c ||_{\infty, \tilde{G}_{\text{mid}}}$.
The constants $L$ and $M$ depend on $\Omega$ but are independent of $\e$ and $\delta$.
Since $\mathcal{O}$ and the extended $\psi$ are defined by reflection, $\psi(\mathcal{O})$ is contained in the rectangle $[0,1]\times[-m,2m]$
and $\int_{\hat{G}_{\text{mid}}} |\nabla h_c |^2 \, dA \leq 3 \int_{\Omega} |\nabla h_c |^2\, dA$.  However,  $\int_{\Omega} |\nabla h_c |^2 \, dA= m$   the modulus of curves in $R$ from $\{ \text{Re}(z) = 0\}$ to $\{ \text{Re}(z) = 1\}$.
Thus we have shown that 
$\mathcal{E}^\bullet(h_d) $ is bounded by a constant (which depends only on $\Omega$), and hence
\begin{equation*}
| h_d(z) - h_d(w) | \leq \frac{C}{\log^{1/2}\left( 1/\lambda \right)}.
\end{equation*}

\end{proof}

\begin{proof}[Proof of Proposition \ref{subdiagHarmConv}.]
In this proof, we first discuss the overall logic and then conclude by establishing the stated claims
(following a similar structure as in \cite{gurel-jerison-nachmias:advm2020}).
We assume that 
there are $z \in V^\bullet$ with dist$(z, A \cup B) > 2\lambda^{1/4}$, since otherwise there is nothing to show.  
Since diam$(\Omega) =1$, this will imply that $2\lambda^{1/4} \leq 1$.  
This is equivalent to $\lambda \leq \frac{1}{16}$ or $\lambda \leq \frac{1}{8} \lambda^{1/4}$, which we will use at times.
We will also assume at times that $\lambda$ is small enough, where this depends only on the analytic quadrilateral $\Omega$.

To begin, we will construct a sub-orthodiagonal map which has distance at least $\lambda^{1/4}$ from $A\cup B$,
putting us in  the setting for  many of the tools developed in the previous section.
This task is accomplished in
the following construction and claim.
Let $P$ be the set of all inner faces $Q$ of $G$ such that  dist$(Q, A \cup B) \geq \lambda^{1/4}$, 
and let $G[P]$ be the subgraph of $G$ that is the union of the boundaries of the faces in $P$.
Note that $G[P]$ may be composed of more than one block (i.e.~maximal 2-connected component).

\begin{claim}\label{claim2}
The set $P$ is nonempty.  The inner faces of $G[P]$ are precisely the elements of $P$, and every point $p$ on the boundary of the outer face of $G[P]$ is either in $T_1 \cup T_2$ or satisfies
\begin{equation} \label{bdrybound}
\dist(p, A \cup B) \leq \frac{5}{4} \lambda^{1/4}.
\end{equation}
Each $z \in V^\bullet$ with $\dist(z, A \cup B) > 2\lambda^{1/4}$ is a vertex of a unique block of $G[P]$.  
If $H$ is a block of $G[P]$, then the inner faces of $H$ are all elements of $P$ and every point $p$ on the boundary of the outer face of $H$ is either in  $T_1 \cup T_2$ or satisfies \eqref{bdrybound}.
\end{claim}

Fix $z_0 \in V^\bullet $ with dist$(z_0, A \cup B) > 2\lambda^{1/4}$, and
let $G_1$ be the block of $G[P]$ containing $z_0$.
Note that if the topological boundary $\partial G_1$ does not intersect $T_1 \cup T_2$, then we are in the setting addressed in $\cite{gurel-jerison-nachmias:advm2020}$ of no free boundary, and their work takes care of this case.  
Hence we will assume that $\partial G_1$ does intersect $T_1 \cup T_2$.
Then Claim \ref{claim2} implies that each block $H$ of $G[P]$ is a sub-orthodiagonal map, where
$   \partial V_1^\bullet$ are the points from  $V_1^\bullet$  that are in  $\partial G_1 \setminus \left(  T_1 \cup T_2 \right) $
and  ${\rm int}(V_1^\bullet) = V_1^\bullet \setminus  \partial V_1^\bullet$.
Lastly, by Lemma \ref{reflection}, we analytically extend $\psi$ (and hence $h_c$) to be defined on a domain $\mathcal{O}$ containing $\hat{G}_1$ so that $\dist(\hat{G}_1, \partial \mathcal{O}) \geq \frac{5}{16}\l^{1/4}$.

Our goal is to show that 
\begin{equation}\label{subgoal}
 |h_d(z_0) - h_c(z_0)| \leq \frac{C}{\log^{1/2}\left(1/ \lambda \right)}
 \end{equation}
  for $\lambda$ small enough (depending only on $\Omega$).
To accomplish this, we need derivative bounds on the extended $h_c=\rea \psi$ on $\hat{G}_1$.

\begin{claim}\label{claim3}
Let $\tilde{G}_1$ denote the union of the convex hulls of the closures of the inner faces of $G_1$.  
Then for $\lambda$ small enough (depending on $\Omega$),
$$ || \nabla h_c ||_{\infty, \tilde{G}_1} \leq C \lambda^{-1/4} \;\;\; \text{ and } \;\;\;
 || H h_c ||_{\infty, \tilde{G}_1} \leq C \lambda^{-1/2}.$$ 
\end{claim}

Define $h_d^{(1)} : V_1^\bullet \to \mathbb{R}$ to be the solution to the discrete Dirichlet problem on ${\rm int}(V_1^\bullet)$ with boundary data $h_c$ on $\partial V_1^\bullet$.
Then
\begin{equation}  \label{2parts}
|h_d(z_0) - h_c(z_0) | \leq | h_d(z_0) - h_d^{(1)}(z_0) | + | h_d^{(1)}(z_0) - h_c(z_0)|. 
\end{equation}
We will bound each of the terms in \eqref{2parts} separately, and we begin with the first term.

\begin{claim}\label{claim4}
The first term of \eqref{2parts} is bounded as follows:
$$ |h_d(z_0) - h_d^{(1)}(z_0)| \leq \frac{C}{\log^{1/2}(1/\lambda)} .$$
\end{claim}

We now turn to the second term of \eqref{2parts} and
set $f =  h_d^{(1)} - h_c$, which satisfies
 that $f=0$ on $\partial V^\bullet_1$.
The following claim is a key part of the proof, and it will be used to  bound $|f(z_0)|$, which is the second term in \eqref{2parts}.

\begin{claim}\label{claim5}
Let  $D_{4\lambda}$ be the disk of radius $4\lambda$ centered at $z_0$.
For $\lambda$ small enough (depending only on $\Omega$), the set $D_{4\lambda}$ is disjoint from $\partial V^\bullet_1$, and for $y \in V^\bullet_1 \cap D_{4\lambda}$
$$ |f(z_0) - f(y)| \leq \frac{C}{\log^{1/2}(1/\lambda)} .$$
\end{claim}

Proposition \ref{modified5.2} and Claim \ref{claim3} imply that 
$$\mathcal{E}^\bullet_1(f) \leq C \left[ M^2 \epsilon^2 + L^2\lambda \right] 
\leq C \lambda^{1/2},$$
where
 $L = || \nabla h_c||_{\infty, \hat{G}_1}$ and $M = || H h_c ||_{\infty, \tilde G_1}.$
Applying Proposition \ref{modified6.1}
yields a  unit flow $\theta$ in $G^\bullet_1$ from $V_1^\bullet \cap D_{4\lambda}$ to $\partial V_1^\bullet$ such that 
$$\mathcal{E}_1^\bullet(\theta) \leq C \log\left( 1/\lambda \right).$$
Then $-\theta$ is a unit flow from $\partial V_1^\bullet$ to $V_1^\bullet \cap D_{4\lambda}$.
Applying Proposition 4.11 in \cite{gurel-jerison-nachmias:advm2020}  to $f$ and $-\theta$ shows that 
$$ \text{gap}_{\partial V_1^\bullet, V_1^\bullet \cap D_{4\lambda} }(f) \leq \mathcal{E}_1^\bullet(-\theta)^{1/2} \mathcal{E}_1^\bullet(f)^{1/2} \leq C\lambda^{1/4}\log^{1/2}\left( 1/\lambda \right),$$
where 
\begin{equation}\label{eq:gap-def}
\displaystyle \text{gap}_{X,Y}(f) :=  \min_{y \in Y} f(y) - \max_{x \in X} f(x).
\end{equation}
Note that Proposition 4.11 only applies in the case that the gap is nonnegative, but  the inequality is immediate in the other case.
Since $f=0$ on $\partial V_1^\bullet$, there is $y \in V_1^\bullet \cap D_{4\lambda}$ so that
$$f(y) \leq   C\lambda^{1/4}\log^{1/2}(1/\lambda).$$
Then combining this with Claim \ref{claim5}, we obtain
$$f(z_0) = f(y) + [ f(z_0) - f(y) ] \leq  C\lambda^{1/4}\log^{1/2}(1/\lambda) + \frac{C}{\log^{1/2}(1/\lambda)} \leq \frac{C}{\log^{1/2}(1/\lambda)}$$
Applying the same argument to $-f$ yields that
$$|f(z_0)| \leq \frac{C}{\log^{1/2}(1/\lambda)} .$$

All that remains is to establish Claims \ref{claim2}-\ref{claim5}.  
Two of these, Claims \ref{claim2} and \ref{claim4}, follow the arguments of the corresponding results, Claims 8.2 and 8.4, in \cite{gurel-jerison-nachmias:advm2020}  with only minor changes of constants.  As a result, we omit those proofs here.

\begin{proof}[Proof of Claim \ref{claim3}.]
Note that $\sup_{y \in \Omega} |h_c(y)| \leq 1$, since $h_c$ is the real part of the conformal map $\psi$ onto $[0,1] \times [0,m]$.
For all $q \in \tilde{G}_1$, Lemma \ref{reflection} implies that for $\lambda$ small enough
\begin{equation}\label{dist2bdry2}
\text{dist}(q, \partial \mathcal{O}) \geq \frac{5}{16} \lambda^{1/4} - \e \geq \frac{3}{16} \lambda^{1/4}.
\end{equation}
Proposition 3.7 of \cite{gurel-jerison-nachmias:advm2020} and \eqref{dist2bdry2} imply that 
$$ || \nabla h_c ||_{\infty, \tilde{G}_1} \leq C \lambda^{-1/4}.$$
We obtain the Hessian bound $ || H h_c ||_{\infty, \tilde{G}_1} \leq C \lambda^{-1/2}$ from Proposition 3.7 of \cite{gurel-jerison-nachmias:advm2020}  and \eqref{dist2bdry2} exactly as in the proof of Claim 8.3 in \cite{gurel-jerison-nachmias:advm2020}.

\end{proof}

\begin{proof}[Proof of Claim \ref{claim5}.]
Recall that we have $z_0 \in V_1^\bullet$ with dist$(z_0, A \cup B) > 2 \lambda^{1/4}$.
From Claim \ref{claim2}, we know that every point $p \in \partial G_1$ satisfies dist$(p, A \cup B) \leq \frac{5}{4} \lambda^{1/4}$.
Thus the distance from $z_0$ to $\partial V_1^\bullet$ is at least 
  $2\lambda^{1/4} -  \frac{5}{4} \lambda^{1/4} = \frac{3}{4} \lambda^{1/4} \geq 6 \lambda$,
since $\lambda^{1/4} \geq 8\lambda$,
and this implies that $D_{4\lambda}$ cannot intersect $\partial V_1^\bullet$.

Let $y  \in V^\bullet_1 \cap D_{4\lambda}$.  Then
$$| f(z_0) - f(y)|  \leq | h_d^{(1)} (z_0) - h_d^{(1)}(y)| +  |h_c(z_0) -h_c(y)|.$$
Our next step is to get a curve from $z_0$ to $y$ in $\hat G_1$ with length bounded by $C \lambda$.  
Then by Claim \ref{claim3}, we will have that
$$ |h_c(z_0) - h_c(y) | \leq || \nabla h_c ||_{\infty, \hat G_1} \cdot C \lambda \leq C \lambda^{3/4}.$$ 
Since $z_0$ and $y$ may not be in $\overline{\Omega}$, we will choose $u$ and $w$ in $\overline \Omega$ so that 
there is a  curve $\gamma_{z_0u}$ from $z_0$ to $u$ in $\hat G_1$ 
and a curve  $\gamma_{wy}$ from $w$ to $y$ in $\hat G_1$ 
with the length of both curves $\gamma_{z_0u}, \gamma_{wy}$ bounded by $\delta$. 
Since $\Omega$ is a linearly connected domain (with constant $K$), there is a curve $\gamma_{uw}$ in $\overline{\Omega} \subset \hat{G}_1$ from $u$ to $w$ so that 
$$\diam(\gamma_{uw}) \leq K |u-w| \leq K \left( |z_0-y| +  2\delta \right) \leq 6K \lambda.$$
Therefore $\gamma_{z_0y} = \gamma_{z_0u} \cup \gamma_{uw}   \cup \gamma_{wy}$ is a path from $z_0$ to $y$ in $\hat{G_1}$ with 
$$\diam(\gamma_{z_0y}) \leq (6K +2) \lambda.$$

It remains to bound $| h_d^{(1)} (z_0) - h_d^{(1)}(y)| $.
Propositions 4.8 and 5.1 of \cite{gurel-jerison-nachmias:advm2020} give
$$ \mathcal{E}^\bullet_1( h_d^{(1)} ) \leq  \mathcal{E}^\bullet_1( h_c ) 
   \leq 2 \int_{\hat G_1} | \nabla h_c |^2 \,dA + 2 \text{area}(\hat G_1) \, ( 10 LM \e + 8 M^2 \e^2) $$
where (using Claim \ref{claim3})
\begin{align*}
L &= || \nabla h_c ||_{\infty, \hat G_1} \leq C \lambda^{-1/4}, \\
M &= || H h_c ||_{\infty, \tilde G_1}  \leq C \lambda^{-1/2}.
\end{align*}
Since $\mathcal{O}$ and the extended $\psi$ are defined by reflection, $\psi(\mathcal{O})$ is contained in the rectangle $[0,1]\times[-m,2m]$, and so
$$\int_{\hat G_1} | \nabla h_c |^2 \, dA \leq 3\int_{\Omega} | \nabla h_c |^2 \, dA = 3m,$$
where $m$ is the continuous modulus of the family of curves from $A$ to $B$ in $\Omega$.
Thus,
$$  \mathcal{E}^\bullet_1( h_d^{(1)} ) \leq C.$$

We will utilize Proposition \ref{modified7.1} to finish the proof, and so we must verify that its hypotheses are satisfied.
By Lemma \ref{nearbypath}, the curve $\gamma_{z_0y}$ from above can be converted into a path $\gamma$ in $G^\bullet_1$ 
so that
$$\diam(\gamma) \leq \diam(\gamma_{z_0y}) + 2\e \leq (6K+4)\lambda. $$
Thus
a disc centered at $\frac{1}{2}(z_0+y)$ will contain $\gamma$ if its radius is larger than
$$\frac{1}{2}|z_0-y| + \diam(\gamma) < 2 \lambda + (6K+4)\lambda \leq (6K+6) \lambda.$$
Set $r=(6K+7) \lambda $ and $R = \frac{1}{8} \lambda^{1/4}$.
Choose $\lambda$ to be small enough so that  $ 2r+3\e \leq R$. 
Lastly, we must check the boundary path condition of Proposition \ref{modified7.1}.
By taking $\lambda$ small enough, we may assume that $D_R$ intersects at most one of $T_1$ or $T_2$.  If $D_R$ intersects neither, then the 
 the boundary path condition  is satisfied.  
 Assume then that $D_R$ intersects $T_1$ (and not $T_2$).
 Further, for the sake of contradiction, assume  the boundary path condition  is not true, that is, suppose there are two vertices in $\partial V^\bullet$, so that both paths in $\partial G_1^\bullet$ between these vertices contain a point in $T_1 \cap D_R $.
This means that $T_1 \cap \partial G_1$ must be disconnected.
For this to happen, look at the original graph $G$ and follow $T_1$:  after $T_1$ intersects $D_R$ for the first time, it  must intersect $\partial V^\bullet$ of $G_1$ (which is within $\frac{5}{4} \lambda^{1/4}$ of $A\cup B$), and then $T_1$ must  intersect $D_R$ for a second time.  
Since $\dist_{\hat{G}}(T_1, \tau_1) < \lambda$, it follows that $\tau_1$ must have a similar oscillation: after $\tau_1$ intersects $D_{R+\lambda}$, it must move close to $A\cup B$ (i.e.~within distance $\frac{5}{4} \lambda^{1/4} + \lambda \leq \frac{11}{8}\lambda^{1/4}$), and then it must again intersect $D_{R+\lambda}$.
Note that
$$\dist(D_{R+\lambda}, A\cup B) \geq \dist(z_0, A\cup B) - \frac{1}{2}|z_0-y| - (R+\lambda) > 2\lambda^{1/4} - 2\lambda - \frac{1}{8} \lambda^{1/4} - \lambda \geq  \frac{3}{2} \lambda^{1/4}.$$
 So  $\tau_1$ moves away from $A\cup B$ to a distance of at least $\frac{3}{2}\lambda^{1/4}$, then  moves within distance $\frac{11}{8}\lambda^{1/4}$, and then again moves at least distance $\frac{3}{2}\lambda^{1/4}$ away.
However, we can take $\lambda$ small enough to ensure that this oscillation cannot happen.
Let $z_1, z_2, z_3, z_4$ be the four corners of $\Omega$, i.e.~the points in the intersection of $A\cup B$ and $\tau_1 \cup \tau_2$.
Then we can take $s$ small enough so that the connected component $\eta_k$ of $(\tau_1\cup \tau_2) \cap D_s(z_k)$ containing $z_k$ 
is closely
approximated by a line segment for $k=1, \cdots, 4$.  Further, $(\tau_1 \cup \tau_2) \setminus \cup\eta_k$ and  $A \cup B$ are compact, disjoint sets.  Therefore, we can assume that $2\lambda^{1/4}$ is smaller than the minimum distance between these two sets.
However this makes the oscillation of $\tau_1$ described above impossible.

Note that since 
the distance from $z_0$ to $\partial V_1^\bullet$ is at least   $ \frac{3}{4} \lambda^{1/4}$,
   the distance from $\frac{1}{2}(z_0+y)$ to $\partial V_1^\bullet$ is at least  $ \frac{3}{4} \lambda^{1/4} - 2 \e \geq \frac{1}{2} \lambda^{1/4}$ (using that $\lambda \leq \frac{1}{8} \lambda^{1/4}$).
   Thus $D_R$ does not intersect $\partial V_1^\bullet$, and so the term $\kappa$ in Proposition \ref{modified7.1} is equal to zero.
Therefore Proposition \ref{modified7.1} implies that 
$$| h_d^{(1)}(z_0) - h_d^{(1)}(y) | \leq \frac{C\, \mathcal{E}_1^\bullet(h_d^{(1)})^{1/2}}{\log^{1/2}\left[ R/(r+\epsilon) \right]} 
  \leq  \frac{C}{\log^{1/2}\left(1/\lambda \right)}.$$

\end{proof}

This completes the proof of Proposition \ref{subdiagHarmConv}.
\end{proof}

\begin{proof}[Proof of Theorem \ref{HarmConvThm}.]
The theorem follows immediately from Propositions \ref{claim1B} and \ref{subdiagHarmConv}.
\end{proof}

\subsection{Proof of Theorem \ref{cor4}}\label{cor4Proof}

\begin{proof}[Proof of Theorem \ref{cor4}.]

Let $\Omega$ be an analytic quadrilateral, as defined in Definition \ref{def:smooth-quad},
with boundary arcs $A, \tau_1, B, \tau_2$ listed counterclockwise, 
and let $\Gamma$ be the curve family between $A$ and $B$ in $\Omega$.
Without loss of generality, assume $\diam(\Omega) = 1$.
Let $\epsilon_n>0$ with $\epsilon_n \to 0$.
Let $G_n$ be an orthodiagonal map with boundary arcs $S^n_1, T^n_1, S^n_2, T^n_2$, with edge-length at most $\epsilon_n$, and so that $(G_n;S^n_1, T^n_1, S^n_2, T^n_2)$  approximates $(\Omega;A, \tau_1, B, \tau_2)$ within distance $\e_n$ 
(in the sense of Definition \ref{def:approx-smooth-quad}). Equip $G_n$ with canonical edge-weights $\si_n$ as in Equation (\ref{eq:canonical-weights}).
Consider the family of all paths in $G^\bullet_n$ from $A_n^\bullet = V_n^\bullet \cap S^n_1$ to $B_n^\bullet= V_n^\bullet \cap S^n_2$, and let $\Gamma_n$ be the unique minimal subfamily with non-crossing paths.

Let $\psi$ be a conformal map of $\Omega$ onto the rectangle $R=[0,1] \times [0,m]$ taking the arcs $A$ and $B$ to the vertical sides of the rectangle,
and let $h = \text{Re } \psi$.
Let $h_n^\bullet: V_n^\bullet \to \mathbb{R}$ be discrete harmonic on ${\rm int}(V_n^\bullet)$, with the following boundary values:
$h_n^\bullet(z) = 0$ for $z \in A_n$ and $h_n^\bullet(z) = 1$ for $z \in B_n$.
Recall that 
$$\Mod_{2,\si_n}(\Gamma_n) = \mathcal{E}_n^\bullet(h_n^\bullet).$$
Similarly, the continuous modulus of $\Gamma$ equals the energy of $h$:
$$ \Mod_2(\Gamma) = \mathcal{E}(h) = \int_\Omega |\nabla h |^2\, dA.$$

Our goal is to show that $\Mod_{2,\si_n}(\Ga_n)$ tends to $m:=\Mod_2(\Ga)$ as $n$ tends to infinity.
In order to do this we introduce four other harmonic functions. First, in the continuous case, we let $\tilde{h}$ be the harmonic conjugate of $h$, i.e. 
\[
\tilde{h}:=\ima \psi.
\]
On the other hand, in the discrete case, we introduce three more discrete harmonic functions 
$$h_n^\circ: V_n^\circ \to \mathbb{R}, \;\; 
\tilde{h}_n^\bullet: V_n^\bullet \to \mathbb{R}, 
\; \text{ and } \; \tilde{h}_n^\circ: V_n^\circ \to \mathbb{R}$$
that satisfy the following boundary values:
\begin{align*}
h_n^\circ(z) &= 0 \; \text{ for } \;z \in  V_n^\circ \cap S_1^n \; \text{ and } \; h_n^\circ(z) = 1 \; \text{ for } z \in  V_n^\circ \cap S_2^n, \\
\tilde{h}_n^\bullet(z) &= 0 \; \text{ for } \; z \in  V_n^\bullet \cap T_1^n \; \text{ and } \; \tilde{h}_n^\bullet(z) = 1 \; \text{ for } \; z \in V_n^\bullet \cap T_2^n,  \\
\tilde{h}_n^\circ(z) &= 0 \; \text{ for } \; z \in  V_n^\circ \cap T_1^n \; \text{ and  } \; \tilde{h}_n^\circ(z) = 1 \; \text{ for } \; z \in  V_n^\circ \cap T_2^n.
\end{align*} 

As in the proof of Theorem \ref{HarmConvThm}, 
let $P_n$ be the set of all inner faces $Q$ of $G_n$ such that  dist$(Q, A \cup B) \geq \epsilon_n^{1/4}$, 
and let $G_{n,1} = G[P_n]$ be the subgraph of $G_n$ that is the union of the boundaries of the faces in $P_n$.
By Lemma \ref{reflection}, when $\epsilon_n$ is small enough, we can analytically extend $\psi$, so that $\psi$ is defined on all the vertices of $G_{n,1}$.  We define $G_{n,2}$ in an analogous way by switching the roles of $A,B$ and $\tau_1, \tau_2$.
For $k=1$ or $2$, the notation $\mathcal{E}_{n,k}^\bullet$ (respectively $\mathcal{E}_{n,k}^\circ$) refers to the energy with respect to the graph $G_{n,k}^\bullet$ (respectively $G_{n,k}^\circ$).

\begin{claim}\label{claimA}
There exists $r_n \searrow 1$ so that 
\begin{equation}\label{energybound1}
  \mathcal{E}_n^\bullet(h_n^\bullet) \leq r_n \, \mathcal{E}_{n, 1}^\bullet(h).
\end{equation}
\end{claim}

We assume Claim \ref{claimA} for now and return to its proof later.  Note that Claim \ref{claimA} implies that there exists $r_n \searrow 1$ so that
\begin{equation}\label{energybound2}
\mathcal{E}_n^\circ(h_n^\circ) \leq r_n \, \mathcal{E}_{n, 1}^\circ(h), \;\;
\mathcal{E}_n^\bullet(\tilde{h}_n^\bullet) \leq r_n \, \mathcal{E}_{n, 2}^\bullet(\tilde{h}/m), \; \text{and } \;
\mathcal{E}_n^\circ(\tilde{h}_n^\circ) \leq r_n \, \mathcal{E}_{n, 2}^\circ(\tilde{h}/m).
\end{equation}

Next we need some properties that follow from the Cauchy-Riemann equations for $\psi$.
\begin{claim}\label{lem:main-identity}
Fix a face $Q = [v_1, w_1, v_2, w_2]$ that is an inner face of $G_{n, 1}$ or $G_{n,2}$, and define 
$$\zeta_Q = \frac{(h(w_2)-h(w_1))+i(\tilde{h}(w_2)-\tilde{h}(w_1))}{|w_2 -w_1|} \;\; \text{and} \;\; \eta_Q = \frac{(h(v_2)-h(v_1))+i(\tilde{h}(v_2)-\tilde{h}(v_1))}{|v_2- v_1|}.$$
Then for $\epsilon_n$ small enough
\begin{enumerate}
\item[(i)] $|\zeta_Q -i\eta_Q| \leq C \ep_n^{1/2}$.
\item[(ii)] $ \ima(\zeta_Q\bar{\eta}_Q) \area(Q) \leq  \int_Q |\nabla h|^2 \,dA + C \ep_n^{1/4} \area{Q}$.
\end{enumerate}
\end{claim}

Again, we will postpone the proof of this claim until the end, and we will assume that we have taken $\epsilon_n$ small enough for its conclusion to hold.
By Claim \ref{lem:main-identity}(i), summing over all $Q$  that are inner faces of either $G_{n,1}$ or $G_{n,2}$ gives
$$ \sum_Q 2\area(Q) |\zeta_Q-i\eta_Q|^2 \leq C\ep_n, $$
since the area of $\Omega$ is finite.
Expanding $|\zeta_Q-i\eta_Q|^2=|\zeta_Q|^2+|\eta_Q|^2-2\ima(\zeta_Q\bar{\eta}_Q)$ in this sum, we obtain three terms, which we will relate to different energies. First,
\begin{align*}
\cE_{n, 1}^\circ(h)+\cE_{n, 2}^\circ(\tilde{h}) & \le \sum_{Q}2\area(Q)\left( |\rea \zeta_Q|^2+|\ima \zeta_Q|^2\right)\\
& = \sum_{Q} 2\area(Q) |\zeta_Q|^2, 
\end{align*}
where the inequality arises since the lefthand side only uses the inner faces of $G_{n,k}$ to compute $\cE_{n, k}^\circ$.
Similarly,
\[
\cE_{n, 1}^\bullet(h)+\cE_{n, 2}^\bullet(\tilde{h}) \leq \sum_{Q} 2\area(Q) |\eta_Q|^2.
\]
Lastly by Claim \ref{lem:main-identity}(ii), 
$$ \sum_{Q} 2\area(Q) \cdot 2 \ima(\zeta_Q\bar{\eta}_Q) \leq 4 \int_{\hat{G}_{n,1} \cup \hat{G}_{n,2}}  |\nabla h|^2 \, dA+ C \ep_n^{1/4}.$$
Therefore collecting everything gives
\[
\cE_{n, 1}^\bullet(h)+\cE_{n, 2}^\bullet(\tilde{h})+\cE_{n, 1}^\circ(h)+\cE_{n, 2}^\circ(\tilde{h}) \leq 4\int_{\hat{G}_{n,1} \cup \hat{G}_{n,2}}  |\nabla h|^2 \, dA + C \ep_n^{1/4},
\]
and from \eqref{energybound1} and  \eqref{energybound2} we obtain
\[
\cE_n^\bullet(h_n^\bullet)+m^2\cE_n^\bullet(\tilde{h}_n^\bullet)+\cE_n^\circ(h_n^\circ)+m^2\cE_n^\circ(\tilde{h}_n^\circ) \leq 4r_n \int_{\hat{G}_{n,1} \cup \hat{G}_{n,2}}  |\nabla h|^2 \, dA+ C \ep_n^{1/4}.
\]
Now, we use the fact that Fulkerson duality (Lemma \ref{lem:fulkerson-duality}(i)) gives
\[
\cE_n^\bullet(h_n^\bullet)\cE_n^\circ(\tilde{h}_n^\circ)=1\qquad\text{and}\qquad \cE_n^\circ(h_n^\circ)\cE_n^\bullet(\tilde{h}_n^\bullet)=1.
\]
So, dividing by $m= \int_\Omega |\nabla h |^2 \, dA$, we get
\[
\left(\frac{\cE_n^\bullet(h_n^\bullet)}{m}+ \frac{m}{\cE_n^\bullet(h_n^\bullet)}\right)+\left(\frac{\cE_n^\circ(h_n^\circ)}{m}+ \frac{m}{\cE_n^\circ(h_n^\circ)}\right) \leq 4r_n \frac{1}{m} \int_{\hat{G}_{n,1} \cup \hat{G}_{n,2}}  |\nabla h|^2 \, dA+ C \ep_n^{1/4}.
\]
Letting $\ep_n$ tend to zero, the right-hand side approaches 4, and hence
\[
\min\left\{ \limsup_{n\rightarrow \infty}\left(\frac{\cE_n^\bullet(h_n^\bullet)}{m}+ \frac{m}{\cE_n^\bullet(h_n^\bullet)}\right), \,\limsup_{n\rightarrow \infty}\left(\frac{\cE_n^\circ(h_n^\circ)}{m}+ \frac{m}{\cE_n^\circ(h_n^\circ)}\right)\right\}\le 2.
\]
But, the function $f(x)=\frac{x}{m}+\frac{m}{x}$, for $x>0$, has an absolute minimum at $x=m$ with value $2$. So, this shows that
\[
 \lim_{n\rightarrow \infty}\cE_n^\bullet(h_n^\bullet)= \lim_{n\rightarrow \infty}\cE_n^\circ(h_n^\circ)=m.
\]

To finish, we much establish the two claims.

\begin{proof}[Proof of Claim \ref{claimA}]

It suffices to prove the statement for $\epsilon_n$ small enough.
Let $\partial V_{n,1}^\bullet$ consist of the vertices in $V_{n,1}^\bullet$ that are in the topological boundary of the outer face of $G_{n, 1}$ but not in  $T_1^n\cup T_2^n$.  
If  $z \in \partial V^\bullet_{n,1}$, then $\text{dist}(z, A \cup B) \leq \frac{5}{4} \epsilon_n^{1/4}$
by Claim \ref{claim2}.  
Set $A_n' = \{ z \in \partial V_{n,1}^\bullet \, : \, \text{dist}(z, A) \leq \frac{5}{4} \epsilon_n^{1/4} \}$ 
and define $B_n'$ similarly (by replacing $A$ with $B$ in the above definition).
We assume that $\epsilon_n$ is small enough so that $A_n'$ and $B_n'$ are disjoint. 
Let $z \in A_n'$.
Then we see from Proposition \ref{claim1B}  that there exists $w \in A_n $ so that
$|h_n^\bullet(z) - h_n^\bullet(w)| \leq C \log^{-1/2} (1/\epsilon_n)$.  
Using this along with Theorem \ref{HarmConvThm} shows that,  for $z\in A_n'$,
$$ |h(z) | = |h(z) - h_n^\bullet(w) | \leq |h(z) - h_n^\bullet(z)| + |h_n^\bullet(z) - h_n^\bullet(w)| \leq \frac{C}{ \log^{1/2} (1/\epsilon_n)}.$$
Similarly, for $z \in B_n'$, 
$$| h(z) -1| \leq \frac{C}{ \log^{1/2} (1/\epsilon_n)}.$$
Therefore, 
$$\text{gap}_{A_n', B_n'}\left(h \right) \geq 1-\frac{C}{\log^{1/2} (1/\epsilon_n)},$$
where $\text{gap}$ is defined in (\ref{eq:gap-def})
and we assume that $\epsilon_n$ is small enough so that the righthand side of this equation is positive.

Let $g$ be the solution to the discrete Dirichlet problem on $G_{n,1}^\bullet$ that is 0 on $A_n'$ and 1 on $B_n'$.
Then Dirichlet's Principle gives that 
$$  \mathcal{E}_{n, 1}^\bullet (g) = \inf_{f} \frac{ \mathcal{E}_{n, 1}^\bullet (f) }{\text{gap}_{A_n', B_n'}(f)^2}
\leq r_n \, \mathcal{E}_{n, 1}^\bullet(h), $$
where $r_n = \left( 1-\frac{C}{\log^{1/2} (1/\epsilon_n)} \right)^{-2}$.
Since every path in $\Gamma_n$ contains a path between $A_n'$ and $B_n'$ in $G_{n,1}^\bullet$,
Statement 5.1 in \cite{APCDSG} implies that
$$\Mod_{2,\si_n}(\Gamma_n) \leq  \Mod_{2,\si_n}\left(\text{paths between } A_n'  \text{ and } B_n' \text{ in } G_{n,1}^\bullet \right) = \mathcal{E}_{n, 1}^\bullet (g). $$
Therefore
$$ \mathcal{E}_n^\bullet(h_n^\bullet) = \Mod_{2,\si_n}(\Gamma_n)  \leq \mathcal{E}_{n, 1}^\bullet (g)  \leq  r_n \, \mathcal{E}_{n, 1}^\bullet(h).$$

\end{proof}

\begin{proof}[Proof of Claim \ref{lem:main-identity}]
Let $q$ be the intersection point of $e_Q^\bullet$ and $e^\circ_Q$, and let ${\bf v}$ and ${\bf w}$ are the unit vectors in the directions of $\vv{v_1v_2}$ and $\vv{w_1w_2}$, respectively.
We will assume that $\ep_n$ is small enough for Claim \ref{claim3} to hold.
Then since $\langle \nabla  h(q), {\bf w} \rangle =-\langle \nabla \tilde h(q), {\bf v} \rangle$,
\begin{align*}
 &\left| \frac{h(w_2) - h(w_1)}{|w_2-w_1|} +  \frac{\tilde{h}(v_2)-\tilde{h}(v_1)}{|v_2-v_1|} \right| \\ 
&\hspace{1.1in} \leq  \left| \frac{h(w_2) - h(w_1)}{|w_2-w_1|} -\langle \nabla  h(q), {\bf w} \rangle \right| + \left|  \frac{\tilde{h}(v_2)-\tilde{h}(v_1)}{|v_2-v_1|}- \langle \nabla \tilde h(q), {\bf v} \rangle \right| \\
&\hspace{1.1in} \leq 2 || Hh ||_{\tilde{Q}} \, \ep_n \\
&\hspace{1.1in} \leq C \ep_n^{1/2}
\end{align*}
 where Lemma 5.3 in \cite{gurel-jerison-nachmias:advm2020} gives the penultimate inequality and Claim \ref{claim3} gives the last inequality.
Similarly, we can use that $\langle \nabla  h(q), {\bf v} \rangle =\langle \nabla \tilde h(q), {\bf w} \rangle$ to show
$\left| \frac{\tilde h(w_2) - \tilde h(w_1)}{|w_2-w_1|} -  \frac{h(v_2)-h(v_1)}{|v_2-v_1|} \right|  \leq C \ep_n^{1/2}$,
and this establishes (i).

For (ii), note that 
$$ \ima(\zeta_Q\bar{\eta}_Q) = \frac{h(v_2)-h(v_1)}{|v_2-v_1|} \frac{\tilde{h}(w_2)-\tilde{h}(w_1)}{|w_2-w_1|}-\frac{h(w_2)-h(w_1)}{|w_2-w_1|}\frac{\tilde{h}(v_2)-\tilde{h}(v_1)}{|v_2-v_1|}.$$
Using the same tools as above (Cauchy-Riemann equations,
Lemma 5.3 in \cite{gurel-jerison-nachmias:advm2020}, and Claim \ref{claim3}), one can show that 
$$\left|\ima(\zeta_Q\bar{\eta}_Q) - |\nabla h(q)|^2\right| \leq C \ep_n^{1/4}.$$
Further, Lemma 5.3 in \cite{gurel-jerison-nachmias:advm2020} and Claim \ref{claim3} also give that 
$|\nabla h(q)|^2 \leq |\nabla h(z) |^2 + C\ep_n^{1/4}$ for any $z \in Q$.
Hence
$$ \ima(\zeta_Q\bar{\eta}_Q) \area(Q) \leq  |\nabla h(q)|^2 \area(Q) + C \ep_n^{1/4}\area(Q) \leq \int_Q |\nabla h|^2 \, dA + C \ep_n^{1/4}\area(Q).$$
\end{proof}

By proving Claims \ref{claimA} and \ref{lem:main-identity}, we have established Theorem  \ref{cor4}.
\end{proof}

\bibliographystyle{acm}
\bibliography{pmodulus}

\begin{thebibliography}{10}

\bibitem{ahlfors1973}
{\sc Ahlfors, L.~V.}
\newblock {\em Conformal invariants: topics in geometric function theory}.
\newblock McGraw-Hill Book Co., New York-D\"usseldorf-Johannesburg, 1973.
\newblock McGraw-Hill Series in Higher Mathematics.

\bibitem{abppcw:ecgd2015}
{\sc Albin, N., Brunner, M., Perez, R., Poggi-Corradini, P., and Wiens, N.}
\newblock Modulus on graphs as a generalization of standard graph theoretic
  quantities.
\newblock {\em Conform. Geom. Dyn. 19\/} (2015), 298--317.
\newblock http://dx.doi.org/10.1090/ecgd/287.

\bibitem{acfpc:ampa2019}
{\sc Albin, N., Clemens, J., Fernando, N., and Poggi-Corradini, P.}
\newblock Blocking duality for {$p$}-modulus on networks and applications.
\newblock {\em Ann. Mat. Pura Appl. (4) 198}, 3 (2019), 973--999.
\newblock https://doi-org.er.lib.k-state.edu/10.1007/s10231-018-0806-0.

\bibitem{apc:jan2016}
{\sc Albin, N., and Poggi-Corradini, P.}
\newblock Minimal subfamilies and the probabilistic interpretation for modulus
  on graphs.
\newblock {\em The Journal of Analysis\/} (2016), 1--26.
\newblock http://dx.doi.org/10.1007/s41478-016-0002-9.

\bibitem{APCDSG}
{\sc Albin, N., Sahneh, F.~D., Goering, M., and Poggi-Corradini, P.}
\newblock Modulus of families of walks on graphs.
\newblock In {\em Complex analysis and dynamical systems {VII}}, vol.~699 of
  {\em Contemp. Math.} Amer. Math. Soc., Providence, RI, 2017, pp.~35--55.

\bibitem{alrayes-phdthesis}
{\sc Alrayes, N.}
\newblock {\em Approximation of $p$-modulus in the plane with discrete grids}.
\newblock PhD thesis, Kansas State University, 2018.

\bibitem{badger:2013finn}
{\sc Badger, M.}
\newblock Beurling's criterion and extremal metrics for {F}uglede modulus.
\newblock {\em Ann. Acad. Sci. Fenn. Math. 38}, 2 (2013), 677--689.

\bibitem{brooks-smith-stone-tutte:duke1940}
{\sc Brooks, R.~L., Smith, C. A.~B., Stone, A.~H., and Tutte, W.~T.}
\newblock The dissection of rectangles into squares.
\newblock {\em Duke Math. J. 7\/} (1940), 312--340.

\bibitem{ciarlet1978}
{\sc Ciarlet, P.~G.}
\newblock {\em The finite element method for elliptic problems}.
\newblock Studies in Mathematics and its Applications, Vol. 4. North-Holland
  Publishing Co., Amsterdam-New York-Oxford, 1978.

\bibitem{ebpc}
{\sc Eriksson-Bique, S., and Poggi-Corradini, P.}
\newblock On the sharp lower bound for duality of modulus.
\newblock {\em Unpublished\/} (2021).
\newblock https://arxiv.org/pdf/2102.03035.pdf.

\bibitem{garnett-marshall2008}
{\sc Garnett, J.~B., and Marshall, D.~E.}
\newblock {\em Harmonic measure}, vol.~2 of {\em New Mathematical Monographs}.
\newblock Cambridge University Press, Cambridge, 2008.
\newblock Reprint of the 2005 original.

\bibitem{gurel-jerison-nachmias:advm2020}
{\sc Gurel-Gurevich, O., Jerison, D.~C., and Nachmias, A.}
\newblock The {D}irichlet problem for orthodiagonal maps.
\newblock {\em Adv. Math. 374\/} (2020), 107379, 53.

\bibitem{hara-mizumoto:jmsj1990}
{\sc Hara, H., and Mizumoto, H.}
\newblock Determination of the modulus of quadrilaterals by finite element
  methods.
\newblock {\em J. Math. Soc. Japan 42}, 2 (1990), 295--326.

\bibitem{hinkkanen1988}
{\sc Hinkkanen, A.}
\newblock Modulus of continuity of harmonic functions.
\newblock {\em J. Analyse Math 51\/} (1988), 1--29.

\bibitem{lovasz2019}
{\sc Lov\'{a}sz, L.}
\newblock {\em Graphs and geometry}, vol.~65 of {\em American Mathematical
  Society Colloquium Publications}.
\newblock American Mathematical Society, Providence, RI, 2019.

\bibitem{lyons-peres:2016}
{\sc Lyons, R., and Peres, Y.}
\newblock {\em Probability on trees and networks}, vol.~42 of {\em Cambridge
  Series in Statistical and Probabilistic Mathematics}.
\newblock Cambridge University Press, New York, 2016.

\bibitem{LP:book}
{\sc Lyons, R., and Peres, Y.}
\newblock {\em Probability on Trees and Networks}.
\newblock Cambridge University Press, New York, 2016.

\bibitem{pommerenke1992}
{\sc Pommerenke, C.}
\newblock {\em Boundary behavior of conformal maps}.
\newblock Springer-Verlag Berlin Heidelberg, 1992.

\bibitem{prasolov-skopenkov:jcombthsera2011}
{\sc Prasolov, M., and Skopenkov, M.}
\newblock Tiling by rectangles and alternating current.
\newblock {\em J. Combin. Theory Ser. A 118}, 3 (2011), 920--937.

\bibitem{rogers_williams_2000}
{\sc Rogers, L. C.~G., and Williams, D.}
\newblock {\em Diffusions, Markov Processes, and Martingales}, 2~ed., vol.~1 of
  {\em Cambridge Mathematical Library}.
\newblock Cambridge University Press, 2000.

\bibitem{shewchuk96b}
{\sc Shewchuk, J.~R.}
\newblock Triangle: {E}ngineering a {2D} {Q}uality {M}esh {G}enerator and
  {D}elaunay {T}riangulator.
\newblock In {\em Applied Computational Geometry: Towards Geometric
  Engineering}, M.~C. Lin and D.~Manocha, Eds., vol.~1148 of {\em Lecture Notes
  in Computer Science}. Springer-Verlag, May 1996, pp.~203--222.
\newblock From the First ACM Workshop on Applied Computational Geometry.

\bibitem{skopenkov:adv-math2013}
{\sc Skopenkov, M.}
\newblock The boundary value problem for discrete analytic functions.
\newblock {\em Adv. Math. 240\/} (2013), 61--87.

\bibitem{skopenkov:advm2013}
{\sc Skopenkov, M.}
\newblock The boundary value problem for discrete analytic functions.
\newblock {\em Adv. Math. 240\/} (2013), 61--87.

\bibitem{Werness2015}
{\sc Werness, B.}
\newblock Discrete analytic functions on non-uniform lattices without global
  geometric control.
\newblock {\em arXiv: Complex Variables\/} (2015).

\end{thebibliography}
\def\cprime{$'$}

\end{document}